\documentclass[a4paper,reqno]{amsart}
\usepackage{amsmath,amsthm,enumitem}
\usepackage{amssymb,stmaryrd,bbm}
\usepackage{hyperref,aliascnt}
\usepackage{tikz}\usetikzlibrary{cd,math}
\usepackage{microtype}
\usepackage{fullpage}
\allowdisplaybreaks

\numberwithin{equation}{section}
\theoremstyle{plain}
\newaliascnt{thm}{equation}
\newaliascnt{lem}{equation}
\newaliascnt{cor}{equation}
\newaliascnt{prop}{equation}
\newaliascnt{df}{equation}
\newaliascnt{rem}{equation}
\newaliascnt{eg}{equation}
\newaliascnt{clm}{equation}
\newtheorem{thm}[thm]{Theorem}
\newtheorem{lem}[lem]{Lemma}
\newtheorem{cor}[cor]{Corollary}
\newtheorem{prop}[prop]{Proposition}
\newtheorem{clm}[clm]{Claim}
\newtheorem*{clm*}{Claim}
\theoremstyle{definition}
\newtheorem{df}[df]{Definition}
\theoremstyle{remark}
\newtheorem{rem}[rem]{Remark}
\newtheorem{eg}[eg]{Example}
\aliascntresetthe{thm}
\aliascntresetthe{lem}
\aliascntresetthe{cor}
\aliascntresetthe{prop}
\aliascntresetthe{df}
\aliascntresetthe{rem}
\aliascntresetthe{eg}
\aliascntresetthe{clm}

\newcommand{\bC}{\mathbb{C}}

\newcommand{\bF}{\mathbb{F}}

\newcommand{\bM}{\mathbb{M}}

\newcommand{\bR}{\mathbb{R}}

\newcommand{\bZ}{\mathbb{Z}}

\newcommand{\cB}{\mathcal{B}}
\newcommand{\cC}{\mathcal{C}}

\newcommand{\cF}{\mathcal{F}}
\newcommand{\cG}{\mathcal{G}}
\newcommand{\cH}{\mathcal{H}}

\newcommand{\cK}{\mathcal{K}}
\newcommand{\cL}{\mathcal{L}}
\newcommand{\cM}{\mathcal{M}}
\newcommand{\cN}{\mathcal{N}}
\newcommand{\cO}{\mathcal{O}}

\newcommand{\cR}{\mathcal{R}}

\newcommand{\cU}{\mathcal{U}}
\newcommand{\cV}{\mathcal{V}}

\newcommand{\cZ}{\mathcal{Z}}

\newcommand{\fu}{\mathfrak{u}}

\newcommand{\qa}{a_q}
\newcommand{\qc}{c_q}
\providecommand{\xto}[1]{\xrightarrow{#1}}
\providecommand{\motimes}{\mathbin{\underset{\mathrm{max}}{\otimes}}}

\newcommand{\wt}{\widetilde}
\newcommand{\wc}[1]{{#1}^{\vee}}
\DeclareMathOperator{\id}{id}
\DeclareMathOperator{\ev}{ev}
\DeclareMathOperator{\pr}{pr}
\DeclareMathOperator{\tr}{tr}

\DeclareMathOperator{\Ker}{Ker}
\DeclareMathOperator{\Hom}{Hom}

\DeclareMathOperator{\Aut}{Aut}
\DeclareMathOperator{\Rep}{Rep_f}
\DeclareMathOperator{\Irr}{Irr_f}
\DeclareMathOperator{\Hilb}{Hilb_f}
\DeclareMathOperator{\K}{K} 
\DeclareMathOperator{\Cor}{Corr} 
\DeclareMathOperator{\SU}{SU}
\DeclareMathOperator{\SO}{SO}
\DeclareMathOperator{\rU}{U}
\DeclareMathOperator{\rO}{O}
\DeclareMathOperator{\Qut}{Qut}
\DeclareMathOperator{\Ad}{Ad} 

\DeclareMathOperator{\op}{op} 
\DeclareMathOperator{\cspan}{\overline{span}}

\newcommand{\barrtimes}{\mathbin{\bar{\rtimes}}}
\newcommand{\barltimes}{\mathbin{\bar{\ltimes}}}
\newcommand{\barotimes}{\mathbin{\bar{\otimes}}}

\newcommand{\bra}{\langle}
\newcommand{\ket}{\rangle}

\begin{document}
\title{Classifiable C*-algebras have quantum symmetries}
\author{Kan Kitamura}
\address{Department of Mathematics, College of Science, Rikkyo University, 3-34-1, Nishi-Ikebukuro, Toshima-ku, Tokyo, 171-8501, Japan}
\address{Center for Interdisciplinary Theoretical and Mathematical Sciences, RIKEN, 2-1 Hirosawa, Wako, Saitama 351-0198 Japan}
\email{kan.kitamura@rikkyo.ac.jp}
\begin{abstract}
	We show that every classifiable simple C*-algebra, in particular the Jiang--Su algebra, 
	admits outer actions of several discrete quantum groups, such as finite quantum groups and duals of certain free compact quantum groups. 
	As a consequence of the construction, we provide new examples of hyperfinite subfactors with the Temperley--Lieb--Jones standard invariant and show that $\operatorname{SU}_q(2)$ can act minimally on some Araki--Woods factors for infinitely many $q$. 
\end{abstract}
\maketitle
\setcounter{tocdepth}{1}
\tableofcontents

\section{Introduction}
In subfactor theory initiated by Jones \cite{Jones1983index}, it has been observed that operator algebras often admit symmetries that cannot be captured by usual groups, which are sometimes called quantum symmetries in the operator algebraic context. 
A celebrated result of Popa \cite{Popa1994classificationamenable} classifies amenable inclusions of hyperfinite $\mathrm{II}_1$ factors of finite index by their standard invariant. 
A way to interpret the standard invariant is through the notion of actions of unitary tensor categories. 

In the C*-algebraic setting, the main target of the classification is a separable simple nuclear C*-algebra in the UCT class that absorbs the Jiang--Su algebra $\cZ$, which is referred to as a classifiable C*-algebra in this paper. 
Over the past few decades, there has been substantial progress in the classification of C*-algebras in this class (see e.g., \cite{Carrion-Gabe-Schafhauser-Tikuisis-White2023+classifying} and references therein). 
Also, recently, Gabe--Szab{\'o} \cite{Gabe-Szabo2024dynamical} established the dynamical version of the Kirchberg--Phillips theorem, providing a breakthrough in the classification of group actions on Kirchberg algebras. 
Along with these developments, quantum symmetries on classifiable C*-algebras and their classification have been an interesting topic. 
For example, actions of fusion categories on AF-algebras of inductive limit type were classified by Chen--Hern{\'a}ndez Palomares--Jones \cite{Chen-HernandezPalomares-Jones2024K-theoretic}. 

Beyond AF-algebras, there are a few known examples of actions of non-trivial unitary tensor categories on classifiable C*-algebras, such as those on Kirchberg algebras \cite{Izumi1993subalgebrasI,Izumi1998subalgebrasII,Kitamura2026actions} and on non-commutative tori \cite{Evans-Jones2025quantum}. 
However, it is in general hard to control the underlying C*-algebra on which a given unitary tensor category $\cC$ acts, as observed by \cite{Evington-GironPacheco2023anomalous,Izumi2023+G-kernels,Evans-Jones2025quantum} in the form of $\K$-theoretic obstructions to the existence of $\cC$-actions. 
This situation is in sharp contrast to the group case, as any classifiable C*-algebra $A$ admits an outer action of a given countable group $\Gamma$. 
Indeed, the Bernoulli shift of $\Gamma$ on the Jiang--Su algebra $\cZ^{\otimes\Gamma}\cong\cZ$ is already outer, but this construction is available only in the group case. 
In this work, we construct an outer $\cC$-action on a classifiable C*-algebra that need not be purely infinite, when $\cC$ is the representation category $\Rep G$ of a certain compact quantum group $G$. 

\begin{thm}[{\autoref{cor_IU_JiangSu}}]\label{main_JiangSu}
	Let $G$ be one of the following compact quantum groups: 
	\begin{itemize}[leftmargin=1.5em]
		\item
		a finite quantum group, 
		\item
		a compact Lie group, 
		\item
		the quantum permutation group $S_n^+$ for $n\geq 4$, 
		\item
		the free unitary quantum group $\rU^+(n)$ for $n\geq 2$, 
		\item
		the free orthogonal quantum groups $\rO^+(n)$ for $n\geq 2$ and $\rO^+\big(\begin{smallmatrix} 0 & -1_{\bM_n} \\ 1_{\bM_n} & 0 \end{smallmatrix}\big)$ for $n\geq 1$. 
	\end{itemize}
	Then, when $A$ is a classifiable C*-algebra, in particular the Jiang--Su algebra $\cZ$, there is an outer action of $\Rep G$ on $A$. 
\end{thm}

This is shown by constructing actions on $A$ of the Pontryagin dual of $G$ (see \autoref{thm_IU_JiangSu}). To the best of the author's knowledge, this result gives the first example of an action on $\cZ$ of a unitary tensor category that cannot be captured by a group (or its Pontryagin dual). 
Note that $\Rep\SU_q(2)$ and $\Rep\SO_q(3)$ for certain $q\in(-1,1)\setminus\{0\}$ are unitarily monoidally equivalent to $\Rep G$ as listed in \autoref{main_JiangSu} by \cite{Bichon-DeRijdt-Vaes2006ergodic,DeRijdt-VanderVennet2010actions,Soltan2010quantumSO(3)}.

Outer actions of $\Rep\SU_q(2)$ and $\Rep\SO_q(3)$ provide a fundamental class of subfactors. 
For $q\in(0,1]$, when we write $\pi_{\frac{1}{2}}\in\Rep\SU_q(2)$ for the fundamental unitary representation of $\SU_q(2)$ on the two-dimensional Hilbert space, the object $\pi_{\frac{1}{2}}^{\otimes2}$ belongs to the subcategory $\Rep\SO_q(3) \subset \Rep\SU_q(2)$ 
and is equipped with the structure of a Q-system induced by regarding $\pi_{\frac{1}{2}}^{\otimes2} \cong \pi_{\frac{1}{2}} \otimes \overline{ \pi_{\frac{1}{2}} }$ as the internal endomorphism of the object $\pi_{\frac{1}{2}} \in \Rep\SU_q(2)$. 
The associated standard invariant is called the Temperley--Lieb--Jones standard invariant or sometimes the trivial standard invariant. 

Since the initiation of the subfactor theory by Jones \cite{Jones1983index}, it has remained an open problem to determine the set of possible indices of irreducible hyperfinite subfactors (see \cite[Problem 1]{Jones1983index}, \cite[6.2.1]{Popa2023W-representations}). 
As a naive candidate for the standard invariant of such a subfactor, it is natural to ask what the set of possible indices of hyperfinite subfactors with the trivial standard invariant is (see \cite[6.3.2]{Popa2023W-representations} and also \cite[Section 1]{Bisch-Caceres2025graph}). 
If we allow the factor to be non-hyperfinite, the free group factor $L(\bF_\infty)$ admits a subfactor with the trivial standard invariant by Shlyakhtenko--Ueda \cite{Shlyakhtenko-Ueda2002irreducible}.

Dually, as a potential source of a hyperfinite subfactor with the trivial standard invariant, the existence of a minimal action of the compact quantum group $\SU_q(2)$ on an injective factor is also an interesting question (see e.g.\ \cite[Section 1]{Vaes2005strictly}). 
There are several constructions of minimal actions of quantum groups. 
By a result of Ueda \cite{Ueda1999minimal}, for any $q\in(0,1)$, the compact quantum group $\SU_q(n)$ can act minimally on some factor with separable predual. 
Its generalization to locally compact quantum groups $G$ with separable $L^2(G)$ was shown by Vaes \cite{Vaes2005strictly}.

In recent work by Bisch--C{\'a}ceres \cite{Bisch-Caceres2025graph,Bisch-Caceres2025+new}, an irreducible subfactor of $\cR$ with the Temperley--Lieb--Jones standard invariant is constructed for every index $d\in (4,3+\sqrt{5}]$ that admits some irreducible subfactor of finite depth. 
To ensure the triviality of the standard invariant of the subfactors therein, the classification of the standard invariants beyond index $4$ played an important role. At present, this classification is completed up to the indices in $(4,5.25]$ by 
\cite{Jones-Morrison-Snyder2014classification,Izumi-Morrison-Penneys-Peters-Snyder2015subfactors,Afzaly-Morrison-Penneys2023classification} (see also references therein). 
Other known examples of irreducible hyperfinite subfactors include Ocneanu's example at index $\|E_{10}\|^2 = 4.02641...$, Schou's family of examples \cite{Schou-thesis} with indices in $(4,5)$, and Bisch's example \cite{Bisch1994example} at index $\frac{9}{2}$. 
These examples have the trivial standard invariant by the classification result. 
When the index is in $\bZ_{\geq 6}$, which is not covered by the classification of the standard invariants at present, our construction of \autoref{main_JiangSu} yields new examples of hyperfinite subfactors with the trivial standard invariant. Since our method has some room for controlling the unitary tensor category, this construction also provides examples of minimal $\SU_q(2)$-actions on some injective factors such that $q\in (-1,1)\setminus\{0\}$. 

\begin{thm}[{\autoref{cor_RFD_UHF}, \autoref{thm_minimal_action}}]\label{main_subfactor}
	We have the following. 
	\begin{enumerate}[leftmargin=*,label=(\arabic*)]
		\item\label{item_main_subfactor_1}
		For any $d\in\bZ_{\geq 4}$, there is an irreducible subfactor of the hyperfinite $\mathrm{II}_1$ factor with index $d$ and trivial standard invariant. 
		\item\label{item_main_subfactor_2}
		The Araki--Woods factor of type $\mathrm{III}_{1}$ admits a minimal $\SU_q(2)$-action for any $q\in (-1,1)\setminus\{0\}$ with $q+q^{-1}\in \{ -n, 2n \mid n\in \bZ_{\geq 1} \}$. 
	\end{enumerate}
\end{thm}

Note that when $d=4$ and $5$, \ref{item_main_subfactor_1} is already covered by Jones \cite{Jones1983index} and Bisch--C{\'a}ceres \cite{Bisch-Caceres2025+new} respectively. 
Regarding \ref{item_main_subfactor_2}, it follows from a result of Vaes \cite{Vaes2005strictly} that there is no minimal $\SU_q(2)$-action on a factor of type $\mathrm{I}$ or $\mathrm{II}$ when $q\in(-1,1)\setminus\{0\}$, 
and thus we can hope for a minimal $\SU_q(2)$-action only on a factor of type $\mathrm{III}$. 
Also, it might be noteworthy that, unlike the aforementioned examples of hyperfinite subfactors with trivial standard invariant, the proofs of \ref{item_main_subfactor_1} when $\sqrt{d}$ is an integer and \ref{item_main_subfactor_2} do not depend on the classification of the standard invariant (cf.\ \autoref{eg_tensor_faith_Sn+} and \autoref{eg_tensor_faith_On+}).

This paper is organized as follows. 
After recalling terminologies on quantum groups in \autoref{sec_prelim}, we investigate finite dimensional unitary representations that satisfy an analog of faithfulness in \autoref{sec_separateDQG}. This property is called inner faithfulness by Banica--Bichon \cite{Banica-Bichon2010Hopf}, but we provide another characterization of it that is suitable for our purposes (\autoref{lem_tensor_faith}). 
As a consequence of the topological generation results for quantum permutation groups and free unitary quantum groups by Brannan--Chirvasitu--Freslon \cite{Brannan-Chirvasitu-Freslon2020topological} and Chirvasitu \cite{Chirvasitu2020topological} respectively, we observe that the Pontryagin duals of the compact quantum groups listed in \autoref{main_JiangSu} fit into this framework.

In \autoref{sec_IU_JiangSu}, we show \autoref{main_JiangSu}. The strategy is to equip the prime dimension drop C*-algebras appearing in the construction of the Jiang--Su algebra \cite{Jiang-Su1999simple} with actions of the discrete quantum group. 
Historically, such an inductive limit construction was investigated through unitary representations of a compact quantum group, which led to the theory of Poisson boundaries of quantum groups \cite{Izumi2002non-commutative} and of tensor categories \cite{Neshveyev-Yamashita2017Poisson}. 
Our approach uses finite dimensional unitary representations of a discrete, rather than compact, quantum group, which imposes the assumption of inner faithfulness. 
The outerness of the resulting $\Gamma$-action is shown separately on the group-like part of $\Gamma$ and the remaining part of $\Gamma$. 
The key technique to deal with the latter case is the Perron--Frobenius theorem for irreducible unital completely positive maps on finite dimensional C*-algebras due to Evans--H{\o}egh-Krohn \cite{Evans-HoeghKrohn1978spectral}. 

Finally, in \autoref{sec_applications}, we show \autoref{main_subfactor} using the construction of \autoref{main_JiangSu}. In \autoref{thm_RFD_UHF}, we also present a variant of \autoref{main_JiangSu} in the case of UHF algebras, which allows for a slightly broader class of discrete quantum groups. 
In \autoref{sec_prelim_action}, we briefly recall the notions of actions of quantum groups and unitary tensor categories that appear in \autoref{sec_IU_JiangSu} and \autoref{sec_applications}, and make some remarks on compatibility.

\subsection*{Acknowledgments}
This work was supported by JSPS KAKENHI Grant Number 25K17272. 
The author would like to thank Yasuyuki Kawahigashi, Yasuhiko Sato, Yuhei Suzuki, and Reiji Tomatsu for stimulating discussions and valuable comments. 
During the preparation of this manuscript, generative AI tools (Gemini Flash, Thinking, Pro and the free plan of ChatGPT) were used solely for language editing and grammatical refinement. The author developed all mathematical concepts independently and takes full responsibility for the contents of the paper. 

\section{Preliminaries}\label{sec_prelim}

We assume the coefficient field of a linear space is $\bC$ unless specified otherwise. 
We fix several conventions. 

\begin{itemize}[leftmargin=1.5em]
	\item
	For $r\in\bR$, we write $\bZ_{\geq r}:=\bZ\cap[r,\infty)$. 
	\item
	For $n\in\bZ_{\geq 1}$, we write $\bM_n$ for the C*-algebra consisting of complex $n\times n$-matrices. Also, $\bM_{n^\infty}$ denotes the UHF algebra of type $n^\infty$ when $n\geq 2$. 
	\item
	We write $\delta_{i,j}$ for the Kronecker delta. 
	\item
	For a non-zero finite dimensional Hilbert space $\cH$, we write $\tr_{\cH} := \tr_{\cB(\cH)}$. 
	For a state $f$ on a C*-algebra $A$ and a normal state $g$ on a von Neumann algebra $M$, we write $L^2(A,f)$ and $L^2(M,g)$ for the Hilbert space of the GNS construction of $f$. When $f$ or $g$ is faithful, we often regard $A\subset\cB(L^2(A,f))$ or $M\subset\cB(L^2(M,g))$ with abuse of notation. 
	\item
	Let $A$ be a C*-algebra. We write $\cM(A)$ for the multiplier algebra of $A$. We write $\cU(A)$ for the set of unitary elements in $A$ (which is empty unless $A$ is unital). When the meaning is clear, we sometimes omit parentheses. 
	\item
	For a $*$-homomorphism $f\colon A\to \cM(B)$ that is non-degenerate, we still write $f\colon\cM(A)\to \cM(B)$ for the $*$-homomorphism extending $f$ on $A$ that is strictly continuous on the unit ball. 
	\item
	For C*-algebras $A$ and $B$, we write $A\otimes B$ for the spatial tensor product. 
	The Hilbert space defined as the tensor product of Hilbert spaces $\cH$ and $\cG$ is denoted by $\cH\otimes\cG$. 
	Otherwise, $\otimes$ denotes an algebraic tensor product. 
	When $A$ and $B$ are von Neumann algebras, $A\barotimes B$ denotes their von Neumann algebraic tensor product. 
	For topological vector spaces $X$ and $Y$, we sometimes write $X\odot Y$ for their tensor product in the algebraic sense to emphasize that we do not consider the topology. 
	\item
	For a family of C*-algebras $(A_i)_{i\in I}$, we write $\bigoplus_{i\in I}A_i$ and $\prod_{i\in I}A_i$ for the $c_0$- and $\ell^\infty$- direct sums respectively. For a family of Hilbert spaces $(\cH_i)_{i\in I}$, we write $\bigoplus_{i\in I}\cH_i$ for the $\ell^2$-direct sum. 
	\item
	For a Banach space $X$ and its subset $S\subset X$, the norm-closed linear span of $S$ is denoted by $\cspan S \subset X$. 
	\item
	By an abuse of notation, the same symbol $\cZ$ is used to denote both the Jiang--Su algebra and the center. The latter case appears only in the form of $\cZ(A)$, where $A$ is an algebra or a group. Otherwise $\cZ$ denotes the Jiang--Su algebra. 
\end{itemize}

In this section, we set up terminologies and conventions on compact and discrete quantum groups. 
\subsection{Compact and discrete quantum groups}\label{ssec_prelim_CQGDQG}
The main interest of this paper is actions of discrete quantum groups and compact quantum groups. To make use of the Pontryagin duality, we begin by recalling the general framework of the locally compact quantum group by Kustermans--Vaes \cite{Kustermans-Vaes2000locally,Kustermans-Vaes2003locally}. 
For a weight $\omega$ on a von Neumann algebra $M$, we write $\cN_\omega := \{ x\in M\mid \omega(x^*x) \in [0,\infty) \}$. 
\begin{df}
	A \emph{locally compact quantum group} is a pair $G=(L^\infty(G),\Delta_G)$ of a von Neumann algebra $L^\infty(G)$ and a unital normal $*$-homomorphism $\Delta_G\colon L^\infty(G)\to L^\infty(G)\barotimes L^\infty(G)$ called the \emph{comultiplication} such that $(\Delta_G\otimes\id)\Delta_G = (\id\otimes\Delta_G)\Delta_G$ and there are faithful normal semifinite weights $h_r, h_l \colon L^\infty(G)_+ \to [0,\infty]$ satisfying $h_r(\id\otimes\omega)\Delta_G(a)=h_r(a)$ and $h_l(\omega\otimes\id)\Delta_G(b)=h_l(b)$ for all $a,b\in L^\infty(G)_+$ and any normal state $\omega\in L^\infty(G)_*$. Such $h_l$ and $h_r$ are unique up to scalar multiplication, and called the \emph{left Haar weight} and the \emph{right Haar weight} of $G$, respectively. 
\end{df}

Let $G$ and $H$ be locally compact quantum groups. If there is a $*$-isomorphism $f\colon L^\infty(G) \to L^\infty(H)$ such that $\Delta_{H}f = (f\otimes f)\Delta_G$, then we write $G\cong H$ and say that $G$ and $H$ are isomorphic. 
We consider the GNS $*$-representation $L^\infty(G)\to \cB(L^2(G))$ for $h_r$, and there is a unitary $V^G\in \cU(L^2(G)\otimes L^2(G))$ such that for any $a,b\in \cN_{h_r}$, we have $\Delta_G(a)(1\otimes b)\in \cN_{h_r\otimes h_r}$ and $V^G(a\otimes b) = \Delta_G(a)(1\otimes b)$ by regarding $a\otimes b$ and $\Delta_G(a)(1\otimes b)$ as elements in $L^2(G)\otimes L^2(G)$ via the GNS construction of $h_r\otimes h_r$. 
Then, $V^G (a\otimes 1) V^{G*} =\Delta_G(a)$ for all $a\in L^\infty(G)$. 
We have the well-defined locally compact quantum group $\wc{G}:=(L^\infty(\wc{G}),\Delta_{\wc{G}})$, such that 
\begin{align*}
	L^\infty(\wc{G}) & := \{ (\id_{\cB(L^2(G))}\otimes\omega)(V^G) \mid \omega\in\cB(L^2(G))_* \}'' \subset \cB(L^2(G)) , 
	\\
	\Delta_{\wc{G}} & \colon L^\infty(\wc{G}) \ni a\mapsto V^{G*}(1\otimes a)V^G \in L^\infty(\wc{G})\barotimes \cB(L^2(G)) , 
\end{align*}
where the image of $\Delta_{\wc{G}}$ actually lies in $L^\infty(G)\barotimes L^\infty(G)$. 
We have $G\cong \wc{{\wc{G}}}$. 

We let 
\begin{align*}
	&
	C_0(G) := \cspan \{ (\omega\otimes\id_{\cB(L^2(G))})(V^G) \mid \omega\in\cB(L^2(G))_* \} \subset \cB(L^2(G)) , 
\end{align*}
which is a non-degenerate C*-subalgebra of $\cB(L^2(G))$ such that $C_0(G)''=L^\infty(G)$, and $\Delta_G$ restricts to a non-degenerate $*$-homomorphism $C_0(G)\to \cM(C_0(G)\otimes C_0(G))$. 
Applying this construction to the locally compact quantum group $\wc{G}$, we obtain $C_0(\wc{G})$ as a C*-subalgebra of $L^\infty(\wc{G})$. The GNS $*$-representation of $C_0(\wc{G})$ on $L^2(\wc{G})$ is canonically identified with $C_0(\wc{G})\subset\cB(L^2(G))$ so that 
\begin{align*}
	&
	C_0(\wc{G}) = \cspan \{ (\id_{\cB(L^2(G))}\otimes\omega)(V^G) \mid \omega\in\cB(L^2(G))_* \} \subset \cB(L^2(G)) . 
\end{align*}
With this identification, we have $V^G\in \cU\cM( C_0(\wc{G}) \otimes C_0(G) )$. 

We write $G^{\op}:=(L^\infty(G),(\cdot)_{21}\circ\Delta_{G})$, where $(\cdot)_{21}$ indicates the $*$-automorphism $L^\infty(G)\barotimes L^\infty(G)\ni x\otimes y \mapsto y\otimes x\in L^\infty(G)\barotimes L^\infty(G)$ by the slight abuse of the so-called leg notation. Then, $G^{\op}$ is a locally compact quantum group. 
Depending on the reference, the Pontryagin dual of $G$, often denoted by $\hat{G}$, may refer to our $\wc{G}$ 
or the locally compact quantum group that is isomorphic to $(\wc{G})^{\op}$. In this paper, we shall use $\wc{G}$ to clarify our convention of the Pontryagin dual. 
As $(\wc{G})^{\op}\cong (G^{\op})^{\vee}$, we can often easily translate the results on one convention of the Pontryagin dual to the other. 

A \emph{compact quantum group} is a locally compact quantum group $G$ such that $h_l(1)<\infty>h_r(1)$. Then, there is a unique normal state $h_G$ on $L^\infty(G)$ that is a left Haar weight and a right Haar weight simultaneously. We call $h_G$ the \emph{Haar state} on $G$. 
A locally compact quantum group $\Gamma$ such that $\Gamma\cong\wc{G}$ for some compact quantum group $G$ is called a \emph{discrete quantum group}. 
It follows from $G\cong \wc{{\wc{G}}}$ that $\wc{\Gamma}$ is a compact quantum group for any discrete quantum group $\Gamma$. 
Following the convention for locally compact groups, we write $\ell^\infty(\Gamma):=L^\infty(\Gamma)$, $c_0(\Gamma):=C_0(\Gamma)$, $\ell^2(\Gamma):=L^2(\Gamma)$, and $C(G):=C_0(G)$ for a discrete quantum group $\Gamma$ and a compact quantum group $G$. 
A locally compact quantum group $G$ is a compact quantum group if and only if $C_0(G)$ is unital. 
A compact quantum group that is also a discrete quantum group is called a \emph{finite quantum group}. By definition, a locally compact quantum group $G$ is a finite quantum group if and only if $\wc{G}$ is a finite quantum group. It is not hard to see 
that a locally compact quantum group $G$ is a finite quantum group if and only if $\dim_\bC C_0(G)<\infty$, which is equivalent to $\dim_\bC L^2(G)<\infty$. 

\begin{rem}\label{notation_GGamma}
	We often use the symbols $\Gamma$ and $\Lambda$ for discrete quantum groups. When a compact quantum group denoted by $G$ or $H$ is given, we adopt the convention that $\Gamma$ or $\Lambda$ denotes the discrete quantum group such that $\Gamma = \wc{G}$ or $\Lambda = \wc{H}$, respectively, unless specified otherwise. 
\end{rem}

\subsection{Unitary representations}\label{ssec_prelim_Rep}

We refer to \cite{Neshveyev-Tuset-book} for terminologies and basic facts on C*-tensor categories and representation theory of compact quantum groups. For objects $\pi,\tau$ in a C*-tensor category $\cC$, we write $\pi\cong \tau$ if there is a unitary in $\Hom_\cC(\pi,\tau)$ and $\pi\leq\tau$ if there is an isometry in $\Hom_\cC(\pi,\tau)$. 
An example of a C*-tensor category is the category $\Hilb$ of finite dimensional Hilbert spaces with tensor products. As in the discussion after \cite[Definition 2.1.1]{Neshveyev-Tuset-book}, we shall regard that $\Hilb$ is strict. 

Let $G$ be a locally compact quantum group. A pair $(\cH,V)$ of a Hilbert space $\cH$ and a unitary $V\in\cU\cM(\cK(\cH)\otimes C(G))$ is a \emph{unitary representation} of $G$ if $(\id_{\cK(\cH)}\otimes\Delta_G)(V) = V_{12}V_{13}$ by using the so-called leg notation. 
For example, $(L^2(G),V^G)$ is a unitary representation of $G$, which stands for the analog of the right regular representation. 
Also, $(L^2(G),V^G_{21})$ is a unitary representation of $\wc{G}$. 
We write $\mathbbm{1} := (\bC, 1_{\cM(C_0(G))})$ and call it the \emph{trivial} unitary representation of $G$. 
There is a `universal' unitary representation $(\cH^u,V^u)$ of $G$ such that $C^u_0(\wc{G}) := \{ (\id_{\cB(\cH^u)}\otimes\omega)(V^u) \mid \omega\in\cB(L^2(G))_* \} \subset\cB(\cH^u)$ is a non-degenerate C*-subalgebra such that, for each Hilbert space $\cG$, there is a canonical bijective correspondence between unitary representations of $G$ on $\cG$ and non-degenerate $*$-representations of $C^u_0(\wc{G})$ on $\cG$. Moreover, there are a canonical non-degenerate $*$-homomorphism $\Delta_G^u\colon C_0^u(G)\to \cM(C_0^u(G)\otimes C_0^u(G))$ such that $(\Delta_G^u\otimes\id_{C_0^u(G)})\Delta_G^u = (\id_{C_0^u(G)}\otimes \Delta_G^u)\Delta_G^u$ and a surjective $*$-homomorphism $\rho_G\colon C_0^u(G)\to C_0(G)$ such that $(\rho_G\otimes\rho_G)\Delta_G^u = \Delta_G\rho_G$. 
Moreover, there is a unique non-zero $*$-homomorphism $\epsilon_G\colon C_0^u(G)\to \bC$ such that $(\epsilon_G\otimes\epsilon_G)\Delta_G^u = \epsilon_G$, called the \emph{counit}. 
Note that $(\id_{C_0^u(G)}\otimes\epsilon_G)\Delta_G^u = \id_{C_0^u(G)} = (\epsilon_G\otimes\id_{C_0^u(G)})\Delta_G^u$. 
We say that $G$ is \emph{coamenable} if $\rho_G$ is a $*$-isomorphism. 
Then, $G$ is coamenable if and only if $\epsilon_G$ factors through $\rho_G$. 
For details, see e.g., \cite[Section 5]{Soltan-Woronowicz2007multiplicativeII}. 
Discrete quantum groups and compact Hausdorff groups are known to be coamenable. 
When $G$ is compact, we write $C^u(G):=C^u_0(G)$. 

We write $\Rep G$ for the C*-tensor category whose objects are unitary representation $(\cH,V)$ of $G$ with $\dim_\bC\cH<\infty$ such that 
\begin{align*}
	&
	\Hom_{\Rep G}((\cH,V),(\cG,U)) := \{ X\in \cB(\cH,\cG) \mid (X\otimes 1)V = U(X\otimes 1) \} 
\end{align*}
and $(\cH,V)\otimes (\cG,U) := (\cH\otimes\cG, V_{13}U_{23})$ for $(\cH,V),(\cG,U) \in \Rep G$ with the unit object $\mathbbm{1}:=(\bC,1)$. 
We say an object $\pi\in\Rep G$ is \emph{irreducible} if $\Hom_{\Rep G}(\pi,\pi)\cong\bC$. 
Let $\Irr G$ be the set of fixed representatives for isomorphism classes of irreducible objects in $\cC$ such that $\mathbbm{1}\in \Irr G$. 
When $G$ is either a compact quantum group or a discrete quantum group, then $\Rep G$ is a unitary tensor category, i.e., a rigid C*-tensor category whose unit object is irreducible. Here, the discrete case is due to \cite[Theorem 4.5]{Soltan2005quantum}. 

Let $G$ be a compact quantum group and $\Gamma:=\wc{G}$ as in \autoref{notation_GGamma}. 
For $\pi=(\cH,V)\in\Rep G$, we write $\cH_\pi:=\cH$ and $V_\pi:=V$. 
Then, there is a unique non-degenerate $*$-homomorphism $\Pi_\pi\colon c_0(\Gamma)\to \cB(\cH_\pi)$ such that $(\Pi_\pi\otimes\id_{C(G)})(V^G)=V_\pi$. Then, $\prod_{\pi\in\Irr G}\Pi_\pi\colon c_0(\Gamma) \to \bigoplus_{\pi\in\Irr G}\cB(\cH_\pi)$ is a $*$-isomorphism (see also \eqref{eq_rem_biGalois_3}). 
If no confusion is likely to occur, we omit $\Pi_\pi$ with abuse of notation. 
We set 
\begin{align*}
	&
	\cO(G) := \{ (\omega\otimes\id_{C(G)})(V_\pi) \mid \pi\in\Rep G, \omega\in\cB(\cH_\pi)_* \} \subset C(G), 
\end{align*}
which is a norm-dense unital $*$-subalgebra of $C(G)$ with $\Delta_G(\cO(G))\subset \cO(G) \odot \cO(G)$. 

Also, $C^u(G)$ is the universal C*-completion of $\cO(G)$, and we have $\Delta^u_G|_{\cO(G)}=\Delta_G|_{\cO(G)}$ (see e.g., \cite[Section 1.2]{Daws-Kasprzak-Skalski-Soltan2012closed}). 
Then, by the universality of $C^u(G)$ above, 
for each Hilbert space $\cH$, there is a canonical bijection between the sets 
$\{ U \mid (\cH,U)\in\Rep\Gamma \}$ and $\{ \text{unital $*$-homomorphisms} \cO(G)\to\cB(\cH) \}$. 
Explicitly, this bijective correspondence is given so that $U$ in the former set and $\Pi$ in the latter set determine each other by, for all $\pi\in\Irr G$, 
\begin{align*}
	&
	(\id_{\cB(\cH_\pi)}\otimes\Pi)(V_\pi)=(\Pi_\pi\otimes\id_{\cB(\cH)})(U_{21}) 
	\in \cU(\cH_\pi\otimes\cH) . 
\end{align*}

For any $\theta=(\cH,U)\in\Rep\Gamma$, we write $\cV_\theta:=\cH$ and $V_\theta := U_{21} \in \cU\cM(c_0(\Gamma)\otimes\cK(\cH))$. Also, we write $\Pi_\theta\colon\cO(G)\to \cB(\cV_\theta)$ for the unital $*$-representation corresponding to $V_\theta$ via the bijection above (with respect to the underlying Hilbert space $\cV_\theta$) and set $V_{\pi,\theta} := (\id_{\cB(\cH_\pi)}\otimes\Pi_\theta)(V_\pi)$ for $\pi\in\Rep G$.

For $\phi,\psi\in\Rep\Gamma$, we have that $(\Pi_\phi\otimes\Pi_\psi)\Delta_{G} = \Pi_{\phi\otimes\psi}$ 
since for all $\pi\in\Irr G$, 
\begin{align*}
	&
	\bigl(\id_{\cB(\cH_\pi)}\otimes((\Pi_\phi\otimes\Pi_\psi)\Delta_{G})\bigr)(V_\pi) = (V_{\pi,\phi})_{12} (V_{\pi,\psi})_{13} 
	\\={}& 
	V_{\pi,\phi\otimes\psi} = (\id_{\cB(\cH_\pi)}\otimes\Pi_{\phi\otimes\psi})(V_\pi) . 
\end{align*}
Also, we have 
\begin{align*}
	&
	\Hom_{\Rep \Gamma}(\phi,\psi) 
	= 
	\{ X\in \cB(\cV_\phi,\cV_\psi) \mid V_{\psi} (1\otimes X) = (1\otimes X)V_\phi \}
	\\={}& 
	\{ X\in \cB(\cV_\phi,\cV_\psi) \mid \Pi_\psi(a)X = X\Pi_\phi(a), \forall a\in\cO(G) \} . 
\end{align*}
Therefore, the unitary tensor category $\Rep\Gamma$ is canonically unitarily monoidally equivalent to the C*-tensor category of finite dimensional unital $*$-representations of $\cO(G)$ with intertwiners as morphisms such that the monoidal structure is given by $(\cH,\Pi)\otimes(\cG,\Upsilon)=(\cH\otimes\cG,(\Pi\otimes\Upsilon)\Delta_G)$ for finite dimensional unital $*$-representations $(\cH,\Pi),(\cG,\Upsilon)$ of $\cO(G)$.

Here, note that this convention of unitary representations of $G$ and $\Gamma$ is somewhat ambiguous when $G$ and $\Gamma$ are finite quantum groups. In this case, we shall follow the convention of $G$ as a compact quantum group and $\Gamma$ as a discrete quantum group. 
For example, for $\pi\in\Rep G$ and $\theta\in\Rep\Gamma$, the underlying Hilbert spaces are denoted by $\cH_\pi$ and $\cV_\theta$ (rather than $\cV_\pi$ and $\cH_\theta$), respectively. 

\subsection{Homomorphisms}\label{ssec_prelim_hom}
For compact quantum groups $H$ and $G$, we say that a linear map $f\colon \cO(H)\to \cO(G)$ is \emph{comultiplicative} if $\Delta_{G}f = (f\otimes f)\Delta_{H}$. 
Let $f\colon \cO(H)\to \cO(G)$ be a unital comultiplicative $*$-homomorphism. We call such $f$ a \emph{$*$-bialgebra map}. Morally, we regard such $f$ as the quantum analog of a group homomorphism $\Lambda\to \Gamma$, where we recall $\Gamma=\wc{G}$ and $\Lambda=\wc{H}$. 
For $\pi\in\Rep H$, we express $\pi=(\cH_\pi,V_\pi)$. 

Take a $*$-bialgebra map $f\colon\cO(H)\to \cO(G)$ and objects $\pi,\tau\in\Rep H$ arbitrarily. 
Then, there is $f_*\pi \in \Rep G$ defined by $f_*\pi = (\cH_{f_*\pi}, V_{f_*\pi}) := (\cH_\pi, (\id_{\cB(\cH_\pi)}\otimes f)(V_\pi))$. 
By definition, we have 
$\Hom_{\Rep H}(\pi,\tau)\subset\Hom_{\Rep G}(f_*\pi,f_*\tau)$ as subspaces of $\cB(\cH_\pi,\cH_\tau)$. 
Also, the identity map on $\cH_{(f_*\pi)\otimes (f_*\tau)} := \cH_{\pi}\otimes \cH_{\tau} = \cH_{\pi\otimes\tau} =: \cH_{f_*(\pi\otimes \tau)}$ induces a unitary isomorphism $(f_*\pi)\otimes (f_*\tau)\cong f_*(\pi\otimes\tau)$ in $\Rep G$. 
With these data, it is routine to check that we have a well-defined unitary tensor functor $f_*\colon \Rep H \ni \pi \mapsto f_*\pi \in \Rep G$. 

Similarly, take $\phi,\psi\in\Rep\Gamma$ arbitrarily. 
Then, there is $f^*\phi\in \Rep \Lambda$ uniquely determined by $(\cV_{f^*\phi},\Pi_{f^*\phi}) := (\cV_\phi,\Pi_{\phi}\circ f)$. 
By definition, we have 
$\Hom_{\Rep \Gamma}(\phi,\tau)\subset\Hom_{\Rep \Lambda}(f^*\phi,f^*\psi)$ as subspaces of $\cB(\cV_\phi,\cV_\psi)$. 
Also, the identity map on $\cV_{(f^*\phi)\otimes (f^*\psi)} := \cV_{\phi}\otimes \cV_{\psi} = \cV_{\phi\otimes\psi} =: \cV_{f^*(\phi\otimes \psi)}$ induces a unitary isomorphism $(f^*\phi)\otimes (f^*\psi)\cong f^*(\phi\otimes\psi)$ in $\Rep\Lambda$. 
With these data, it is routine to check that we have a well-defined unitary tensor functor $f^*\colon \Rep \Gamma \ni \phi \mapsto f^*\phi \in \Rep \Lambda$. 

If we are further given a compact quantum group $F$ and a $*$-bialgebra map $g\colon \cO(G)\to \cO(F)$, then we have natural unitary monoidal isomorphisms $(gf)_* \cong g_*f_* \colon \Rep H\to \Rep F$ and $(gf)^* \cong f^*g^*\colon \Rep \wc{F}\to \Rep \Lambda$. 

\begin{eg}\label{eg_bialg_hom_counit}
	Let $H,G$ be compact quantum groups and $f\colon\cO(H)\to \cO(G)$ be a $*$-bialgebra map. 
	\begin{enumerate}[leftmargin=*,label=(\arabic*)]
		\item\label{item_eg_bialg_hom_counit_1}
		The restriction of the counit is a $*$-bialgebra map $\epsilon_H\colon \cO(H)\to \bC$, where we regard $\bC$ as continuous functions on the trivial group $\{1\}$ equipped with the trivial comultiplication. 
		\item\label{item_eg_bialg_hom_counit_2}
		Note that $(\id_{\cK(\cH)}\otimes \epsilon_F)(U)=1_{\cB(\cH)}$ for any unitary representation $(\cH,U)$ of a locally compact quantum group $F$ since a unitary representation of the trivial group is trivial. 
		In particular, $(\epsilon_G)_*\colon \Rep G\to \Rep \{1\}=\Hilb$ and $(\epsilon_G)^*\colon\Rep\Gamma\to \Rep\wc{\{1\}}=\Hilb$ are nothing but forgetful functors. 
		Thus, compositions of $f_*$ and $f^*$ with the forgetful functors $\Rep H\to \Hilb$ and $\Rep \Gamma\to \Hilb$, respectively, are naturally unitarily monoidally isomorphic to the forgetful functors from $\Rep G$ and $\Rep\Lambda$, respectively. 
	\end{enumerate}
\end{eg}

\begin{rem}\label{rem_bialg_hom_inj}
	Let $H,G$ be compact quantum groups and $f\colon\cO(H)\to \cO(G)$ be an injective $*$-bialgebra map. Then, by universality, $f$ uniquely extends to the unital $*$-homomorphism denoted by $f^u\colon C^u(H)\to C^u(G)$, which satisfies $\Delta^u_G f^u= (f^u\otimes f^u)\Delta^u_H$. We claim that $f^u$ is injective. 
	
	Note that for all $x\in\cO(H)$, since 
	\begin{align*}
		&
		f(\id_{\cO(H)}\otimes h_Gf) \Delta_H(x) = (\id_{\cO(G)}\otimes h_G) \Delta_Gf(x) = f (h_{G}f(x) 1_{\cO(H)}) , 
	\end{align*}
	it follows from the injectivity of $f$ that 
	\begin{align*}
		h_Gf(x) ={}& h_H( h_Gf(x) 1_{\cO(H)} ) = h_H(\id_{\cO(H)}\otimes h_Gf) \Delta_H(x) 
		\\={}& h_H(x) h_Gf(1_{\cO(H)}) = h_H(x) . 
	\end{align*}
	Thus, $f$ uniquely extends to a unital normal $*$-homomorphism $L^\infty(H)\to L^\infty(G)$ denoted by $\overline{f}$, which satisfies $\Delta_G \overline{f} = (\overline{f}\otimes\overline{f}) \Delta_H$ by the $\sigma$-weak density. 
	In this case, there is the well-defined surjective $*$-homomorphism $f^{\vee}\colon c_0(\Gamma)\to c_0(\Lambda)$ determined by the relation $(f^{\vee}\otimes\id_{C(G)})(V^G) = (\id_{c_0(\Lambda)}\otimes f)(V^H)$, and $f^\vee$ satisfies $\Delta_{\Lambda}f^{\vee}=(f^{\vee}\otimes f^{\vee})\Delta_{\Gamma}$ by \cite[Theorem 3.5, Theorem 3.6]{Daws-Kasprzak-Skalski-Soltan2012closed} and the coamenability of $\Gamma$ and $\Lambda$. 
	Then, $f^\vee$ uniquely extends to a surjective unital normal $*$-homomorphism $\overline{f}^{\vee}\colon \ell^\infty(\Gamma) = \cM(c_0(\Gamma))\to \cM(c_0(\Lambda)) = \ell^\infty(\Lambda)$ such that $\Delta_{\Lambda}\overline{f}^{\vee}=(\overline{f}^{\vee}\otimes \overline{f}^{\vee})\Delta_{\Gamma}$. 
	Now, it follows from \cite[Theorem 5.2]{Kalantar-Kasprzak-Skalski2016open} that $f^u\colon C^u(H)\to C^u(G)$ is injective as claimed. 
\end{rem}

\begin{rem}\label{rem_bialg_hom_surj}
	Let $H,G$ be compact quantum groups and $r\colon\cO(H)\to \cO(G)$ be a surjective $*$-bialgebra map. 
	\begin{enumerate}[leftmargin=*,label=(\arabic*)]
		\item\label{item_rem_bialg_hom_surj_1}
		Then, for any $\pi\in\Irr G$, there is $\wt{\pi}\in\Irr H$ such that $\pi\leq\wt{\pi}$. Indeed, by the surjectivity of $r$, there is some $\tau\in\Rep H$ such that $\{(\omega\otimes\id_{C(G)})(V_\pi) \mid \omega \in\cB(\cH_\pi)_*\} \subset \{(\omega\otimes r)(V_\tau) \mid \omega \in\cB(\cH_\tau)_*\}$, and by Schur's orthogonality relation, we see $\Hom_{\Rep G}(\pi,r_*\tau) \neq 0$, which means $\Hom_{\Rep G}(\pi,r_*\wt{\pi}) \neq 0$ for some $\wt{\pi} \in\Irr H$ appearing the irreducible decomposition of $\tau$. 
		\item\label{item_rem_bialg_hom_surj_2}
		Also, $r$ is bijective if and only if $r_*\colon\Rep H\to \Rep G$ is fully faithful. The only if part is clear. Conversely, if $r_*\colon\Rep H\to \Rep G$ is fully faithful, then it is a unitarily monoidal equivalence as it is essentially surjective by \ref{item_rem_bialg_hom_surj_1}. Since $r_*$ is compatible with the forgetful functor, the quasi-inverse of $r_*$ induces the inverse of $r$ via the Tannaka--Krein duality. 
		\item\label{item_rem_bialg_hom_surj_3}
		If $r$ is bijective, then it uniquely extends to a $*$-isomorphism $C^u(H)\to C^u(G)$ and to a normal $*$-isomorphism $L^\infty(H)\to L^\infty(G)$, where the latter restricts to a $*$-isomorphism $C(H)\to C(G)$ (cf.\ \autoref{rem_bialg_hom_inj}). These $*$-isomorphisms preserve the comultiplication by the appropriate density of $\cO(H)$. In particular, $H\cong G$. 
	\end{enumerate}
\end{rem}

\section{Separation by unitary representations}\label{sec_separateDQG}
As noted in \autoref{notation_GGamma}, we promise that $\Gamma=\wc{G}$ and $\Lambda=\wc{H}$ when $G$ and $H$ are compact quantum groups. 

\subsection{An analog of faithfulness}\label{ssec_separateDQG}

For our approach to outer actions on the Jiang--Su algebra of a discrete quantum group $\Gamma$, we need a finite dimensional unitary representation of $\Gamma$ that is faithful in appropriate sense. 
With this in mind, we consider the following property. 

\begin{df}\label{def_tensor_faith}
	Let $G$ be a compact quantum group and $\theta \in\Rep\Gamma$. 
	We write $\Pi_{\theta}^{(\infty)}:=\prod_{n=1}^{\infty}\Pi_{\theta^{\otimes n}} \colon \cO(G) \to \prod_{n=1}^{\infty}\cB(\cV_{\theta^{\otimes m}})$. 
	We say that $\theta$ is \emph{tensorially faithful} if $\Pi_{\theta}^{(\infty)}$ is injective. 
\end{df}

\begin{eg}\label{eg_tensor_faith_finite}
	Let $G$ be a finite quantum group with $\Gamma=\wc{G}$. 
	Then, $\rho:=(L^2(G),V^{G}_{21})\in\Rep \Gamma$ is tensorially faithful because we have $\Ker\Pi_\rho=0$ by the definitions of $\cO(G)$ and $\Pi_\rho$. 
\end{eg}

\begin{eg}\label{eg_tensor_faith_Lie}
	Let $G$ be a compact Hausdorff group. For each $g\in G$, we define the $*$-homomorphism $\ev_g\colon C(G)\to\bC$ by the evaluation at each $g\in G$. Then, by the Gelfand duality for $C^u(G)=C(G)$, $(\cV_\theta,\Pi_\theta)$ for any $\theta\in\Rep\Gamma$ is unitarily equivalent to $\bigoplus_{i=1}^{n}(\bC,\ev_{g_i})$ for some $g_1,\cdots,g_n\in G$. 
	Then, $\theta\in\Rep G$ is tensorially faithful if and only if the subsemigroup generated by $g_1,\cdots,g_n$ is dense in $G$. 
	This is equivalent to saying that the subgroup generated by $g_1,\cdots,g_n$ is dense in $G$ since a closed subsemigroup in a compact Hausdorff group is a subgroup (see e.g., \cite[Example 1.1.2]{Neshveyev-Tuset-book}). 
	
	Note that any compact Lie group $G$ has a finite subset $F\subset G$ that generates a dense subgroup of $G$. 
		Indeed, inductively on $k$, we can take the closures $T_k$ of some $1$-parameter subgroups in $G$ such that the closed subgroup $G_k$ of $G$ generated by $T_1,\cdots,T_k$ satisfies $\dim G_k\geq \min\{ k, \dim G \}$. We take a subset $F_0\subset G$ of representatives for the classes in $G/G_{\dim G}$, which is finite, and for each $k=1,\cdots,\dim G$, a finite subset $F_k\subset T_k$ that generates a dense subgroup of $T_k$. Then, the finite subset $F:=\bigcup_{k=0}^{\dim G}F_k$ generates a dense subgroup of $G$. 
	Thus, there is $\theta\in\Rep\Gamma$ that is tensorially faithful. 
\end{eg}

\begin{rem}\label{rem_tensor_faith_op}
	For a compact quantum group $G$, 
	the existence of $\theta\in\Rep\Gamma$ that is tensorially faithful is equivalent to the existence of $\theta\in\Rep{\Gamma}^{\op}$ that is tensorially faithful. 
	Indeed, $\prod_{k=1}^{\infty} \Pi_{\theta}^{\otimes k}\Delta_G^{(k-1)}$ is injective on $\cO(G)$ if and only if $\prod_{k=1}^{\infty} \Pi_{\theta}^{\otimes k}\Delta_{G^{\op}}^{(k-1)}$ is injective on $\cO(G^{\op})=\cO(G)$. 
	Here, we let $\Delta_{G}^{(0)} := \id_{C(G)}$ and 
	$\Delta_{G}^{(n)} := (\Delta_{G}^{(n-1)}\otimes\id_{C(G)}) \Delta_{G}$ inductively on $n\in\bZ_{\geq 1}$. 
\end{rem}

\begin{rem}\label{rem_RFD_Kac}
	Let $G$ be a compact quantum group. 
	We say that $\Gamma$ is \emph{RFD} if for any $0\neq a\in\cO(G)$, there are $n\in\bZ_{\geq 1}$ and a unital $*$-homomorphism $\Pi\colon \cO(G)\to \bM_n$ with $\Pi(a)\neq0$. 
	If $\Gamma$ is RFD, then the compact quantum group $G$ is of Kac type (i.e., the Haar state is tracial) by \cite[Corollary A.3]{Soltan2005quantum} and the discussion after \cite[Corollary 2.12]{Chirvasitu2015residually}. 
	Note that $\Gamma$ is RFD if and only if $\bigcap_{\theta\in\Rep\Gamma}\Ker\Pi_\theta =0$. 
	In particular, if there is $\theta\in\Rep \Gamma$ that is tensorially faithful, then $\Gamma$ is RFD and thus $G$ is of Kac type. 
	
	If $G$ is a compact quantum group such that $\Gamma$ is RFD, then for any $\pi\in\Rep G$, we have that the transposition of $V_\pi$ with respect to some fixed basis of $\cH_\pi$, denoted by $V_\pi^{\top}$, is also a unitary because $G$ is of Kac type. Then, for any $\theta\in\Rep\Gamma$, we have that $V_{\pi,\theta}$ is a biunitary. Here, for $m,n\in\bZ_{\geq 1}$ and $U\in\cU(\bM_m\otimes\bM_n)$, we say that $U$ is a \emph{biunitary} if, when we express $U=(U_{i,j})_{i,j}$ using $U_{i,j}\in\bM_m$ for $i,j\in\{1,\cdots,n\}$, then $(U_{j,i})_{i,j}$ is also a unitary in $\bM_m\otimes\bM_n$. Note that $U$ is a biunitary if and only if $(U_{i,j}^{\top})_{i,j} = ((U_{j,i})_{i,j})^{\top}$ is a unitary. 
\end{rem}

For a compact quantum group $G$, we recall that $\cO(G)$ is a Hopf algebra with the canonically defined antipode denoted by $S^G\colon \cO(G)\to \cO(G)$. 
From the general theory of Hopf algebras, 
each of 
\begin{align}\label{eq_fact_Hopf_cancellative}
	&
	(\cO(G)\otimes 1_{\cO(G)}) \Delta_{G}(\cO(G)) 
	\qquad \text{and} \qquad 
	(1_{\cO(G)}\otimes \cO(G)) \Delta_{G}(\cO(G)) , 
\end{align}
linearly spans $\cO(G) \odot \cO(G)$. 
Also, for a compact group $H$ and a $*$-bialgebra map $f\colon\cO(G)\to \cO(H)$, we have 
\begin{align}\label{eq_fact_Hopf_antipodehom}
	fS^G = S^H f. 
\end{align}
A \emph{Hopf $*$-ideal} of $\cO(G)$ is a $*$-ideal $I\subset\cO(G)$ such that $\Delta_{G}(I) \subset I\odot\cO(G) + \cO(G)\odot I$ and $S^G(I)\subset I$. 
In \cite{Banica-Bichon2010Hopf,Brannan-Chirvasitu-Freslon2020topological}, $\theta\in\Rep\Gamma$ is called \emph{inner faithful} if no non-zero Hopf $*$-ideal of $\cO(G)$ is contained in $\Ker\Pi_\theta$. 
Then, \autoref{lem_tensor_faith} below implies that $\theta$ is inner faithful if and only if $\theta$ is tensorially faithful.

\begin{lem}\label{lem_tensor_faith}
	Let $G$ be a compact quantum group and $\theta=(\cV_\theta,\Pi_\theta)\in\Rep\Gamma$. 
	\begin{enumerate}[leftmargin=*,label=(\arabic*)]
		\item\label{item_lem_tensor_faith_ker}
		$\Ker\Pi_{\theta}^{(\infty)}$ is a Hopf $*$-ideal of $\cO(G)$ satisfying $\Ker\Pi_{\theta}^{(\infty)} \subset \Ker\Pi_\theta$. 
		\item\label{item_lem_tensor_faith_max}
		If $I$ is a $*$-ideal of $\cO(G)$ satisfying $I\subset \Ker\Pi_\theta$ and $\Delta_{G}(I) \subset I\odot \cO(G) + \cO(G)\odot I$, then $I\subset \Ker\Pi_{\theta}^{(\infty)}$. 
		\item\label{item_lem_tensor_faith_quot}
		There is a compact quantum group $G_\theta$ with a surjective $*$-bialgebra map $r_\theta\colon \cO(G)\to \cO(G_\theta)$ satisfying $\Ker r_\theta = \Ker\Pi_{\theta}^{(\infty)}$. 
	\end{enumerate}
\end{lem}

\begin{proof}
	We define $\Delta_{G}^{u(0)} := \id_{C(G)}$ and 
	$\Delta_{G}^{u(n)} := (\Delta_{G}^{u(n-1)}\otimes\id_{C^u(G)}) \Delta_{G}^u$ inductively on $n\in\bZ_{\geq 1}$. 
	By definition, $\Ker\Pi_{\theta}^{(\infty)} = \bigcap_{n=1}^{\infty} \Ker\Pi_{\theta^{\otimes n}}$. It follows that $\Ker\Pi_{\theta}^{(\infty)}$ is a $*$-ideal of $\cO(G)$ contained in $\Ker\Pi_\theta$. 
	Consider any $*$-ideal $I\subset\cO(G)$ with $I\subset\Ker\Pi_\theta$ and $\Delta_{G}(I)\subset I\odot\cO(G) + \cO(G)\odot I$. 
	Since for any $n\in\bZ_{\geq 1}$, we have 
	\begin{align*}
		&
		\Delta_{G}^{(n)}(I) 
		\subset 
		\sum_{k=0}^{n} \cO(G)^{\odot k} \odot I \odot \cO(G)^{\odot n-k} 
		\subset 
		\Ker(\Pi_\theta^{\otimes n+1}) , 
	\end{align*}
	we see that $I\subset \Ker ( \Pi_\theta^{\otimes n+1} \Delta_{G}^{(n)} ) = \Ker\Pi_{\theta^{\otimes n+1}}$. 
	Thus, $I\subset\bigcap_{n=1}^{\infty} \Ker\Pi_{\theta^{\otimes n}} = \Ker\Pi_{\theta}^{(\infty)}$, which implies \ref{item_lem_tensor_faith_max}. 
	
	Next, we show \ref{item_lem_tensor_faith_quot}. 
	For each $\phi\in\Rep\Gamma$, by universality, $\Pi_{\phi}$ and $\Pi_{\phi}^{(\infty)}$ uniquely extend to the unital $*$-homomorphisms still denoted by $\Pi_\phi\colon C^u(G)\to \cB(\cV_\phi)$ and $\Pi_\phi^{(\infty)}\colon C^u(G)\to \prod_{k=1}^{\infty} \cB(\cV_\phi^{\otimes k})$. 
	For all $m,n\in\bZ_{\geq 1}$, we have 
	\begin{align*}
		&
		(\Pi_{\theta^{\otimes m}}\otimes \Pi_{\theta^{\otimes n}}) \Delta_{G}^{u} 
		= \Pi_\theta^{\otimes m+n} \Delta_{G}^{u(m+n-1)} 
		\colon 
		C^u(G)\to \cB(\cV_\theta)^{\otimes m+n}=\cB(\cV_{\theta^{\otimes m+n}}) 
	\end{align*}
	since they coincide on $\cO(G)$. 
	Thus, we have the well-defined unital $*$-homomorphism 
	\begin{align*}
		&
		\Delta\colon \Pi_{\theta}^{(\infty)}(C^u(G)) 
		\to 
		\prod_{m,n=1}^{\infty} \cB(\cV_{\theta^{\otimes m}}) \otimes \cB(\cV_{\theta^{\otimes n}}) 
		\subset 
		\cB\biggl( \bigoplus_{m=1}^{\infty} \cV_{\theta^{\otimes m}} \biggr)^{\barotimes 2} 
	\end{align*}
	such that $\Delta\Pi_{\theta}^{(\infty)} = (\Pi_{\theta}^{(\infty)} \otimes \Pi_{\theta}^{(\infty)}) \Delta_G^u$. 
	The image of $\Delta$ is contained in the norm-closure of $\Pi_{\theta}^{(\infty)}(\cO(G))\odot\Pi_{\theta}^{(\infty)}(\cO(G))$, which is the spatial tensor product $\Pi_{\theta}^{(\infty)}(C^u(G))\otimes \Pi_{\theta}^{(\infty)}(C^u(G)) \subset \cB( \bigoplus_{m=1}^{\infty} \cV_{\theta^{\otimes m}} )^{\barotimes 2}$. 
	By \eqref{eq_fact_Hopf_cancellative}, we have 
	\begin{align*}
		&
		\cspan (\Pi_{\theta}^{(\infty)}(C^u(G))\otimes 1) \Delta(\Pi_{\theta}^{(\infty)}(C^u(G))) 
		= \Pi_{\theta}^{(\infty)}(C^u(G)) \otimes \Pi_{\theta}^{(\infty)}(C^u(G)) 
		\\={}& 
		\cspan (1\otimes \Pi_{\theta}^{(\infty)}(C^u(G))) \Delta(\Pi_{\theta}^{(\infty)}(C^u(G))) .
	\end{align*}
	Then, the pair $(\Pi_{\theta}^{(\infty)}(C^u(G)),\Delta)$ is a compact quantum group possibly except for the faithfulness of the Haar state; 
	more precisely, if we define $\cO \subset \Pi_{\theta}^{(\infty)}(C^u(G))$ as the linear span of elements of the form $(\omega\otimes\id_{\Pi_{\theta}^{(\infty)}(C^u(G))})(U)$, where $\cH$ is a finite dimensional Hilbert space, $\omega\in\cB(\cH)_*$, and $U\in \cU(\cB(\cH)\otimes \Pi_{\theta}^{(\infty)}(C^u(G)))$ satisfies $(\id_{\cB(\cH)}\otimes\Delta)(U) = U_{12} U_{13}$, then it holds that $\cO \subset \Pi_{\theta}^{(\infty)}(C^u(G))$ is a well-defined norm-dense unital $*$-subalgebra with $\Delta(\cO)\subset\cO\odot\cO$ and that there are a compact quantum group $H$ (in our convention) and a unital surjective $*$-homomorphism 
	$\rho\colon \Pi_{\theta}^{(\infty)}(C^u(G))\to C(H)$ 
	such that $\Delta_{H} \rho = (\rho\otimes \rho) \Delta$ and that $\rho|_\cO$ is injective and satisfies $\rho(\cO)=\cO(H)$. 
	For details, see e.g., \cite[Chapter 1]{Neshveyev-Tuset-book}. 
	
	For any $\pi\in\Rep G$, we have $(\cH_\pi , (\id_{\cB(\cH_\pi)}\otimes\rho\Pi_{\theta}^{(\infty)})(V_\pi)) \in \Rep H$ by $\Delta_{H} \rho\Pi_{\theta}^{(\infty)} = (\rho\otimes\rho)\Delta\Pi_{\theta}^{(\infty)} = (\rho\Pi_{\theta}^{(\infty)}\otimes\rho\Pi_{\theta}^{(\infty)})\Delta_{G}$ on $\cO(G)$ and $\Pi_{\theta}^{(\infty)}(\cO(G)) \subset \cO$ by the definitions of $\cO$ and $\cO(G)$. 
	We have the well-defined $*$-bialgebra map $\rho \Pi_{\theta}^{(\infty)} \colon \cO(G)\to \cO(H)$ whose kernel equals $\Ker\Pi_{\theta}^{(\infty)}$. 
	Moreover, if $\rho \Pi_{\theta}^{(\infty)} ( \cO(G) ) \neq \cO(H)$, then there is $\tau=(\cH_\tau,V_\tau)\in \Irr H$ such that there is no $\pi\in\Irr G$ with $\tau\leq (\rho\Pi_{\theta}^{(\infty)})_*\pi$ by the definitions of $\cO(G)$ and $\cO(H)$. But Schur's orthogonality implies $h_H( x^* (\omega\otimes\id_{\cO(H)})(V_\tau) )=0$ for all $x\in\rho\Pi_\theta^{(\infty)}(\cO(G))$ and $\omega\in\cB(\cH_\tau)_*$, which contradicts the norm-density of $\rho\Pi_{\theta}^{(\infty)}(\cO(G))\subset \rho\Pi_{\theta}^{(\infty)}(C^u(G))= C(H)$. 
	Therefore, $G_\theta:=H$ with $r_\theta:=\rho\Pi_{\theta}^{(\infty)}$ satisfies the desired property. 
	
	Finally, \ref{item_lem_tensor_faith_ker} follows from \ref{item_lem_tensor_faith_quot} by \eqref{eq_fact_Hopf_antipodehom}. 
\end{proof}

\begin{eg}\label{eg_tensor_faith_group}
	When $\Gamma$ is a discrete group, a unitary representation $u\colon\Gamma\to\rU(n):=\cU(\cB(\bC^{\oplus n}))$ with $n\in\bZ_{\geq 1}$ is tensorially faithful in the sense above if and only if $\Ker u=\{1\}$. Indeed, since the kernel of $u^{\otimes k}\colon \Gamma\to \rU(n^k)$ equals $\Ker u$ for all $k\in\bZ_{\geq 1}$, in order for $u$ to be tensorially faithful, $\Ker u$ needs to be trivial. 
	The converse implication follows from \autoref{lem_tensor_faith} \ref{item_lem_tensor_faith_quot} since a compact quantum group $H$ with a surjective $*$-bialgebra map $\bC[\Gamma]\to\cO(H)$ satisfies $\Delta_{H}(x)_{21} = \Delta_{H}(x)$ for all $x\in\cO(H)$ and thus $\Lambda=\wc{H}$ is a quotient group of $\Gamma$. 
\end{eg}

Next, we show a few permanence properties. 

\begin{prop}\label{prop_tensor_faith}
	Let $n\in\bZ_{\geq 1}$ and $G_1,\cdots,G_n$, and $G$ be compact quantum groups. For each $i=1,\cdots,n$, we write $\Gamma_i:=\wc{G_i}$ and consider a surjective $*$-bialgebra map $r_i\colon \cO(G)\to\cO(G_i)$ such that 
	\begin{align}\label{eq_prop_tensor_faith}
		&
		\Hom_{\Rep G}(\pi, \varpi) = \bigcap_{i=1}^{n} \Hom_{\Rep G_i}({r_i}_*\pi, {r_i}_*\varpi) 
	\end{align}
	as subspaces of $\cB(\cH_\pi,\cH_\varpi)$ for all $\pi,\varpi\in\Rep G$. 
	If there is $\theta_i \in \Rep\Gamma_i$ that is tensorially faithful on $\cO(G_i)$ for all $i=1,\cdots,n$, then there is $\theta\in\Rep\Gamma$ that is tensorially faithful on $\cO(G)$. 
\end{prop}

Following \cite{Brannan-Chirvasitu-Freslon2020topological}, we say that $G$ is \emph{topologically generated} by $G_1,\cdots,G_n$ if, together with a specific choice of $r_i$ for each $G_i$, \eqref{eq_prop_tensor_faith} holds for all $\pi,\varpi\in\Rep G$. 
Later, we will use the topological generation results from \cite{Brannan-Chirvasitu-Freslon2020topological,Chirvasitu2020topological} for quantum permutation groups and free unitary quantum groups (see \autoref{eg_tensor_faith_Sn+} and \autoref{eg_tensor_faith_Un+}). 

\begin{proof}
	We are going to show that $\theta:=\bigoplus_{i=1}^{n}r_i^*\theta_i$ is tensorially faithful on $\cO(G)$. Writing $\theta_0:=\theta$, $G_0:=G$, and $\Gamma_0:=\Gamma$, we consider the unital $*$-homomorphism 
	$\Pi_{\theta_i}^{(\infty)}\colon \cO(G_i)\to \prod_{m=1}^{\infty}\cB(\cV_{\theta_i^{\otimes m}})$ 
	for each $i=0,1,\cdots,n$. 
	There is a compact quantum group $G_{\theta}$ with a surjective $*$-bialgebra map $r_\theta\colon\cO(G)\to \cO(G_{\theta})$ with $\Ker r_\theta=\Ker\Pi_{\theta}^{(\infty)}$ as in \autoref{lem_tensor_faith}. 
	
	Note that $\Ker\Pi_{\theta_i}^{(\infty)}=0$ if $1\leq i\leq n$ by assumption. 
	Since $\Pi_{\theta}^{(\infty)}$ contains $\Pi_{r_i^* \theta_i}^{(\infty)} = \Pi_{\theta_i}^{(\infty)} r_i$ as a $*$-subrepresentation of $\cO(G)$ for each $1\leq i\leq n$, we have $\Ker\Pi_{\theta}^{(\infty)}\subset \Ker \Pi_{\theta_i}^{(\infty)} r_i = \Ker r_i$, and thus $r_i\colon \cO(G)\to\cO(G_i)$ goes through $r_\theta$ with some unital surjective $*$-bialgebra homomorphism $\cO(G_{\theta}) \to \cO(G_i)$. 
	Then, for all $\pi,\varpi\in\Rep G$, we have 
	\begin{align*}
		&
		\Hom_{\Rep G}(\pi,\varpi) \subset \Hom_{\Rep G_\theta}({r_\theta}_*\pi, {r_\theta}_*\varpi) \subset \bigcap_{i=1}^{n}\Hom_{\Rep G_i}({r_i}_*\pi,{r_i}_*\varpi) 
	\end{align*}
	as subspaces of $\cB(\cH_\pi,\cH_\varpi)$, but the leftmost and rightmost subspaces must coincide by assumption. This implies that the canonical unitary tensor functor $\Rep G\to \Rep G_\theta$ is fully faithful, and thus $\Ker r_\theta=0$ by \autoref{rem_bialg_hom_surj} \ref{item_rem_bialg_hom_surj_2}. Therefore, $0=\Ker \Pi_{\theta}^{(\infty)}$. 
\end{proof}

\begin{prop}\label{prop_tensor_faith_ext}
	Let $G$ and $H$ be compact quantum groups, $F$ be a finite quantum group, and $f\colon \cO(H)\to \cO(G)$ be an injective $*$-bialgebra map. 
	\begin{enumerate}[leftmargin=*,label=(\arabic*)]
		\item\label{item_prop_tensor_faith_ext_sub}
		If $\Gamma$ admits some tensorially faithful $\theta=(\cV_\theta,\Pi_\theta)\in\Rep\Gamma$, then so does $\Lambda$. 
		\item\label{item_prop_tensor_faith_ext_split}
		Suppose that there is a surjective $*$-bialgebra map $r\colon \cO(G)\to \cO(F)$ such that the sequence 
		\begin{align}\label{eq_prop_tensor_faith_ext}
			&
			\begin{tikzcd}[ampersand replacement=\&]
				\cO(H) \arrow[r, "f"] \& \cO(G) \arrow[r,"r"] \& \cO(F) 
			\end{tikzcd}
		\end{align}
		is a \emph{short exact sequence} in the sense that 
		\begin{align*}
			&
			\{ x\in\cO(G) \mid (\id_{\cO(G)}\otimes r)\Delta_{G}(x) = x\otimes1_{\cO(F)} \} = f(\cO(H)) .
		\end{align*}
		If $\Lambda$ admits some tensorially faithful $\theta=(\cV_\theta,\Pi_\theta)\in\Rep \Lambda$, then so does $\Gamma$. 
	\end{enumerate}
\end{prop}

\begin{proof}
	We show \ref{item_prop_tensor_faith_ext_sub}. For $\theta\in\Rep\Gamma$ that is tensorially faithful, $f^*\theta\in\Rep\Lambda$ is tensorially faithful since $\Pi_{\theta^{\otimes n}}f=\Pi_{(f^*\theta)^{\otimes n}}$ for all $n\in\bZ_{\geq 1}$. 
	
	We show \ref{item_prop_tensor_faith_ext_split}. 
	By the universality, $f$ and $r$ uniquely extend to unital $*$-homomorphisms $f^u\colon C^u(H)\to C^u(G)$ and $r^u\colon C^u(G)\to C^u(F)$ respectively. Note that 
	$C^u(F)=\cO(F)$ by the finite dimensionality. Also, $f^u$ is injective by \autoref{rem_bialg_hom_inj}. 
	Since $h_{F}$ is a Haar state, we see that 
	\begin{align*}
		( \id_{C^u(G)} \otimes r^u) ( \Delta_{G}^u \otimes h_{F} r^u ) \Delta_{G}^u 
		={}& 
		( \id_{C^u(G)}\otimes (\id_{\cO(F)}\otimes h_{F})\Delta_{F} r^u ) \Delta_{G}^u 
		\\={}& 
		(\id_{C^u(G)}\otimes h_{F}r^u ) \Delta_{G}^u , 
	\end{align*}
	and that this map sends $\cO(G)$ to $f(\cO(H))$ and thus $C^u(G)$ to $f(C^u(H))$ by the short exactness of \eqref{eq_prop_tensor_faith_ext}. 
	Using the injectivity of $f^u$, we obtain the well-defined map $E\colon C^u(G)\to C^u(H)$ such that $f^u E = (\id_{C^u(G)}\otimes h_{F}r^u)\Delta_{G}^u$. 
	Then, $f^u E$ is a conditional expectation onto $f^u(C^u(H))$, and thus $E$ is completely positive with $Ef^u=\id_{C^u(H)}$. 
	Also, by construction, $E(x)\in\cO(H)$ for any $x\in\cO(G)$. 
	
	We claim that for any $*$-ideal $I\subset\cO(G)$ such that $0\neq I\subset \Ker r$ and that $\Delta_{G}(I) \subset I\odot \cO(G) + \cO(G)\odot I$, we have $I\cap f(\cO(H))\neq 0$. Indeed, if there is $0\neq x\in I$, then $0\neq fE(x^*x)\in f(\cO(H))$ since 
	\begin{align*}
		&
		h_{G}fE(x^*x) = (h_{G}\otimes h_{F} r)\Delta_{G}(x^*x) = h_{G}(x^*x) h_{F} r(1)  = h_{G}(x^*x) >0 
	\end{align*}
	by the faithfulness of $h_{G}$ on $C(G)\supset \cO(G)$. 
	Since $I\subset\Ker r$, 
	\begin{align*}
		fE(x^*x) 
		\in{}& 
		(\id_{\cO(G)}\otimes h_{F} r)\Delta_{G}(I) 
		\\\subset{}& 
		(\id_{\cO(G)}\otimes h_{F} r)(I \odot \cO(G) + \cO(G)\odot I) 
		\subset 
		I.
	\end{align*}
	Thus, $0\neq fE(x^*x) \in I\cap f(\cO(H))$ as claimed. 
	
	As in \autoref{eg_tensor_faith_finite}, we take $\rho\in \Rep \wc{F}$ such that $\Pi_\rho\colon \cO(F)\to \cB(\cV_\rho)$ coincides with the canonical inclusion $C(F)\to \cB(L^2(F))$. In particular, $\Ker\Pi_\rho=0$. 
	By assumption, there is $\phi\in\Rep \Lambda$ that is tensorially faithful. 
	If there is $\psi\in\Rep\Gamma$ such that $\phi\leq f^*\psi$, then by letting $\theta:= r^*\rho \oplus \psi$, 
	we see that $\Ker\Pi_{\theta}^{(\infty)}\subset\Ker r$ by $r^*\rho\leq \theta$. 
	Thus, by \autoref{lem_tensor_faith} \ref{item_lem_tensor_faith_ker}, the claim above shows that if $\Ker\Pi_{\theta}^{(\infty)}\neq 0$, then $\Ker\Pi_{\theta}^{(\infty)}\cap f(\cO(H))\neq 0$. However, since $\phi\leq f^*\psi \leq f^*\theta$, we have $\Ker\Pi_{\theta}^{(\infty)}\cap f(\cO(H)) \subset f (\Ker \Pi_{\phi}^{(\infty)} \cap\cO(H)) = 0$, which implies $\Ker\Pi_{\theta}^{(\infty)}=0$. Hence, $\theta$ must be tensorially faithful. 
	Therefore, it suffices to see that for any $\phi\in\Rep\Lambda$, there is $\psi\in\Rep\Gamma$ such that $\phi\leq f^*\psi$. 
	
	For $\pi=(\cH_\pi,V_\pi)\in\Rep F$, we take an orthonormal basis of $\cH_\pi$ to express $V_\pi=(v^{\pi}_{i,j})_{i,j}$ using $v^{\pi}_{i,j}\in \cO(G)$ for $i,j\in\{1,\cdots,\dim_\bC \cH_\pi\}$. By the definition of unitary representations, we have $\Delta_{F}(v^{\pi}_{i,j})=\sum_k v^{\pi}_{i,k}\otimes v^{\pi}_{k,j}$. 
	Since $F$ is of Kac type (cf.\ \autoref{rem_RFD_Kac}), Schur's orthogonality implies that $h_{F}(v^{\pi *}_{i,j}v^{\varpi}_{k,l}) = (\dim_\bC \cH_\pi)^{-1} \delta_{\pi,\varpi} \delta_{i,k} \delta_{j,l}$. 
	Also, we have $\cO(F) = \bigoplus_{\pi\in\Irr F}\bigoplus_{i,j=1}^{\dim_\bC\cH_\pi} \bC v^{\pi}_{i,j}$. 
	Then, it follows from the uniqueness of the counit $\epsilon_{F}$ that 
	\begin{align*}
		&
		\epsilon_{F}= \dim_\bC\cH_\pi\sum_{\pi\in \Irr F} \sum_{k=1}^{\dim_\bC\cH_\pi} h_{F}(v^{\pi *}_{k,k}(-)) 
	\end{align*}
	since for all $\tau\in\Irr F$ and $i,j\in \{ 1,\cdots,\dim_\bC\cH_\tau \}$, 
	\begin{align*}
		&
		\dim_\bC\cH_\pi
		\sum_{\pi\in \Irr F} \sum_{k=1}^{\dim_\bC\cH_\pi} (\id_{\cO(F)}\otimes h_{F}) \big( ( 1_{\cO(F)}\otimes v^{\pi *}_{k,k} ) \Delta_{F}(v^{\tau}_{i,j}) \big) 
		\\={}& 
		\dim_\bC\cH_\pi
		\sum_{\pi\in \Irr F} \sum_{k=1}^{\dim_\bC\cH_\pi} \sum_{l=1}^{\dim_\bC\cH_\tau} (\id_{\cO(F)}\otimes h_{F}) ( v^{\tau}_{i,l}\otimes v^{\pi *}_{k,k}v^{\tau}_{l,j} ) 
		= 
		v^{\tau}_{i,j} . 
	\end{align*}
	
	It follows from \eqref{eq_fact_Hopf_cancellative} and the surjectivity of $r$ that for each $\pi\in\Irr F$ and $k\in\{1,\cdots,\dim_\bC\cH_\pi\}$ there are $a_\lambda,b_\lambda\in \cO(G)$ for $\lambda$ in some finite index set $\Lambda_{\pi,k}$ such that 
	\begin{align*}
		&
		\dim_{\bC} \cH_\pi
		( 1_{\cO(G)}\otimes v^{\pi *}_{k,k} ) 
		= 
		\sum_{\lambda\in \Lambda_{\pi,k}} (a_\lambda\otimes 1_{\cO(F)}) (\id_{\cO(G)}\otimes r)\Delta_{G}(b_\lambda) . 
	\end{align*}
	Then, for any $c\in \cO(G)$, we have 
	\begin{align*}
		&
		\sum_{\pi\in \Irr F}\sum_{k=1}^{\dim_\bC\cH_\pi}\sum_{\lambda\in \Lambda_{\pi,k}}a_\lambda fE(b_\lambda c) 
		\\={}& 
		\sum_{\pi\in \Irr F}\sum_{k=1}^{\dim_\bC\cH_\pi}\sum_{\lambda\in \Lambda_{\pi,k}}a_\lambda (\id_{\cO(G)}\otimes h_{F}r)\Delta_{G}(b_\lambda c) 
		\\={}& 
		\dim_{\bC} \cH_\pi 
		\sum_{\pi\in \Irr F}\sum_{k=1}^{\dim_\bC\cH_\pi} (\id_{\cO(G)}\otimes h_{F}) \big( (1_{\cO(G)}\otimes v^{\pi *}_{k,k}) (\id_{\cO(G)}\otimes r)\Delta_{G}(c) \big) 
		\\={}& 
		(\id_{\cO(G)}\otimes \epsilon_{F} r)\Delta_{G}(c) 
		= 
		(\id_{\cO(G)}\otimes \epsilon_{G}) \Delta_{G}(c) 
		= 
		c , 
	\end{align*}
	where for the second last equality we have used $\epsilon_{F} r=\epsilon_{G}$ by the uniqueness of the counit. 
	Thus, $\cO(G)$ is generated by the finite subset $\{a_\lambda\}_{\pi,k,\lambda} \subset \cO(G)$ as a right $\cO(H)$-module. 
	
	We still write $\Pi_{\phi}\colon C^u(H)\to \cB(\cV_\phi)$ for the unique extension of $\Pi_\phi$ to a unital $*$-homomorphism. 
	Consider the KSGNS construction $\wt{\cV} := C^u(G)\otimes_{\Pi_\phi E}\cV_\phi$, on which we have the unital $*$-representation $\id_{\cO(G)}\otimes 1_{\cB(\cV_\phi)}$ of $\cO(G)$. 
	Since the image of $\cO(G)\otimes_{\cO(H)}\cV_\phi$ in $\wt{\cV}$ is norm-dense and $\cO(G)$ is a finitely generated right $\cO(H)$-module, we have $\dim_\bC\wt{\cV}<\infty$ and thus there is $\psi\in\Rep\Gamma$ with $\cV_\psi=\wt{\cV}$ and $\Pi_\psi = \id_{\cO(G)}\otimes 1_{\cB(\cV_\phi)}$. 
	Finally, since $Ef^u=\id_{C^u(H)}$, we have $\cV_\phi \cong C^u(H)\otimes_{\Pi_\phi} \cV_\phi \subset C^u(G)\otimes_{\Pi_\phi E}\cV_\phi$ as $*$-representations of $\cO(H)$, which implies $\phi\leq f^*\psi$ as desired. 
\end{proof}

\subsection{Example: free quantum groups}\label{ssec_freeQG}
We recall the definitions of several examples of compact quantum groups. 
Let $n\in \bZ_{\geq 1}$. 
\begin{itemize}[leftmargin=1.5em]
	\item
	We write $S_n^+$ for the quantum permutation group, which is described by $C^u(S_n^+) = C^*_{\rm univ}(p_{i,j} \mid i,j \in \{1,\cdots,n\}, p_{i,j}=p_{i,j}^*=p_{i,j}^2, \sum_{i=1}^{n}p_{i,j}=\sum_{i=1}^{n}p_{j,i}=1)$ and $\Delta_{S_n^+}(p_{i,j})=\sum_{k=1}^{n} p_{i,k}\otimes p_{k,j}$. 
	Then, $\cO(S_n^+)$ is the unital $*$-subalgebra generated by $\{p_{i,j}\}_{i,j}$. 
	\item
	We write $\rU^+(n)$ for the free unitary quantum group, which is described by $C^u(\rU^+(n)) = C^*_{\rm univ}(u_{i,j} \mid i,j \in \{1,\cdots,n\}, \sum_{k=1}^{n}u_{i,k}u_{j,k}^*=\sum_{k=1}^{n}u_{k,i}^*u_{k,j}=\delta_{i,j})$ and $\Delta_{\rU^+(n)}(u_{i,j})=\sum_{k=1}^{n} u_{i,k}\otimes u_{k,j}$. 
	Then, $\cO(\rU^+(n))$ is the unital $*$-subalgebra generated by $\{u_{i,j}\}_{i,j}$. 
	\item
	Let $\varsigma\in\{\pm1\}$ and $F=(f_{i,j})_{i,j}\in\bM_n$ such that $F\overline{F}=\varsigma1_{\bM_n}=\overline{F}F$, where $\overline{F}:=(f^*_{i,j})_{i,j}\in\bM_n$. We write $\rO^+(F)$ for the free orthogonal group, which is described by 
	$C^u(\rO^+(F)) = C^*_{\rm univ}( s_{i,j} \mid i,j \in \{1,\cdots,n\}, \sum_{k=1}^{n}s_{i,k}s_{j,k}^*=\sum_{k=1}^{n}s_{k,i}^*s_{k,j}=\delta_{i,j}, s_{i,j} = \varsigma \sum_{k,l=1}^{n}f_{i,k}s_{k,l}^*f^*_{l,j} )$ and $\Delta_{\rO^+(F)}(s_{i,j})=\sum_{k=1}^{n} s_{i,k}\otimes s_{k,j}$. 
	\item
	We write $\rO^+(n) := \rO^+(1_{\bM_n})$. 
	It holds that $\rO^+\big(\begin{smallmatrix}0 & 1_{\bM_n} \\ 1_{\bM_n} & 0\end{smallmatrix}\big)\cong \rO^+(2n)$ by \cite[Theorem 5.3]{Bichon-DeRijdt-Vaes2006ergodic} (see also \cite[Corollary 2.5.4 (ii), Proposition 2.5.6]{Neshveyev-Tuset-book}). 
\end{itemize}

\begin{eg}\label{eg_tensor_faith_Sn+}
	For $n\in\bZ_{\geq 1}$, there is $\theta=(\cV_\theta,\Pi_\theta)\in\Rep\wc{{S_n^+}}$ that is tensorially faithful. 
	When $n\in [1,5]\cup[10,\infty)$, this is due to \cite[Theorem 4.11]{Brannan-Chirvasitu-Freslon2020topological} in view of the remark before \autoref{lem_tensor_faith}. 
	Suppose $n\geq 5$. Then, by \cite[Corollary 3.4]{Brannan-Chirvasitu-Freslon2020topological}, $S_n^+$ is topologically generated by $S_n$ and $S^+_{n-1}$. 
	Here, the case $n=5$ is due to \cite{Banica2021homogeneous}, which relies on the classification of the standard invariant at index $5$ \cite{Izumi-Morrison-Penneys-Peters-Snyder2015subfactors} (this is the only place in this paper where we need the classification result of the standard invariant). 
	Thus, by using \autoref{prop_tensor_faith}, the claim follows inductively from the cases of $S_4^+$ as above and $S_n$ by \autoref{eg_tensor_faith_finite}. 
\end{eg}

\begin{eg}\label{eg_tensor_faith_Un+}
	For $n\in\bZ_{\geq 2}$, consider the free unitary quantum group $\rU^+(n)$. 
	By \cite[Theorem 3.4]{Chirvasitu2020topological}, $\rU^+(2)$ is topologically generated by $\rU(2)$ and the Pontryagin dual of $\bZ^{\ast 2}$. 
	Since there is an injective group homomorphism from $\bZ^{\ast 2}$ to some finite dimensional unitary group (for example, it is well-known from the argument of the Banach--Tarski paradox that there is an injective group homomorphism $\bZ^{\ast 2}\to\SO(3)\subset\rU(3)$), we see from \autoref{eg_tensor_faith_group}, \autoref{eg_tensor_faith_Lie}, and \autoref{prop_tensor_faith} that there is $\theta \in \Rep \wc{{\rU^+(n)}}$ that is tensorially faithful. 
	For $n\in\bZ_{\geq 3}$, it is shown in \cite[Proposition 2.1]{Chirvasitu2020topological} that $\rU^+(n)$ is topologically generated by $\rU(n)$ and $\rU^+(n-1)$. 
	Thus, using \autoref{eg_tensor_faith_Lie} and \autoref{prop_tensor_faith}, we inductively see that there is $\theta \in \Rep \wc{{\rU^+(n)}}$ that is tensorially faithful. 
	
	Here, note that the results of Chirvasitu \cite{Chirvasitu2020topological} we used do not rely on the classification of the standard invariant. Indeed, \cite[Theorem 3.4]{Chirvasitu2020topological} was shown directly, and \cite[Proposition 2.1]{Chirvasitu2020topological} was shown by the folding trick based on the $\operatorname{GL}(n;\bC)$-equivariance of a $\rU(n)$-equivariant linear operator (cf.\ \cite[proof of Lemma 3.11]{Chirvasitu2015residually}). 
\end{eg}

\begin{eg}\label{eg_tensor_faith_On+}
	For $n\in\bZ_{\geq 2}$, consider the free orthogonal quantum group $\rO^+(n)$. 
	
	We write $\cO$ for the unital $*$-subalgebra $\cO(\rO^+(n))$ generated by $\{ s_{i,j}s_{k,l} \mid i,j,k,l\in\{1,\cdots,n\} \}$. Then, $\Delta_{\rO^+(n)}(\cO)\subset \cO\odot\cO$, and there is a compact quantum group called the \emph{quantum automorphism group} of $\bM_n$ equipped with the tracial state, denoted by $\Qut(\bM_n)$ in this paper, such that there is an injective $*$-bialgebra map $f\colon \cO(\Qut(\bM_n))\to \cO(\rO^+(n))$ whose image is $\cO$ by \cite[Corollary 4.1]{Banica1999symmetries}. 
	Moreover, there is an injective unital $*$-bialgebra map $g\colon \cO(\Qut(\bM_n))\cong \cO \to \cO(\rU^+(n))$ whose image is the unital $*$-subalgebra of $\cO(\rU^+(n))$ generated by $\{ u_{i,j}u_{k,l}^* \mid i,j,k,l\in\{1,\cdots,n\} \}$ by \cite[Theorem 12]{Gromada2022presentations}. 
	
	Then, it is routine to check that we have the following short exact sequence of $*$-bialgebra maps, 
	\begin{align*}
		\begin{tikzcd}[ampersand replacement=\&]
			\cO(\Qut(\bM_n)) \arrow[r, "f"] \& \cO(\rO^+(n)) \arrow[r,"r"] \& \cO(\wc{\bZ/2\bZ}) 
		\end{tikzcd} , 
	\end{align*}
	where we wrote $r\colon \cO(\rO^+(n))\to \cO(\wc{\bZ/2\bZ})=\bC[\bZ/2\bZ]$ for the unital $*$-homomorphism induced by $r(s_{i,j}) := \delta_{i,j} u$ for all $i,j\in\{1,\cdots,n\}$, using the unitary $u\in \bC[\bZ/2\bZ]$ such that $u^2=1$ and $\Delta_{\wc{\bZ/2\bZ}}(u) = u\otimes u$. 
	Thus, there are $\theta\in\Rep\wc{\Qut(\bM_n)}$ that is tensorially faithful by \autoref{prop_tensor_faith_ext} \ref{item_prop_tensor_faith_ext_sub} and \autoref{eg_tensor_faith_Un+} and thus $\vartheta\in\Rep\wc{\rO^+(n)}$ that is tensorially faithful by \autoref{prop_tensor_faith_ext} \ref{item_prop_tensor_faith_ext_split}. 
\end{eg}

\begin{eg}\label{eg_tensor_faith_OJn+}
	For $n\in\bZ_{\geq 1}$ and $\varsigma\in\{\pm1\}$, consider the free orthogonal quantum group $H^\varsigma := \rO^+\big(\begin{smallmatrix}0 & \varsigma 1_{\bM_n} \\ 1_{\bM_n} & 0\end{smallmatrix}\big)$. 
	We write $s^{\varsigma}_{i,j}$ for the generators $s_{i,j}\in\cO(\rO^+\big(\begin{smallmatrix}0 & \pm1_{\bM_n} \\ 1_{\bM_n} & 0\end{smallmatrix}\big))$ from the definition of $\rO^+(F)$ above. 
	Consider the involutive $*$-automorphism $f^\varsigma\in \Aut(C^u(H^\varsigma))$ 
	determined by 
	$f^\varsigma(s^{\varsigma}_{i,j}) = s^{\varsigma}_{i,j}$, $f^\varsigma(s^{\varsigma}_{n+i,n+j}) = s^{\varsigma}_{n+i,n+j}$, $f^\varsigma(s^{\varsigma}_{i,n+j}) = -s^{\varsigma}_{i,n+j}$, $f^\varsigma(s^{\varsigma}_{n+i,j}) = -s^{\varsigma}_{n+i,j}$ for $i,j\in\{1,\cdots,n\}$. 
	It is routine to check that $f^\varsigma$ is indeed well-defined and satisfies $\Delta^u_{ H^\varsigma } f^{\varsigma}= (f^\varsigma\otimes f^\varsigma)\Delta^u_{ H^\varsigma }$ and $f^{\varsigma} f^{\varsigma} = \id_{C^u(H^\varsigma)}$. Thus, $f^\varsigma$ restricts to a bijective $*$-bialgebra map on $\cO(H^\varsigma)\to \cO(H^\varsigma)$, which extends to the unital normal $*$-automorphism on $L^\infty(H^\varsigma)$ since $h_{H^\varsigma}f^\varsigma=h_{H^\varsigma}$ (cf.\ \autoref{rem_bialg_hom_surj} \ref{item_rem_bialg_hom_surj_3}). 
	
	Then, there is a compact quantum group $G^\varsigma := (L^\infty(H^\varsigma) \barrtimes_{f^\varsigma} \bZ/2\bZ , \Delta^{\varsigma})$, 
	where, by letting $u\in \bC[\bZ/2\bZ] \subset L^\infty(H^\varsigma) \barrtimes_{f^\varsigma} \bZ/2\bZ$ denote the unitary as in \autoref{eg_tensor_faith_On+}, 
	the comultiplication and the Haar state are defined by 
	$\Delta_{G^\varsigma} (x + yu) :=  \Delta_{H^\varsigma}(x) + \Delta_{H^\varsigma}(y)(u\otimes u)$ and 
	$h_{G^\varsigma} (x + yu) := h_{H^\varsigma}(x)$ for all $x,y\in L^\infty(H^\varsigma)$. 
	
	Then, it is routine to check that $G^{+}\cong G^{-}$ by sending $s^{\varsigma}_{i,j} u^k$ to $s^{-\varsigma}_{i,j} u^{k+1}$ for $i,j\in\{1,\cdots,2n\}$ and $k\in\{0,1\}$, 
	and that for each $\varsigma\in\{\pm1\}$, we have the following short exact sequence of $*$-bialgebra maps: 
	\begin{align*}
		\begin{tikzcd}[ampersand replacement=\&]
			\cO(H^\varsigma) \arrow[r, "\subset"] \& \cO(G^\varsigma) \arrow[r,"r"] \& \cO(\wc{\bZ/2\bZ}) 
		\end{tikzcd}, 
	\end{align*}
	where $r\colon \cO(G^\varsigma) \ni x+yu \mapsto \varepsilon_{H^\varsigma}(x)1 + \varepsilon_{H^\varsigma}(y)u \in \bC[\bZ/2\bZ] = \cO(\wc{\bZ/2\bZ})$ for all $x,y\in\cO(H^\varsigma)$.

	By \autoref{prop_tensor_faith_ext}, we see that there is $\theta\in \Rep\rO^+\bigl(\begin{smallmatrix}0 & -1_{\bM_n} \\ 1_{\bM_n} & 0\end{smallmatrix}\bigr)$ that is tensorially faithful if and only if there is $\vartheta\in \Rep\rO^+\bigl(\begin{smallmatrix}0 & 1_{\bM_n} \\ 1_{\bM_n} & 0\end{smallmatrix}\bigr)$ that is tensorially faithful, where the latter is the case by \autoref{eg_tensor_faith_On+} and $\rO^+\bigl(\begin{smallmatrix}0 & 1_{\bM_n} \\ 1_{\bM_n} & 0\end{smallmatrix}\bigr) \cong \rO^+(2n)$. 
\end{eg}

We recall the compact quantum groups $\SU_q(2)$ and $\SO_q(3)$. 
Let $q\in[-1,1]\setminus\{0\}$. 
There is a compact quantum group $\SU_q(2)$ such that $C^u(\SU_q(2))$ is the universal unital C*-algebra generated by $\qa$ and $\qc$ with the relations $\qa^*\qa+\qc^*\qc=1=\qa\qa^*+q^2\qc^*\qc$, $\qa\qc-q\qc\qa=0=\qa\qc^*-q\qc^*\qa$, and $\qc\qc^*=\qc^*\qc$, and that $\Delta_{\SU_q(2)}\colon C(\SU_q(2))\to C(\SU_q(2))^{\otimes 2}$ is the unique unital $*$-homomorphism with $\Delta_{\SU_q(2)}(\qa)=\qa\otimes \qa -q \qc^*\otimes \qc$ and $\Delta_{\SU_q(2)}(\qc) = \qc\otimes \qa +\qa^*\otimes \qc$. 
Then, we have that $C^u(\SU_q(2))=C(\SU_q(2))$ and that $\cO(\SU_q(2))$ is the unital $*$-subalgebra generated by $\qa$ and $\qc$. 

Also, there is a compact quantum group $\SO_q(3)$ such that $C^u(\SO_q(3))=C(\SO_q(3))$ is the unital C*-subalgebra of $C(\SU_q(2))$ generated by $\qc^2$, $\qc^*\qc$, $\qa\qc$, $\qa^*\qc$, and $\qa^2$ and that $\Delta_{\SO_q(3)} = \Delta_{\SU_q(2)}$ on $C(\SO_q(3))$. Then, $\cO(\SO_q(3))$ is the unital $*$-subalgebra of $\cO(\SU_q(2))$ generated by $\qc^2$, $\qc^*\qc$, $\qa\qc$, $\qa^*\qc$, and $\qa^2$, and $h_{\SO_q(3)} = h_{\SU_q(2)}$ on $C(\SO_q(3))$. 
Note that $\SO_q(3)\cong\SO_{-q}(3)$ (see e.g., \cite{Podles1995symmetries}). 


For a compact quantum group $G$ and $\pi\in\Rep G$, we write $\dim_{\Rep G}\pi$ for the intrinsic dimension. 
For each $k\in\bZ_{\geq 0}$, there is a unique $\pi = \pi_{\frac{k}{2}}\in \Irr\SU_q(2)$ such that $\dim_\bC\cH_{\pi} = k+1$. Also, $\Irr\SO_q(3) = \{ \pi_{\frac{k}{2}} \mid k\in 2\bZ_{\geq 0} \}\subset \Irr\SU_q(2)$. 
For each $\pi\in\Irr\SU_q(2)$, expressing $V_{\pi}=(u^{\pi}_{i,j})_{i,j}$ using an appropriate orthonormal basis of $\cH_{\pi}$, we have that 
$h_{\SU_q(2)}(u^{\varpi*}_{k,l} u^{\pi}_{i,j}) = \frac{\delta_{\pi,\varpi}\delta_{j,l}\delta_{i,k}}{\dim_{\Rep \SU_q(2)}\pi} |q|^{2i-1-\dim_\bC\cH_\pi}$ for all $\pi,\varpi\in\Irr\SU_q(2)$, $i,j\in\{1,\cdots,\dim_\bC\cH_\pi\}$, $k,l\in\{1,\cdots,\dim_\bC\cH_\varpi\}$ (see e.g., \cite[Section 5]{Tomatsu2008compact}). 
If $\pi,\varpi\in\Irr\SO_q(3)$, then the left hand side equals 
$h_{\SO_q(3)}(u^{\varpi*}_{k,l} u^{\pi}_{i,j})$. 

Finally, note that we have unitary monoidal equivalences $\Rep \rO^+(n)\simeq \Rep\SU_{\frac{\sqrt{n^2-4}-n}{2}}(2)$ for $n\in\bZ_{\geq 2}$, $\Rep \rO^+\bigl(\begin{smallmatrix}0 & -1_{\bM_n} \\ 1_{\bM_n} & 0\end{smallmatrix}\bigr) \simeq \Rep\SU_{n-\sqrt{n^2-1}}(2)$ for $n\in\bZ_{\geq 1}$ by \cite[Corollary 5.4]{Bichon-DeRijdt-Vaes2006ergodic}, and $\Rep S^+_n\simeq \Rep\SO_{\frac{\sqrt{k}-\sqrt{k-4}}{2}}(3)$ for $k\in\bZ_{\geq 4}$ by \cite[Theorem 4.7]{DeRijdt-VanderVennet2010actions} and \cite[Theorem 1.1]{Soltan2010quantumSO(3)}. 

\section{Outer actions on the Jiang--Su algebra}\label{sec_IU_JiangSu}

In this section, we show \autoref{main_JiangSu}. We begin with the following lemma. 
For a linear operator $\Theta\colon\cH\to\cH$ on a vector space $\cH$ and a polynomial $F(X)=\sum_{k=0}^{n}a_kX^k\in \bC[X]$, we define $F(\Theta) := \sum_{k=0}^{n}a_k\Theta^k \colon \cH\to \cH$, where we promise $\Theta^0:=\id_\cH$. 
Also, we promise $\cH^{\otimes 0}:=\bC$. 
We write $\bZ_{\geq 0}[X] := \{ \sum_{k=0}^{n}a_kX^k \mid n, a_0,a_1,\cdots,a_n\in\bZ_{\geq 0} \} \subset \bC[X]$ for the set of polynomials with non-negative integer coefficients. 

\begin{lem}\label{lem_quantum_channel_irred}
	Let $\cH,\cV$ be non-zero finite dimensional Hilbert spaces. For $V\in\cU(\cH\otimes\cV)$ and $n\in\bZ_{\geq 1}$, we inductively define $V^{(0)}:=1_{\cB(\cH)}$ and $V^{(n)}:=V_{12}V^{(n-1)}_{13}$, 
	where the legs $1$, $2$, and $3$ in the last expression are assigned to the tensorial components $\cB(\cH)$, the leftmost $\cB(\cV)$, and the remaining $\cB(\cV^{\otimes n-1})$, respectively. 
	We consider the unital completely positive linear map $\Theta_V\colon \cB(\cH)\ni x\mapsto (\id_{\cB(\cH)}\otimes\tr_{\cV})(V(x\otimes 1_{\cB(\cV)})V^*)\in \cB(\cH)$. 
	\begin{enumerate}[leftmargin=*,label=(\arabic*)]
		\item\label{item_lem_quantum_channel_irred_1}
		For $F(X)=\sum_{k=0}^{n}a_kX^k\in \bZ_{\geq 0}[X]\setminus\{0\}$ with $n, a_0,a_1,\cdots,a_n\in\bZ_{\geq 0}$, we set $\cV_F:=\bigoplus_{k=0}^{n} (\cV^{\otimes k})^{\oplus a_k}$ and $V_F:=\bigoplus_{k=0}^{n} (V^{(k)})^{\oplus a_k} \in \cB(\cH\otimes\cV_F)$. 
		Then, $\Theta_{V_F} = \frac{ F( \dim_\bC \cV \Theta_V ) }{\dim_\bC\cV_F} \colon \cB(\cH)\to\cB(\cH)$. 
		\item\label{item_lem_quantum_channel_irred_2}
		We write $\wt{\tr}_{\cH} := \tr_{\cH}(-)1_{\cB(\cH)}\colon \cB(\cH)\to\bC1_{\cB(\cH)}$. 
		Then, for $F(X) \in \bZ_{\geq 0}[X]\setminus\{0\}$, we have 
		$\wt{\tr}_{\cH}\circ\Theta_{V_F}=\wt{\tr}_{\cH}=\Theta_{V_F}\circ\wt{\tr}_{\cH}$ and 
		$\wt{\tr}_{\cH}\circ\wt{\tr}_{\cH}=\wt{\tr}_{\cH}$. 
		\item\label{item_lem_quantum_channel_irred_3}
		Suppose that the linear map $\cB(\cH)_*\ni\omega\mapsto (\omega\otimes\id_{\cB(\cV_F)})(V_F) \in \cB(\cV_F)$ is injective for some $F(X)\in \bZ_{\geq 0}[X]\setminus\{0\}$. 
		Then, the positive map $\Theta_V\colon \cB(\cH)\to \cB(\cH)$ is \emph{irreducible} in the sense that there is no projection $P\in\cB(\cH)$ with $P\not\in\bC \id_{\cH}$ such that $\Theta_V(P)\in P\cB(\cH)P$. 
	\end{enumerate}
\end{lem}

\begin{proof}
	It is straightforward to see \ref{item_lem_quantum_channel_irred_1} and \ref{item_lem_quantum_channel_irred_2}. 
	By \ref{item_lem_quantum_channel_irred_1}, if $\Theta_V$ is not irreducible, then $\Theta_{V_G}$ is not irreducible for all $G(X)=\sum_{k=0}^{m}b_k X^k$ with $m\in\bZ_{\geq 1}$ and $b_0,b_1,\cdots,b_m \in\bZ_{\geq 0}$. 
	Thus, to see \ref{item_lem_quantum_channel_irred_3}, it is enough to show that $\Theta_{V_F}$ is irreducible. 
	
	We suppose that there is a projection $P\in\cB(\cH)$ with $P\not\in\bC \id_{\cH}$ such that $\Theta_{V_F}(P)\in P\cB(\cH)P$ to derive a contradiction. 
	We have 
	\begin{align*}
		0 
		={}& 
		(1_{\cB(\cH)}-P)\bigl((\id_{\cB(\cH)}\otimes\tr_{\cV_F})(V_F(P\otimes 1_{\cB(\cV_F)})V_F^*)\bigr)(1_{\cB(\cH)}-P) 
		\\={}& 
		(\id_{\cB(\cH)}\otimes\tr_{\cV_F}) \bigl( \bigl( (1_{\cB(\cH)}-P)\otimes1_{\cB(\cV_F)} \bigr) (V_F(P\otimes 1_{\cB(\cV_F)})V_F^*) \bigl( (1_{\cB(\cH)}-P)\otimes1_{\cB(\cV_F)} \bigr) \bigr) 
	\end{align*}
	and thus $(P\otimes 1_{\cB(\cV_F)})V_F^*((1_{\cB(\cH)}-P)\otimes1_{\cB(\cV_F)}) =0$ by the faithfulness of $\tr_{\cV_F}$. 
	Since any $0\neq\omega\in\cB(\cH)_*$ cannot annihilate the subspace $\{ (\id_{\cB(\cH)}\otimes\upsilon)(V_F) \mid \upsilon \in\cB(\cV_F)_* \}$ by assumption, 
	it follows from $P\neq0\neq1_{\cB(\cH)}-P$ that there is $\upsilon\in \cB(\cV_F)_*$ with $(1_{\cB(\cH)}-P) \bigl((\id_{\cB(\cH)}\otimes\upsilon)(V_F)\bigr) P \neq0$, but this contradicts $(P\otimes 1_{\cB(\cV_F)})V_F^*((1_{\cB(\cH)}-P)\otimes1_{\cB(\cV_F)}) =0$. 
\end{proof}

Thanks to the Perron--Frobenius theorem for irreducible positive maps by Evans--H{\o}egh-Krohn \cite{Evans-HoeghKrohn1978spectral}, we have the following. 

\begin{cor}\label{cor_quantum_channel_irred}
	Let $\cH,\cV$ be non-zero finite dimensional Hilbert spaces and $V\in\cU(\cH\otimes\cV)$ such that the linear map $\cB(\cH)_*\ni\omega\mapsto (\omega\otimes\id_{\cB(\cV^{\otimes n})})(V^{(n)}) \in \cB(\cV^{\otimes n})$ is injective (or equivalently, the dimension of its image equals $\dim_\bC\cB(\cH)$) for all $n\in\bZ_{\geq 1}$. 
	Then, the map $\Theta_V\colon \cB(\cH)\ni X\mapsto (\id_{\cB(\cH)}\otimes\tr_{\cV})(V(X\otimes 1_{\cB(\cV)})V^*)\in \cB(\cH)$ satisfies the following. 
	\begin{enumerate}[leftmargin=*,label=(\arabic*)]
		\item\label{item_cor_quantum_channel_irred_1}
		The generalized eigenspace of $\Theta_V$ with the eigenvalue $1$ is $\bC1_{\cB(\cH)}$. In particular, the eigenspace of $\Theta_V$ with the eigenvalue $1$ is also $\bC1_{\cB(\cH)}$. 
		\item\label{item_cor_quantum_channel_irred_2}
		$\lambda_V:=\max\{ |\lambda| \mid \lambda\in\bC\setminus\{ 1 \}\text{ is an eigenvalue of $\Theta_V$} \} <1$, where we promise $\max\varnothing=-\infty$. 
		\item\label{item_cor_quantum_channel_irred_3}
		$\wt{\tr}_{\cH} = \tr_{\cH}(-)1_{\cB(\cH)}\colon \cB(\cH)\to\bC1_{\cB(\cH)}$ is the projection onto the eigenspace of $\Theta_V$ with eigenvalue $1$. 
		\item\label{item_cor_quantum_channel_irred_4}
		$\lim_{n\to\infty} \|\Theta_V^n - \wt{\tr}_{\cH}\|=0$. 
	\end{enumerate}
\end{cor}

\begin{proof}
	Clearly, $1_{\cB(\cH)}$ is an eigenvector of $\Theta_V$ with the eigenvalue $1$. 
	It follows from $\|\Theta_V\|=1$ and the positivity of $\Theta_V$ that $1$ is an eigenvalue of $\Theta_V$ with largest absolute value. 
	Moreover, Stinespring's dilation applied to the unital completely positive map $\Theta_V$ shows that $\Theta_V(x^*x)\geq \Theta_V(x)^*\Theta_V(x)$ for all $x\in\cB(\cH)$ (see e.g., \cite[Proposition 1.5.7 (1)]{Brown-Ozawa-book}). In particular, $\Theta_V$ is a Schwartz map in the sense of \cite[(4.1)]{Evans-HoeghKrohn1978spectral}. 
	Then, \ref{item_cor_quantum_channel_irred_2} and the one-dimensionality of the eigenspace of $\Theta_V$ with the eigenvalue $1$ are consequences of \cite[Theorem 2.3, Theorem 2.4, Theorem 4.2]{Evans-HoeghKrohn1978spectral} since $\Theta_{V^{(n)}}=\Theta_{V}^n$ is irreducible for all $n\in\bZ_{\geq 1}$ by \autoref{lem_quantum_channel_irred} \ref{item_lem_quantum_channel_irred_3}. 
	
	To see \ref{item_cor_quantum_channel_irred_1}, it remains to see that any generalized eigenvector of $\Theta_V$ with eigenvalue $1$ is an eigenvector. 
	If not, there is $x\in \Ker(\Theta_V - \id_{\cB(\cH)})^n \setminus \Ker(\Theta_V - \id_{\cB(\cH)})^{n-1}$ for some $n\in\bZ_{\geq 2}$. 
	Since $(\Theta_V - \id_{\cB(\cH)})^{n-1}(x) \in \Ker (\Theta_V - \id_{\cB(\cH)}) = \bC1_{\cB(\cH)}$, there is $c\in\bC$ such that $(\Theta_V - \id_{\cB(\cH)})^{n-1}(x)=c1_{\cB(\cH)}$. Note that $c\neq 0$ by $x\not\in \Ker(\Theta_V - \id_{\cB(\cH)})^{n-1}$. 
	However, since $\tr_{\cH}\Theta_V=\tr_{\cH}$, we have 
	\begin{align*}
		0 ={}& \tr_{\cH}\Theta_V(\Theta_V - \id_{\cB(\cH)})^{n-2}(x) - \tr_{\cH}(\Theta_V - \id_{\cB(\cH)})^{n-2}(x) 
		\\={}& \tr_{\cH}(\Theta_V - \id_{\cB(\cH)})^{n-1}(x) = \tr_{\cH}(c1_{\cB(\cH)}) = c \neq 0, 
	\end{align*}
	which is a contradiction. 
	Hence, \ref{item_cor_quantum_channel_irred_1} holds. 
	
	By \ref{item_cor_quantum_channel_irred_1} and \ref{item_cor_quantum_channel_irred_2}, there is the polynomial $F(X)\in\bC[X]$ such that $(X-1)F(X)$ is the characteristic polynomial of $\Theta_V$, and we have $F(1)\neq 0$. Thus, $F(1)^{-1}F(\Theta_V)$ is the projection from $\cB(\cH)$ onto the (generalized) eigenspace of $\Theta_V$ with the eigenvalue $1$. For any $x\in \cB(\cH)$, since $(F(1)^{-1}F(\Theta_V))(x)$ is of the form $c1_{\cB(\cH)}$ for some $c\in\bC$, using $\wt{\tr}_{\cH} = \wt{\tr}_{\cH}\Theta_V$ we see that 
	\begin{align*}
		\wt{\tr}_{\cH}(x) = \wt{\tr}_\cH \circ(F(1)^{-1} F(\Theta_V)) (x) = \wt{\tr}_\cH(c1_{\cB(\cH)}) = c1_{\cB(\cH)} = (F(1)^{-1} F(\Theta_V)) (x) , 
	\end{align*}
	which implies \ref{item_cor_quantum_channel_irred_3}. 
	
	Finally, by \ref{item_cor_quantum_channel_irred_2} and \ref{item_cor_quantum_channel_irred_3}, we have 
	\begin{align*}
		&
		\max\{ |\lambda| \mid \lambda\in\bC\text{ is an eigenvalue of $\Theta_V-\wt{\tr}_{\cH}$} \} = \max\{ \lambda_V, 0 \} <1. 
	\end{align*}
	Since the leftmost expression equals the spectral radius of $\Theta_V-\wt{\tr}_{\cH}$, which coincides with $\lim_{n\to\infty}\| (\Theta_V-\wt{\tr}_{\cH})^n \|^{\frac{1}{n}}$, we have $\lim_{n\to\infty}\| (\Theta_V-\wt{\tr}_{\cH})^n \|=0$. 
	Combined with $\Theta_V\wt{\tr}_{\cH}=\wt{\tr}_{\cH}=\wt{\tr}_{\cH}\Theta_V$, we see \ref{item_cor_quantum_channel_irred_4}. 
\end{proof}

\begin{df}\label{def_tr_pres_action}
	Let $G$ be a compact quantum group, $A$ be a C*-algebra, and $M$ be a von Neumann algebra. 
	Let $(f,\gamma)$ be a pair either of a state $f$ on $A$ and a left $\Gamma$-action $\gamma$ on $A$ or of a normal state $f$ on $M$ and a left $\Gamma$-action $\gamma$ on $M$. 
	Then, we say that the left $\Gamma$-action $\gamma$ is \emph{$f$-preserving} if $(\id_{c_0(\Gamma)}\otimes f)\gamma = 1_{\ell^\infty(\Gamma)}\otimes f$. 
	When $A$ and $M$ are unital and have unique tracial states and $f$ is the tracial state, then we say that $\gamma$ is trace-preserving. 
\end{df}

\begin{lem}\label{lem_tr_pres_action}
	Let $G$ be a compact quantum group, $A$ be a unital C*-algebra, $f$ be a faithful tracial state on $A$, and $\gamma$ be a left $\Gamma$-action on $A$ that is $f$-preserving. 
	Then, when we write $M$ for the von Neumann completion of $A$ with respect to the GNS construction of $f$ and extend $f$ to the normal tracial state on $M$, then $\gamma$ uniquely extends to a left $\Gamma$-action on $M$ that is $f$-preserving. 
\end{lem}

\begin{proof}
	For each $\pi\in\Irr G$, consider the $*$-homomorphism $\gamma_{\pi} := (\Pi_\pi\otimes\id_A) \gamma\colon A\to \cB(\cH_\pi)\otimes A$. Note that $\prod_{\pi\in\Irr G}\gamma_{\pi} = \prod_{\pi\in\Irr G}(\Pi_\pi\otimes\id_A) \gamma$. 
	There is an isometry $U_{\pi}\in\cB\bigl(\cH_\pi\otimes L^2( A,f ) \bigr)$ such that for all $a\in A$ and $\xi\in\cH_\pi$, 
	\begin{align*}
		U_{\pi}(\xi\otimes a) 
		={}& 
		\gamma_{\pi}(a) (\xi\otimes1_A) 
		\in 
		\cH_\pi\otimes L^2( A, f ) 
	\end{align*}
	where $a$ and $1_A$ are regarded as elements in $L^2(A,f)$. 
	Indeed, fixing an orthonormal basis $(e_i)_{i=1}^{\dim_\bC\cH_\pi}$ of $\cH_\pi$ and setting $a_{i,j}\in A$ for the $(i,j)$-th component of $\gamma_{\pi}(a)\in \cB(\cH_\pi)\otimes A$ with respect to $(e_i)_i$, we can check that, for all $a,b\in A$, $k,l\in\{1,\cdots,\dim_\bC\cH_\pi\}$, 
	\begin{align*}
		&
		\bra \gamma_{\pi}(b)(e_l\otimes 1_{A}), \gamma_{\pi}(a)(e_k\otimes 1_{A}) \ket 
		= 
		\bra e_l, (\id_{\cB(\cH_\pi)} \otimes f) \gamma_{\pi}(b^*a) e_k \ket 
		\\={}&
		\bra e_l, f(b^*a)1_{\cB(\cH_\pi)} e_k \ket 
		= 
		\delta_{k,l}  f ( b^*a ) 
		= 
		\bra e_l\otimes b, e_k\otimes a \ket . 
	\end{align*}
	It follows from the continuity of the left $\Gamma$-action $\gamma$ by \autoref{lem_coame_conti_action} that the image of $U_\pi \in \cB(\cH_\pi\otimes L^2(A,f))$ is norm-dense. Thus, $U_\pi$ is a unitary. 
	
	Since $(\prod_{\pi\in\Irr G}\Pi_\pi\otimes\id_A)\gamma(a)=(\Ad U_\pi (1_{\cB(\cH_\pi)}\otimes a))_{\pi\in\Irr G}$ for all $a\in A$, and we see that $\gamma$ uniquely extends to the unital normal $*$-homomorphism $\bar{\gamma} := \prod_{\pi\in\Irr G} \Ad U_\pi (1_{\cB(\cH_\pi)}\otimes (-)) \colon M\to \prod_{\pi\in\Irr G}(\cB(\cH_\pi) \otimes M)$. 
	It is easy to see that $\bar{\gamma}$ is a left $\Gamma$-action on $M:=A''\subset\cB(L^2(A,f))$ and that we have $1_{\ell^\infty(\Gamma)}\otimes f=(\id_{\ell^\infty(\Gamma)}\otimes f)\bar{\gamma}$. 
\end{proof}

\begin{thm}\label{thm_IU_JiangSu}
	Let $G$ be a compact quantum group with $\theta\in\Rep\Gamma$ that is tensorially faithful. 
	Then, there is a trace-preserving outer left $\Gamma^{\op}$-action on the Jiang--Su algebra $\cZ$ that extends to a trace-preserving outer left $\Gamma^{\op}$-action on the von Neumann completion of $\cZ$ with respect to the unique faithful tracial state on $\cZ$. 
\end{thm}

\begin{proof}
	We divide the proof into several steps. 
	For each $\varphi\in\Rep\Gamma$, note that 
	$\Ad V_{\varphi}(1_{\ell^\infty(\Gamma)}\otimes (-))\colon \cB(\cV_\varphi)\otimes A \to \cM(c_0(\Gamma)\otimes \cB(\cV_\varphi) \otimes A)$ is a well-defined left $\Gamma^{\op}$-action on $\cB(\cV_\varphi)\otimes A$ for any C*-algebra $A$ since $V^{G*}_{12} (V_\varphi)_{23} V^G_{12} = (V_\varphi)_{13}(V_\varphi)_{23}$. 
	
	\medskip\paragraph{\bf Step 1}
	
	Firstly, we set up symbols. 
	Take $\theta\in\Rep\Gamma$ that is tensorially faithful. 
	We may assume that $\mathbbm{1}\leq\theta$ by considering $\theta\oplus\mathbbm{1}$ instead if necessary. 
	Then, for all $r\in\bZ_{\geq 1}$, we have $\theta\leq \theta^{\otimes r}$. 
	
	For $F(X)=\sum_{k=0}^{n}a_k X^k \in\bZ_{\geq 0}[X]$, we let $F(\theta) := \bigoplus_{k=0}^{n} (\theta^{\otimes k})^{\oplus a_k}$, where we promise $\theta^{\otimes 0}:=\mathbbm{1}$. 
	Here, we emphasize that when $F(X)$ is given by a product of polynomials, $F(\theta)$ is defined using the expanded form of $F$. For example, when $F(X):=(X+1)(X+2)$, then $F(\theta):=\theta^{\otimes 2}\oplus \theta^{\oplus 3}\oplus \mathbbm{1}^{\oplus 2}$ rather than $(\theta\oplus\mathbbm{1})\otimes(\theta\oplus\mathbbm{1}^{\oplus 2})$. 
	
	For $n\in\bZ_{\geq 0}$, we define $F_{n,0},F_{n,1},P,Q\in\bZ_{\geq 0}[X]$ by 
	\begin{align*}
		F_{n,0}(X) :={}& (X+2)^{2^n}+1, 
		\qquad
		F_{n,1}(X) := (X+2)^{2^n}, 
		\\
		P(X) :={}& X+2, 
		\qquad\text{and}\qquad 
		Q(X) := X+1. 
	\end{align*}
	For all $n\in\bZ_{\geq 1}$, these polynomials satisfy 
	\begin{align*}
		F_{n,0}(X) ={}& 1+ (X+2)\prod_{k=0}^{n-1} (X+2)^{2^{k}} = 
		\biggl( \prod_{k=0}^{n-1} F_{k,1}(X) \biggr)P(X) + 1, 
		\\
		F_{n,1}(X) ={}& 1+ ((X+2)-1)\prod_{k=0}^{n-1} ((X+2)^{2^{k}} +1) 
		= 
		\biggl( \prod_{k=0}^{n-1} F_{k,0}(X) \biggr) Q(X) + 1, 
	\end{align*}
	and, since $(X+2)^{2^n} - 2 \in \bZ_{\geq 0}[X]$, 
	\begin{align*}
		F_{n,\frac{1}{2}}(X) :={}&
		F_{n,0}(X) F_{n,1}(X) 
		- F_{n,0}(X) - F_{n,1}(X)
		\\={}& 
		((X+2)^{2^n}+1) ((X+2)^{2^n} - 2) + 1 
		\in{} \bZ_{\geq 0}[X]. 
	\end{align*}
	
	For $m,n\in\bZ_{\geq 0}$ with $n\leq m$, we define the following objects in $\Rep \Gamma$, 
	\begin{align*}
		\phi_n := (\theta\oplus \mathbbm{1}^{\oplus 2})^{\otimes 2^n } , 
		\qquad 
		\phi_{[n,m]}
		:={}& \phi_n \otimes \phi_{n+1} \otimes \cdots \otimes \phi_m , 
		\\
		\psi_n := \phi_n \oplus \mathbbm{1} , 
		\qquad\text{and}\qquad
		\psi_{[n,m]}
		:={}& \psi_n \otimes \psi_{n+1} \otimes \cdots \otimes \psi_m . 
	\end{align*}
	Also, for $m,n\in\bZ_{\geq 1}$ with $n\leq m$, we define the following objects in $\Rep\Gamma$, 
	\begin{align*}
		&
		\vartheta_{n,0} := \mathbbm{1} \otimes \psi_{n} ( \cong F_{n,0}(\theta) ) , 
		\qquad
		\vartheta_{n,\frac{1}{2}} := F_{n,\frac{1}{2}}(\theta) , 
		\qquad
		\vartheta_{n,1} := \mathbbm{1} \otimes \phi_{n} ( \cong F_{n,1}(\theta) ), 
		\\
		&
		\vartheta_{n} := \vartheta_{n,0} \oplus \vartheta_{n,\frac{1}{2}} \oplus \vartheta_{n,1} , 
		\qquad\text{and}\qquad
		\vartheta_{[n,m]} := \vartheta_n \otimes \vartheta_{n+1} \otimes \cdots \otimes \vartheta_m . 
	\end{align*}
	For $n\in\bZ_{\geq 1}$, we take the following unitary morphisms in $\Rep\Gamma$, 
	\begin{align*}
		u_{n,0}\colon \phi_{[1,n]}\otimes \psi_{[1,n]} &\to \vartheta_{[1,n]} , 
		\qquad
		u_{n,1}\colon \psi_{[1,n]}\otimes \phi_{[1,n]} \to \vartheta_{[1,n]} , 
	\end{align*}
	\begin{align*}
		v_{n,0}\colon 
		\mathbbm{1} \oplus \psi_{[1,n]} \otimes Q(\theta) \otimes \psi_0 
		&\to \phi_{n+1}, 
		\qquad 
		v_{n,1}\colon 
		\phi_{[1,n]} \otimes P(\theta) \otimes \phi_0 \oplus \mathbbm{1} 
		\to \psi_{n+1}, 
	\end{align*}
	\begin{align*}
		w_{n,0} \colon 
		\vartheta_{n,\frac{1}{2}} \oplus \vartheta_{n,1} 
		&\to 
		Q(\theta)\otimes\psi_{[0,n]} , 
		\qquad
		w_{n,1} \colon 
		\vartheta_{n,0} \oplus \vartheta_{n,\frac{1}{2}} 
		\to 
		P(\theta)\otimes\phi_{[0,n]} , 
	\end{align*}
	\begin{align*}
		x_{n,0}\colon \psi_{[1,n]} \otimes \mathbbm{1} &\to \mathbbm{1} \otimes \psi_{[1,n]}, 
		\qquad
		x_{n,1}\colon \phi_{[1,n]} \otimes \mathbbm{1} \to \mathbbm{1} \otimes \phi_{[1,n]}. 
	\end{align*}
	Since the unitary group of a finite dimensional C*-algebra is path-connected, there is a norm-continuous path of unitaries 
	\begin{align*}
		&
		U_{n}\colon [0,1]\to 
		\cU\bigl( \Hom_{\Rep\Gamma}( \vartheta_{[1,n+1]}, \vartheta_{[1,n+1]} ) \bigr) , 
	\end{align*}
	such that 
	\begin{align*}
		U_{n}(0) \colon{}& 
		\vartheta_{[1,n+1]} = 
		\vartheta_{[1,n]} \otimes 
		( \vartheta_{n+1,0} \oplus \vartheta_{n+1,\frac{1}{2}} \oplus \vartheta_{n+1,1} ) 
		\\&\xto{u_{n,0}^*\otimes (\id_{\vartheta_{n+1,0}}\oplus w_{n+1,0})} 
		\phi_{[1,n]} \otimes \psi_{[1,n]} \otimes 
		( \mathbbm{1}\otimes\psi_{n+1} \oplus Q(\theta)\otimes \psi_{[0,n+1]} ) 
		\\&\xto{\text{distr.}} 
		( \phi_{[1,n]} \otimes \psi_{[1,n]} \otimes \mathbbm{1} \otimes \psi_{n+1} ) 
		\oplus ( \phi_{[1,n]} \otimes \psi_{[1,n]} \otimes Q(\theta) \otimes \psi_{[0,n+1]} ) 
		\\&
		\xrightarrow[\oplus \id_{\phi_{[1,n]}} \otimes v_{n,0}^+ \otimes \id_{ \psi_{[1,n+1]}} ]{ \id_{\phi_{[1,n]}}\otimes ( ( v_{n,0}^-\otimes \id_{\psi_{[1,n]}}) x_{n,0} ) \otimes \id_{ \psi_{n+1}} } 
		\phi_{[1,n]} \otimes \phi_{n+1} \otimes 
		\psi_{[1,n]} \otimes \psi_{n+1} 
		\\& 
		\xto{u_{n+1,0}}
		\vartheta_{[1,n+1]} , 
	\end{align*}
	and 
	\begin{align*}
		U_{n}(1) \colon{}& 
		\vartheta_{[1,n+1]} = 
		\vartheta_{[1,n]} \otimes 
		( \vartheta_{n+1,0} \oplus \vartheta_{n+1,\frac{1}{2}} \oplus \vartheta_{n+1,1} ) 
		\\&\xto{u_{n,1}^*\otimes (w_{n+1,1}\oplus \id_{\vartheta_{n+1,1}})} 
		\psi_{[1,n]} \otimes \phi_{[1,n]} \otimes 
		( P(\theta)\otimes \phi_{[0,n+1]} \oplus \mathbbm{1}\otimes\phi_{n+1} ) 
		\\&\xto{\text{distr.}} 
		( \psi_{[1,n]} \otimes \phi_{[1,n]} \otimes P(\theta) \otimes \phi_{[0,n+1]} ) 
		\oplus ( \psi_{[1,n]} \otimes \phi_{[1,n]} \otimes \mathbbm{1} \otimes \phi_{n+1} ) 
		\\&
		\xrightarrow[\oplus \id_{\psi_{[1,n]}}\otimes ( ( v_{n,1}^- \otimes \id_{\phi_{[1,n]}}) x_{n,1} ) \otimes \id_{ \phi_{n+1}} ]{ \id_{\psi_{[1,n]}} \otimes v_{n,1}^+ \otimes \id_{ \phi_{[1,n+1]}} } 
		\psi_{[1,n]} \otimes \psi_{n+1} \otimes 
		\phi_{[1,n]} \otimes \phi_{n+1} 
		\\& 
		\xto{u_{n+1,1}}
		\vartheta_{[1,n+1]} , 
	\end{align*}
	where $v_{n,0}^{+}$, $v_{n,1}^{+}$, and $v_{n,i}^{-}$ for $i=0,1$ are the restrictions of $v_{n,0}$, $v_{n,1}$, and $v_{n,i}$ to the direct summands $\psi_{[1,n]}\otimes Q(\theta)\otimes\psi_0$, $\phi_{[1,n]}\otimes P(\theta)\otimes\phi_0$, and $\mathbbm{1}$, respectively, and the maps denoted by $\text{distr.}$ indicate the `distributive' unitary isomorphisms witnessing the compatibility of the monoidal structure and the direct sum of $\Rep \Gamma$. 
	
	\medskip\paragraph{\bf Step 2}
	
	We are going to construct the C*-algebra $A_\infty$ that is $*$-isomorphic to $\cZ$ and the left $\Gamma^{\op}$-action $\gamma$ on $A_\infty$. For $n\in\bZ_{\geq 1}$, we let 
	\begin{align*}
		A_n 
		:= 
		\biggl\{ f\in \cB(\cV_{\vartheta_{[1,n]}}) \otimes C([0,1]) 
		\ \bigg|\ 
		\begin{array}{l}
			\Ad u_{n,0}^* f(0) \in \cB(\cV_{\phi_{[1,n]}})\otimes \id_{\cV_{\psi_{[1,n]}}} , 
			\\
			\Ad u_{n,1}^* f(1) \in \cB(\cV_{\psi_{[1,n]}})\otimes \id_{\cV_{\phi_{[1,n]}}}
		\end{array}
		\biggr\}, 
	\end{align*}
	which is a $\Gamma^{\op}$-equivariant C*-subalgebra of $\cB(\cV_{\vartheta_{[1,n]}}) \otimes C([0,1])$ with the left $\Gamma^{\op}$-action $\gamma_n := \Ad (V_{\vartheta_{[1,n]}} \otimes 1_{C([0,1])})(1_{\ell^\infty(\Gamma)}\otimes (-))$. 
	We have the injective unital $\Gamma^{\op}$-$*$-homomorphism 
	\begin{align*}
		\iota_{n,n+1} \colon \cB(\cV_{\vartheta_{[1,n]}}) \otimes C([0,1]) &\to \cB(\cV_{\vartheta_{[1,n+1]}}) \otimes C([0,1]) 
	\end{align*}
	such that $\iota_{n,n+1}(f)$ sends each $t\in[0,1]$ to 
	\begin{align*}
		&
		\Ad U_n(t) \left( 
		f(\tfrac{t}{2}) \otimes \id_{\vartheta_{n+1,0}} \oplus f(\tfrac{1}{2}) \otimes \id_{\vartheta_{n+1,\frac{1}{2}}} \oplus f(\tfrac{1+t}{2}) \otimes \id_{\vartheta_{n+1,1}} 
		\right), 
	\end{align*}
	for any $f\in \cB(\cV_{\vartheta_{[1,n]}}) \otimes C([0,1])$. 
	We set $\iota_{n,m} := \iota_{m-1,m}\circ\iota_{m-2,m-1}\circ\cdots\circ\iota_{n,n+1}$ 
	for $m,n\in \bZ_{\geq 1}$ with $n\leq m$. Note that $\iota_{n,n}=\id_{\cB(\cV_{\vartheta_{[1,n]}})\otimes C([0,1])}$. 
	We can check that $\iota_{n,m}$ restricts to a unital $\Gamma^{\op}$-$*$-homomorphism $A_{n}\to A_{m}$, for which we still write $\iota_{n,m}$, by using the commutativity of the upper and the lower rectangles in the following diagrams for $n\in\bZ_{\geq 1}$, 
	\begin{align*}
		\begin{tikzcd}[column sep=14em,row sep=4em,ampersand replacement=\&]
			\cB(\cV_{\phi_{[1,n]}}) 
			\arrow[d,"(-)\otimes\id_{\psi_{[1,n]}}"'{name=D1}]
			\arrow[r,"\Ad (\id_{\phi_{[1,n]}}\otimes v_{n,0}^{-}) ((-)\otimes \id_{\mathbbm{1}})"] \& 
			\cB(\cV_{\phi_{[1,n+1]}}) 
			\arrow[d,"(-)\otimes\id_{\psi_{[1,n+1]}}"{name=D2}]
			\\ 
			\cB(\cV_{\phi_{[1,n]}}\otimes\cV_{\psi_{[1,n]}})
			\arrow[r,bend left=6,"\bigl( \Ad \bigl( \id_{\phi_{[1,n]}}\otimes ( ( v_{n,0}^{-}\otimes \id_{\psi_{[1,n]}}) x_{n,0} ) \bigr) ((-) \otimes \id_{\mathbbm{1}}) \bigr) \otimes \id_{\psi_{n+1}}"]
			\arrow[r,bend right=6,"\bigl( \Ad (\id_{\phi_{[1,n]}}\otimes v_{n,0}^{+}) ((-) \otimes \id_{Q(\theta)\otimes \psi_0}) \bigr) \otimes \id_{\psi_{[1,n+1]}}"'] \& 
			\cB(\cV_{\phi_{[1,n+1]}}\otimes\cV_{\psi_{[1,n+1]}}) 
			\\ 
			\cB(\cV_{\phi_{[1,n]}\otimes\psi_{[1,n]}}) 
			\arrow[u,"||"{name=U1}]
			\arrow[r,"\Ad (\id_{\phi_{[1,n]}}\otimes v_{n,0}^{+}) ((-) \otimes \id_{Q(\theta)\otimes \psi_0})"'] \& 
			\cB(\cV_{\phi_{[1,n+1]}}) \ar[u,"(-)\otimes\id_{\psi_{[1,n+1]}}"'{name=U2}]
			\arrow[from=D1, to=D2, phantom, "\circlearrowleft" above]
			\arrow[from=U1, to=U2, phantom, "\circlearrowleft" below]
		\end{tikzcd},
	\end{align*}
	\begin{align*}
		\begin{tikzcd}[column sep=14em,row sep=4em,ampersand replacement=\&]
			\cB(\cV_{\psi_{[1,n]}\otimes\phi_{[1,n]}}) 
			\arrow[d,"||"'{name=D3}]
			\arrow[r,"\Ad (\id_{\psi_{[1,n]}}\otimes v_{n,1}^{+}) ((-) \otimes \id_{P(\theta)\otimes \phi_0})"] \& 
			\cB(\cV_{\psi_{[1,n+1]}}) \ar[d,"(-)\otimes\id_{\phi_{[1,n+1]}}"{name=D4}]
			\\ 
			\cB(\cV_{\psi_{[1,n]}}\otimes\cV_{\phi_{[1,n]}}) 
			\arrow[r,bend left=6,"\bigl( \Ad (\id_{\psi_{[1,n]}}\otimes v_{n,1}^{+}) ((-) \otimes \id_{P(\theta)\otimes \phi_0}) \bigr) \otimes \id_{\phi_{[1,n+1]}}"] \arrow[r,bend right=6,"\bigl( \Ad \bigl( \id_{\psi_{[1,n]}}\otimes ( ( v_{n,1}^{-}\otimes \id_{\phi_{[1,n]}}) x_{n,1} ) \bigr) ((-) \otimes \id_{\mathbbm{1}}) \bigr) \otimes \id_{\phi_{n+1}}"']
			\& 
			\cB(\cV_{\psi_{[1,n+1]}}\otimes\cV_{\phi_{[1,n+1]}}) 
			\\ 
			\cB(\cV_{\psi_{[1,n]}}) 
			\arrow[u,"(-)\otimes\id_{\phi_{[1,n]}}"{name=U3}]
			\arrow[r,"\Ad (\id_{\psi_{[1,n]}}\otimes v_{n,1}^{-}) ((-)\otimes \id_{\mathbbm{1}})"'] \& 
			\cB(\cV_{\psi_{[1,n+1]}}) 
			\arrow[u,"(-)\otimes\id_{\phi_{[1,n+1]}}"'{name=U4}]
			\arrow[from=D3, to=D4, phantom, "\circlearrowleft" above]
			\arrow[from=U3, to=U4, phantom, "\circlearrowleft" below]
		\end{tikzcd}. 
	\end{align*}
	
	For $n\in\bZ_{\geq 1}$, note that $\dim_\bC\cV_{\psi_{[1,n]}} = \prod_{k=1}^{n} ( (\dim_\bC \cV_\theta +2)^{2^{k}} +1 ) $ is coprime to $\dim_\bC\cV_{\phi_{[1,n]}}$, which is a power of $\dim_\bC \cV_\theta +2$. 
	It follows from \cite[Theorem 2]{Jiang-Su1999simple} that $A_{\infty}:=\lim_{n\to\infty}A_n$ is $*$-isomorphic to the Jiang--Su algebra $\cZ$. 
	In fact, our construction non-equivariantly coincides with a specific case of the original construction of $\cZ$ in \cite[Section 2]{Jiang-Su1999simple}. 
	We write $\iota_{n,\infty}\colon A_n\to A_\infty$ for the injective unital $*$-homomorphism induced by the inductive limit. 
	Since each $\iota_{n,n+1}$ is a $\Gamma^{\op}$-$*$-homomorphism, we have the well-defined left $\Gamma^{\op}$-action $\gamma$ on $A_\infty$ such that $\iota_{n,\infty}$ is a $\Gamma^{\op}$-$*$-homomorphism for all $n\in\bZ_{\geq 1}$. 
	
	\medskip\paragraph{\bf Step 3}
	
	To see the outerness of $\gamma$, we extend $\gamma$ to the von Neumann completion $M$ of $A_\infty$ with respect to the faithful tracial state. 
	For $a=(a_n)_n \in \prod_{n=1}^{\infty} \{ 0,1 \}$, we set 
	$\Sigma(a):=\sum_{n=1}^{\infty} a_n 2^{-n} \in [0,1]$ and 
	\begin{align*}
		&
		l(a) := \sup( \{ n\in\bZ_{\geq 1} \mid a_n= 1 \}\cup\{0\} )\in \bZ_{\geq 0}\cup\{\infty\} . 
	\end{align*}
	We define the countable subset 
	\begin{align*}
		&
		J := \biggl\{ a=(a_n)_{n=1}^{\infty} \in \prod_{n=1}^{\infty} \{0,1\} \ \bigg|\  
		1\leq l(a)<\infty
		\biggr\} \subset \prod_{n=1}^{\infty} \{0,1\} , 
	\end{align*}
	and, for $n\in\bZ_{\geq 1}$, the measure $\mu_n$ on $[0,1]$ by 
	\begin{align*}
		&
		\mu_n := 
		\sum_{a=(a_k)_k\in J} 
		\frac{ \dim_\bC \cV_{\vartheta_{n+l(a),\frac{1}{2}}} }{ \dim_\bC \cV_{\vartheta_{n+l(a)}} } 
		\prod_{k=1}^{l(a)-1} 
		\frac{ \dim_\bC \cV_{\vartheta_{n+k,a_k}} }{ \dim_\bC \cV_{\vartheta_{n+k}} } 
		\delta_{ \Sigma (a) } , 
	\end{align*}
	where $\delta_t$ for each $t\in[0,1]$ denotes the Dirac probability measure on $[0,1]$ concentrated at $t$, i.e., $\delta_t(\{t\})=1$. 
	Note that $\mu_n([0,1])=\mu_n(\Sigma(J))$ and in particular $\mu_n(\{0\})=\mu_n(\{1\})=0$. 
	We have $\mu_n([0,1])=1$ because inductively on $l\in\bZ_{\geq 1}$, 
	\begin{align*}
		&
		1- \mu_n (\{ a\in J \mid l(a)\leq l \}) 
		= 
		\prod_{k=1}^{l} \biggl( 1 - \frac{ \dim_\bC \cV_{\vartheta_{n+k,\frac{1}{2}}} }{ \dim_\bC \cV_{\vartheta_{n+k}} } \biggr) 
		\\={}& 
		\prod_{k=1}^{l} \frac{ F_{n+k,0}(\dim_\bC\cV_{\theta}) + F_{n+k,1}(\dim_\bC\cV_{\theta}) }{ F_{n+k,0}(\dim_\bC\cV_{\theta}) F_{n+k,1}(\dim_\bC\cV_{\theta}) } \xto{l\to\infty} 0. 
	\end{align*}
	Note that $\int_{[0,1]}(-) d\mu_n$ is a well-defined state on $C([0,1])$. 
	For $n\in\bZ_{\geq 1}$ and $a=(a_k)_k\in J$, since $l(\Sigma^{-1}(\frac{\Sigma(a)}{2})) = l(\Sigma^{-1}(\frac{1+\Sigma(a)}{2})) = 1+l(a)$,  
	\begin{align}\label{eq_prf_thm_IU_JiangSu_1}
		&
		\begin{aligned}
			&
			\biggl( \frac{ \dim_\bC \cV_{\vartheta_{n+1,0}} + \dim_\bC \cV_{\vartheta_{n+1,1}} }{ \dim_\bC \cV_{\vartheta_{n+1}} } \biggr) \mu_{n+1}(\{ \Sigma(a) \}) 
			\\={}& 
			\sum_{a_0\in\{0,1\}} 
			\frac{ \dim_\bC \cV_{\vartheta_{n+1+l(a),\frac{1}{2}}} }{ \dim_\bC \cV_{\vartheta_{n+1+l(a)}} } 
			\prod_{k=0}^{l(a)-1} 
			\frac{ \dim_\bC \cV_{\vartheta_{n+1+k,a_k}} }{ \dim_\bC \cV_{\vartheta_{n+1+k}} } 
			= 
			\mu_n(\{ \tfrac{\Sigma(a)}{2}, \tfrac{1+\Sigma(a)}{2} \}) 
			\\={}& 
			\biggl( 1 + \frac{ \dim_\bC \cV_{\vartheta_{n+1,1}} }{ \dim_\bC \cV_{\vartheta_{n+1,0}} } \biggr) 
			\mu_n(\{ \tfrac{\Sigma(a)}{2} \}) 
			= 
			\biggl( 1 + \frac{ \dim_\bC \cV_{\vartheta_{n+1,0}} }{ \dim_\bC \cV_{\vartheta_{n+1,1}} } \biggr) 
			\mu_n(\{ \tfrac{1+\Sigma(a)}{2} \})
			. 
		\end{aligned}
	\end{align}
	Since $\mu_n(\{ \Sigma(a) \})>0$ for all $a\in J$, every $\mu_n$ is equivalent to $\mu_1$, and $L^\infty([0,1],\mu_n)=L^\infty([0,1],\mu_1)=\ell^\infty(\Sigma(J))\cong \ell^\infty(J)$. Also, by the density of $\Sigma(J)\subset [0,1]$, the canonical unital $*$-homomorphism $C([0,1])\to L^\infty([0,1],\mu_1)$ is injective, and we identify $C([0,1])\subset \ell^\infty(\Sigma(J))$. 
	
	For $n\in\bZ_{\geq 1}$, we consider the von Neumann algebra $M_n:=\cB(\cV_{\vartheta_{[1,n]}}) \barotimes \ell^\infty(\Sigma(J))$, which contains $A_n$ as a $\sigma$-weakly dense $*$-subalgebra. For $\pi\in\Irr G$, we have the injective unital normal $*$-homomorphism 
	\begin{align*}
		\bar{\gamma}_{\pi,n} 
		:={}& 
		\Ad(V_{\pi,\vartheta_{[1,n]}} \otimes 1_{\ell^\infty(\Sigma(J))}) (1_{\cB(\cH_\pi)} \otimes (-)) 
		\colon 
		M_{n} 
		\to \cB(\cH_\pi)\otimes M_{n} , 
	\end{align*}
	and the normal $*$-homomorphism $\bar{\gamma}_{n}\colon M_n\to \ell^\infty(\Gamma)\barotimes M_n$ such that $(\Pi_\pi\otimes\id_{M_n})\bar{\gamma}_n = \bar{\gamma}_{\pi,n}$ for all $\pi\in\Irr G$ and that $\bar{\gamma}_{n}$ is a left $\Gamma^{\op}$-action on $M_{n}$ normally extending the left $\Gamma^{\op}$-action $\gamma_n$ on $A_n$. 
	For $m,n\in\bZ_{\geq 1}$ with $n\leq m$, the map $\iota_{n,m}\colon A_{n}\to A_{m}$ normally extends to the injective unital normal $*$-homomorphism still denoted by $\iota_{n,m}\colon M_{n} \to M_{m}$ since $\frac{t}{2}$ and $\frac{1+t}{2}$ restrict to maps $\Sigma(J)\to\Sigma(J)$ and any $s\in\Sigma(J)$ is either of the form $\frac{t}{2}$, $\frac{1+t}{2}$, or $\frac{1}{2}$ for some $t\in\Sigma(J)$. This normal extension $\iota_{n,m}$ is a $\Gamma^{\op}$-$*$-homomorphism. 
	We often regard $M_n\subset M_m$ via $\iota_{n,m}$. 
	
	For $n\in\bZ_{\geq 1}$, consider the faithful normal tracial state 
	\begin{align*}
		\tau_n \colon M_n \ni 
		f &\mapsto \int_{[0,1]} \tr_{\cV_{\vartheta_{[1,n]}}}( f(t) ) d\mu_n(t) 
		\in \bC 
	\end{align*}
	and the normal conditional expectation $E_{n,n+1}\colon M_{n+1} \to M_n$ defined by, for $f\in M_{n+1}$, 
	\begin{align*}
		&
		E_{n,n+1}(f)(t):=
		\begin{cases}
			(\id_{\cB(\cV_{\vartheta_{[1,n]}})}\otimes \tr_{\cV_{\vartheta_{n+1,0}}}) ( (U_{n}^*f U_{n})(2t) ) & \text{if $0\leq t < \tfrac{1}{2}$}
			\\
			\int_{[0,1]}
			(\id_{\cB(\cV_{\vartheta_{[1,n]}})}\otimes \tr_{\cV_{\vartheta_{n+1,\frac{1}{2}}}}) ( (U_{n}^*f U_{n})(s) ) d\mu_{n+1}(s) & \text{if $t=\frac{1}{2}$}
			\\
			(\id_{\cB(\cV_{\vartheta_{[1,n]}})}\otimes \tr_{\cV_{\vartheta_{n+1,1}}}) ( (U_{n}^*f U_{n})(2t-1) ) & \text{if $\tfrac{1}{2}< t\leq 1$}
		\end{cases}, 
	\end{align*}
	where, for $i=0,\frac{1}{2},1$ and $x\in\cB(\cV_{\vartheta_{n+1}})$, we abused notation to write $\tr_{\cV_{\vartheta_{n+1,i}}}(x) := \tr_{\cV_{\vartheta_{n+1,i} }}(p_{n+1,i} x p_{n+1,i})$ using the projection $p_{n+1,i}$ from $\cV_{\vartheta_{n+1}}$ onto $\cV_{\vartheta_{n+1,i}}$. 
	For $m,n\in\bZ_{\geq 1}$ with $n\leq m$, we set $E_{n,m}:=E_{n,n+1}\cdots E_{m-1,m}$. In particular, $E_{n,n}=\id_{M_n}$. 
	Using \eqref{eq_prf_thm_IU_JiangSu_1}, we see for all $n\in\bZ_{\geq 1}$ and $f\in M_{n+1}$ that 
	\begin{align*}
		\tau_n E_{n,n+1}(f) 
		&= 
		\int_{[0,\frac{1}{2})}
		(\tr_{\cV_{\vartheta_{[1,n]}}}\otimes \tr_{\cV_{\vartheta_{n+1,0}}}) ( (U_n^*f U_n)(2t) ) d\mu_n(t)
		\\
		&+
		\mu_n(\{\tfrac{1}{2}\}) 
		\int_{[0,1]}
		(\tr_{\cV_{\vartheta_{[1,n]}}}\otimes \tr_{\cV_{\vartheta_{n+1,\frac{1}{2}}}}) ( (U_n^*f U_n)(s) ) d\mu_{n+1}(s) 
		\\
		&+
		\int_{(\frac{1}{2},1]}
		(\tr_{\cV_{\vartheta_{[1,n]}}}\otimes \tr_{\cV_{\vartheta_{n+1,1}}}) ( (U_n^*f U_n)(2t-1) ) d\mu_n(t)
		\\&= 
		\frac{\dim_\bC \cV_{\vartheta_{n+1,0}}}{\dim_\bC \cV_{\vartheta_{n+1}}} 
		\int_{[0,1)}
		(\tr_{\cV_{\vartheta_{[1,n]}}}\otimes \tr_{\cV_{\vartheta_{n+1,0}}}) ( (U_n^*f U_n)(s) ) d\mu_{n+1}(s)
		\\
		&+
		\frac{\dim_\bC \cV_{\vartheta_{n+1,\frac{1}{2}}}}{\dim_\bC \cV_{\vartheta_{n+1}}} 
		\int_{[0,1]}
		(\tr_{\cV_{\vartheta_{[1,n]}}}\otimes \tr_{\cV_{\vartheta_{n+1,\frac{1}{2}}}}) ( (U_n^*f U_n)(s) ) d\mu_{n+1}(s) 
		\\
		&+
		\frac{\dim_\bC \cV_{\vartheta_{n+1,1}}}{\dim_\bC \cV_{\vartheta_{n+1}}} 
		\int_{(0,1]}
		(\tr_{\cV_{\vartheta_{[1,n]}}}\otimes \tr_{\cV_{\vartheta_{n+1,1}}}) ( (U_n^*f U_n)(s) ) d\mu_{n+1}(s)
		\\&= 
		\int_{[0,1]}
		(\tr_{\cV_{\vartheta_{[1,n]}}}\otimes \tr_{\cV_{\vartheta_{n+1}}}) ( (U_n^*f U_n)(t) ) d\mu_{n+1}(t) 
		= \tau_{n+1}(f), 
	\end{align*}
	where for the second last equality we used for $x\in\cB(\cV_{\vartheta_{n+1}})$, 
	\begin{align*}
		&
		\tr_{\cV_{\vartheta_{n+1}}}(x) 
		= 
		\sum_{i\in\{0,\tfrac{1}{2},1\}} \tr_{\cV_{\vartheta_{n+1} }}(x p_{n+1,i})
		= 
		\sum_{i\in\{0,\tfrac{1}{2},1\}} \frac{\dim_\bC \cV_{\vartheta_{n+1,i}}}{\dim_\bC \cV_{\vartheta_{n+1}}} \tr_{\cV_{\vartheta_{n+1,i}}}(p_{n+1,i} x p_{n+1,i}) . 
	\end{align*}
	Thus $\tau_n E_{n,m}=\tau_m$ and $\tau_n=\tau_nE_{n,m}\iota_{n,m}=\tau_m \iota_{n,m}$ for all $m,n\in\bZ_{\geq 1}$ with $n\leq m$. 
	It follows from the faithfulness of $\tau_m$ that $E_{n,m}$ is faithful. 
	(Note that $E_{n,n+1}(A_{n+1})\not\subset A_n$ as $E_{n,n+1}(f)$ might not be continuous at $\tfrac{1}{2}$ for $f\in A_{n+1}$.)
	
	The inductive limit of $\tau_n$ on $A_n$ yields the well-defined tracial state on $A_\infty$ denoted by $\tau_\infty$, which is faithful by the uniqueness of the tracial state on $\cZ$. 
	We write $M:=A_\infty''\subset \cB(L^2(A_\infty,\tau_\infty))$. 
	Then $\tau_\infty$ uniquely extends to the faithful normal tracial state on $M$ still denoted by $\tau_\infty$, and $\iota_{n,\infty}\colon A_n\to M$ normally extends to an injective unital $*$-homomorphism $M_n\to M$ still denoted by $\iota_{n,\infty}$ for each $n\in\bZ_{\geq 1}$. 
	We often regard $M_n\subset M$ via $\iota_{n,\infty}$. 
	Also, for $n\in\bZ_{\geq 1}$, we have the normal conditional expectation $E_{n,\infty}\colon M\to M_n$ with $\tau_n E_{n,\infty}=\tau_\infty$ such that $E_{n,\infty}(x)=E_{n,m}(x)$ for $x\in M_m$ with $m\in\bZ_{\geq n}$ by considering the orthogonal projection from $L^2(M,\tau_\infty)$ onto the closed subspace $L^2(M_n,\tau_\infty)$. 
	Then, $E_{n,\infty}(x)\to x$ as $n\to\infty$ in the weak topology of $L^2(M,\tau_\infty)$. 
	
	For each $\pi\in\Irr G$, fixing an orthonormal basis $(e_i)_{i=1}^{\dim_\bC\cH_\pi}$ of $\cH_\pi$ and setting $v_{i,j}\in \cB(\cV_{\vartheta_{[1,n]}})$ for the $(i,j)$-th component of $V_{\pi,\vartheta_{[1,n]}}\in \cB(\cH_\pi\otimes\cV_{\vartheta_{[1,n]}})$ with respect to $(e_i)_i$, we check that for all $n\in\bZ_{\geq 1}$, $f\in A_n$, and $k,l\in\{1,\cdots,\dim_\bC\cH_\pi\}$, 
	\begin{align*}
		&
		\bra e_l, (\id_{\cB(\cH_\pi)} \otimes\tau_\infty) \gamma_{\pi}\iota_{n,\infty}(f) e_k \ket 
		\\={}& 
		\Big\bra e_l, (\id_{\cB(\cH_\pi)} \otimes\tau_n) \bigl(  (V_{\pi,\vartheta_{[1,n]}} \otimes 1_{\ell^\infty(\Sigma(J))}) (1_{\cB(\cH_\pi)}\otimes f)  (V_{\pi,\vartheta_{[1,n]}}^* \otimes 1_{\ell^\infty(\Sigma(J))}) \bigr) e_k \Big\ket 
		\\={}&
		\sum_{i=1}^{\dim_\bC\cH_{\pi}} 
		\tau_n \bigl( (v_{l,i}\otimes 1_{\ell^\infty(\Sigma(J))}) f (v_{k,i}^*\otimes 1_{\ell^\infty(\Sigma(J))}) \bigr) 
		\\={}& 
		\sum_{i=1}^{\dim_\bC\cH_{\pi}} 
		\tau_n ( (v_{k,i}^*v_{l,i}\otimes 1_{\ell^\infty(\Sigma(J))}) f ) 
		=
		\delta_{k,l} \tau_n ( f ) 
		= 
		\bra e_l, \tau_\infty\iota_{n,\infty}(f) e_k \ket , 
	\end{align*}
	where we used the biunitarity of $V_{\pi,\vartheta_{[1,n]}}$ by \autoref{rem_RFD_Kac} at the second last equality. 
	Since $A_\infty = \cspan\bigcup_{n\in\infty}\iota_{n,\infty}(A_n)$, we see that $\gamma$ is $\tau_\infty$-preserving. 
	Thus, by \autoref{lem_tr_pres_action}, we see that $\gamma$ uniquely extends to the left $\Gamma^{\op}$-action on $M=A_{\infty}''\subset \cB(L^2(A,\tau_\infty))$ denoted by $\bar{\gamma}$, which is again $\tau_\infty$-preserving. For $n\in\bZ_{\geq 1}$, we have that $\iota_{n,\infty}\colon M_{n}\to M$ is an injective unital normal $\Gamma^{\op}$-$*$-homomorphism by the normality of $\bar{\gamma}$ and $\bar{\gamma}_n$. 
	
	Note that $V_{\pi,\vartheta_{[1,n+1]}}\otimes1_{\ell^\infty(\Sigma(J))}$ commutes with $\id_{\cH_\pi}\otimes U_n$ since each $U_n(t)$ is a morphism in $\Rep\Gamma$. 
	Using the definition of $E_{n,m}$ and the biunitarity of $V_{\pi,\vartheta_{k,i}}$ for $\pi\in\Irr G$, $k\in\bZ_{\geq 1}$, and $i\in\{0,\tfrac{1}{2},1\}$, we can check that $(\id_{\cB(\cH_\pi)}\otimes E_{n,m}) \bar{\gamma}_{\pi,m} = \bar{\gamma}_{\pi,n}E_{n,m}$ for all $m,n\in\bZ_{\geq 1}$ with $n\leq m$. 
	It follows from the normality that $(\id_{\cB(\cH_\pi)}\otimes E_{n,\infty}) \bar{\gamma}_{\pi} = \bar{\gamma}_{\pi,n}E_{n,\infty}$. 
	
	\medskip\paragraph{\bf Step 4}
	
	We show the outerness of the left $\Gamma^{\op}$-action $\gamma$ on $M$, which is divided into \autoref{clm_thm_IU_JiangSu_inv} and \autoref{clm_thm_IU_JiangSu_irr} below. 
	For $\pi\in\Rep G$, we consider the finite dimensional linear subspace 
	\begin{align*}
		&
		C_\pi := \{ (\omega\otimes\id_{C(G)})(V_\pi) \mid \omega\in\cB(\cH_\pi)_* \}\subset\cO(G) . 
	\end{align*}
	Note that for any $\varphi\in\Rep\Gamma$ and $\omega\in\cB(\cH_\pi)_*$, we have 
	\begin{align}\label{eq_prf_thm_IU_JiangSu_2}
		&
		\Pi_\varphi( (\omega\otimes\id_{C(G)})(V_\pi) ) = (\omega\otimes\id_{\cB(\cV_\varphi)})(V_{\pi,\varphi}). 
	\end{align}
	
	\begin{clm}\label{clm_thm_IU_JiangSu_inv}
		For any $\pi\in\Irr G\setminus\{\mathbbm{1}\}$ that is invertible, the normal $*$-automorphism $\bar{\gamma}_\pi\colon M\to \cB(\cH_\pi)\barotimes M \cong M$ via the identification $\cB(\cH_\pi)\cong \bC$ is not inner. 
	\end{clm}
	
	We show \autoref{clm_thm_IU_JiangSu_inv}. 
	Note that we have $C_{\mathbbm{1}\oplus\pi} = \bC1_{\cO(G)} \oplus C_\pi$ by $\mathbbm{1}\neq \pi$. 
	By our choice of $\theta$, we can take $r\in\bZ_{\geq 1}$ such that $\bigoplus_{k=1}^{r} \Pi_{\theta^{\otimes k}} \colon C_{\mathbbm{1}\oplus\pi} \subset \cO(G) \to \bigoplus_{k=1}^{r} \cB(\cV_{\theta^{\otimes k}})$ and thus $\Pi_{\phi_{r}} \colon C_{\mathbbm{1}\oplus\pi} \subset \cO(G) \to \cB(\cV_{\phi_{r}})$ are injective. 
	Since $V_{\mathbbm{1},\phi_{r}}=1\in \cZ( \cB(\cH_{\mathbbm{1}}\otimes\cV_{\phi_{r}})) \cong \cZ(\cB(\cV_{\phi_{r}}))$, it follows from \eqref{eq_prf_thm_IU_JiangSu_2} that $V_{\pi,\phi_{r}}\in \cU(\cH_\pi\otimes\cV_{\phi_{r}})$ is not contained in $\bC1= \cZ(\cB(\cV_{\phi_{r}}))\cong \cZ(\cB(\cH_\pi\otimes\cV_{\phi_{r}}))$. 
	Thus, $c:= |\tr_{\cV_{\phi_{r}}}(V_{\pi,\phi_{r}})|<1$. Then, for any $n\in\bZ_{\geq 1}$ with $n\geq r$, 
	by taking $\varphi := (\theta\oplus\mathbbm{1}^{\oplus 2})^{\otimes 2^{n} - 2^{r}} \in \Rep\Gamma$ and $\varphi' \in \Rep\Gamma$ such that $\varphi'\oplus\mathbbm{1}^{\oplus2}\cong \phi_{n}$, 
	we see that 
	\begin{align*}
		&
		|\tr_{\cV_{\vartheta_{n,1}}}(V_{\pi,\vartheta_{n,1}})| 
		= 
		|\tr_{\cV_{\phi_{n}}}(V_{\pi,\phi_{n}})| 
		= 
		|\tr_{\cV_{\phi_{r}}}(V_{\pi,\phi_{r}})| |\tr_{\cV_{\varphi}}(V_{\pi,\varphi})| 
		\leq c, 
	\end{align*}
	\begin{align*}
		|\tr_{\cV_{\vartheta_{n,0}}}(V_{\pi,\vartheta_{n,0}})| 
		={}& 
		|\tr_{\cV_{\psi_{n}}}(V_{\pi,\psi_{n}})| 
		\\={}& 
		\biggl| \tr_{\cV_{\phi_{n}}}(V_{\pi,\phi_{n}}) \frac{\dim_\bC\cV_{\phi_{n}}}{1+\dim_\bC\cV_{\phi_{n}}} + \frac{1}{1+\dim_\bC\cV_{\phi_{n}}} \biggr| 
		\\\leq{}& 
		c\frac{\dim_\bC\cV_{\phi_{n}}}{1+\dim_\bC\cV_{\phi_{n}}} + \frac{1}{1+\dim_\bC\cV_{\phi_{n}}} 
		\leq
		\frac{1+c}{2} <1, 
	\end{align*}
	\begin{align*}
		|\tr_{\cV_{\vartheta_{n,\frac{1}{2}}}}(V_{\pi,\vartheta_{n,\frac{1}{2}}})| 
		={}& 
		\biggl| \tr_{\cV_{\psi_{n}}}(V_{\pi,\psi_{n}}) \tr_{\cV_{\varphi'}}(V_{\pi,\varphi'}) 
		\frac{\dim_\bC\cV_{\vartheta_{n,\frac{1}{2}}} -1}{\dim_\bC\cV_{\vartheta_{n,\frac{1}{2}}}} + \frac{1}{\dim_\bC\cV_{\vartheta_{n,\frac{1}{2}}}} \biggr| 
		\\\leq{}& 
		\frac{1+c}{2}\frac{\dim_\bC\cV_{\vartheta_{n,\frac{1}{2}}} -1}{\dim_\bC\cV_{\vartheta_{n,\frac{1}{2}}}} + \frac{1}{\dim_\bC\cV_{\vartheta_{n,\frac{1}{2}}}} 
		\leq
		\frac{3+c}{4}<1 , 
	\end{align*}
	where we have used $\dim_\bC\cV_{\phi_{n}}\geq 1$ and $\dim_\bC\cV_{\vartheta_{n,\frac{1}{2}}}\geq 2$. 
	Thus, for all $n\in\bZ_{\geq 1}$ with $n\geq r$ and $i\in\{0,\frac{1}{2},1\}$, we have $|\tr_{\cV_{\vartheta_{n,i}}}(V_{\pi,\vartheta_{n,i}})| \leq c':=\frac{3+c}{4}<1$.

	Note that $\bar{\gamma}_\pi \in \Aut(M)$ is the normal extension of the limit of $\Ad V_{\pi,\vartheta_{[1,n]}} =\Ad \bigotimes_{k=1}^{n} V_{\pi,\vartheta_{k}} \in \Aut(\cB(\cV_{\vartheta_{[1,n]}}))$ via $\cB(\cH_\pi)\cong \bC$. 
	We suppose that there is $\wt{V}\in \cU(M)$ such that $\bar{\gamma}_\pi=\Ad \wt{V}$ to derive a contradiction. 
	For any $n\in\bZ_{\geq 1}$ and $x\in M_n$, we have $x V_{\pi,\vartheta_{[1,n]}}^* E_{n,\infty}(\wt{V}) = V_{\pi,\vartheta_{[1,n]}}^* E_{n,\infty}(\wt{V}) x$ by $E_{n,\infty}\bar{\gamma}_\pi=\bar{\gamma}_{\pi,n} E_{n,\infty}$ to see that 
	\begin{align*}
		&
		f_n := V_{\pi,\vartheta_{[1,n]}}^* E_{n,\infty}(\wt{V}) \in \cZ( M_n ) = \id_{\cV_{\vartheta_{[1,n]}}} \otimes \ell^\infty(\Sigma(J)) . 
	\end{align*}
	Note that $\| f_n \|\leq 1$ by definition. 
	Identifying $\cB(\cH_\pi)\cong \bC$ and using $V_{\pi,\vartheta_{[1,n+1]}}$ commutes with $U_n(t)$ for all $t\in\Sigma(J)$, we see that 
	\begin{align*}
		&
		\|f_{n}\| 
		= \| V_{\pi,\vartheta_{[1,n]}}^* E_{n,n+1}E_{n+1,\infty}( \wt{V} ) \| 
		= \| V_{\pi,\vartheta_{[1,n]}}^* E_{n,n+1}( V_{\pi,\vartheta_{[1,n+1]}} f_{n+1} ) \| 
		\\\leq{}& 
		\max\{ \| (\id_{\cB(\cV_{\vartheta_{[1,n]}})}\otimes \tr_{\cV_{\vartheta_{n+1,i}}}\otimes \id_{\ell^\infty(\Sigma(J))}) (U_n^* V_{\pi,\vartheta_{[1,n+1]}} f_{n+1} U_n ) \| \mid i=0,\tfrac{1}{2},1 \} 
		\\={}& 
		\| f_{n+1} \| V_{\pi,\vartheta_{[1,n]}} \| 
		\max\{ | \tr_{\cV_{\vartheta_{n+1,i}}}( V_{\pi,\vartheta_{n+1,i}} ) | \mid i=0,\tfrac{1}{2},1 \} 
		\leq
		\| f_{n+1} \| c' .
	\end{align*}
	Thus, for any $n,k\in\bZ_{\geq 1}$ with $n\geq r$, we see that $\| f_n \| \leq \| f_{n+k} \| c^{\prime k} \leq c^{\prime k} \xto{k\to\infty} 0$. 
	It follows that $E_{n,\infty}(\wt{V})= V_{\pi,\vartheta_{[1,n]}} f_n = 0$ for all $n\in \bZ_{\geq r}$ and thus that $\wt{V}=0$ by the weak convergence, which contradicts $\wt{V}\in\cU(M)$ as desired. 
	Thus, \autoref{clm_thm_IU_JiangSu_inv} holds. 
	In particular, there is no unitary $\wt{V}\in\cU(A_\infty)$ such that $\Ad \wt{V} = \gamma_{\pi} \colon A_\infty\to \cB(\cH_\pi)\otimes A_\infty\cong A_\infty$ by the normality of $\bar{\gamma}_\pi$.

	\begin{clm}\label{clm_thm_IU_JiangSu_irr}
		For each $\pi\in\Irr G$, we have $\bar{\gamma}_\pi(M)'\cap \cB(\cH_\pi)\otimes M=\bC1_{\cB(\cH_\pi)\otimes M}$. 
	\end{clm}
	
	We show \autoref{clm_thm_IU_JiangSu_irr}. 
	By our choice of $\theta$, for any $F(X)\in\bZ_{\geq 0}[X]\setminus \bZ_{\geq 0}$, we can take $a\in \bZ_{\geq 1}$ such that $\Pi_{F(\theta)^{\otimes n}}\colon C_\pi\subset \cO(G)\to \cB(\cV_{F(\theta)^{\otimes n}})$ is injective for all $n\in\bZ_{\geq 1}$ with $n\geq a$. 
	Thus, letting $F(X):=X+2$ and using \eqref{eq_prf_thm_IU_JiangSu_2}, we can take $a\in\bZ_{\geq 1}$ such that the unitary $V_{\pi,\phi_a}\in\cU(\cH_\pi\otimes\cV_{\phi_a})$ satisfies the assumption of \autoref{cor_quantum_channel_irred} because $\dim_\bC C_\pi = \dim_\bC \cB(\cH_\pi)$ by Schur's orthogonality. 
	It follows from \autoref{cor_quantum_channel_irred} \ref{item_cor_quantum_channel_irred_4} that $\rho:= \| \Theta_{V_{\pi,\phi_b}} - \wt{\tr}_{\cH_\pi} \| \in [0,1)$ for some $b\in\bZ_{\geq 1}$ with $b\geq a$. 
	Then, for any $n\in\bZ_{\geq 1}$ with $n\geq b$, 
	by taking $\varphi \in \Rep\Gamma$ such that $\varphi\oplus\mathbbm{1}^{\oplus2}\cong \phi_{n}$, 
	using \autoref{lem_quantum_channel_irred} we see that 
	\begin{align*}
		&
		\| \Theta_{V_{\pi,\vartheta_{n,1}}} - \wt{\tr}_{\cH_\pi} \| 
		= 
		\| \Theta_{V_{\pi,\phi_{n}}} - \wt{\tr}_{\cH_\pi} \| 
		= 
		\| (\Theta_{V_{\pi,\phi_{b}}} - \wt{\tr}_{\cH_\pi})^{2^{n-b}} \| 
		\leq \rho^{2^{n-b}}, 
	\end{align*}
	\begin{align*}
		&\| \Theta_{V_{\pi,\vartheta_{n,0}}} - \wt{\tr}_{\cH_\pi} \| 
		= \| \Theta_{V_{\pi,\psi_{n}}} - \wt{\tr}_{\cH_\pi} \| 
		\\={}& 
		\biggl\| ( \Theta_{V_{\pi,\phi_{n}}} - \wt{\tr}_{\cH_\pi} ) \frac{\dim_\bC\cV_{\phi_{n}}}{1+\dim_\bC\cV_{\phi_{n}}} 
		+ ( \Theta_{V_{\pi,\mathbbm{1}}} - \wt{\tr}_{\cH_\pi} ) \frac{1}{1+\dim_\bC\cV_{\phi_{n}}} \biggr\| 
		\\\leq{}& 
		\rho^{2^{n-b}} + \frac{2}{1+\dim_\bC\cV_{\phi_{n}}} , 
	\end{align*}
	\begin{align*}
		&\| \Theta_{V_{\pi,\vartheta_{n,\frac{1}{2}}}} - \wt{\tr}_{\cH_\pi} \| 
		\\={}& 
		\biggl\| ( \Theta_{V_{\pi,\varphi}} - \wt{\tr}_{\cH_\pi} ) ( \Theta_{V_{\pi,\psi_{n}}} - \wt{\tr}_{\cH_\pi} ) 
		\frac{\dim_\bC\cV_{\vartheta_{n,\frac{1}{2}}} -1}{\dim_\bC\cV_{\vartheta_{n,\frac{1}{2}}}} + ( \Theta_{V_{\pi,\mathbbm{1}}} - \wt{\tr}_{\cH_\pi} ) \frac{1}{\dim_\bC\cV_{\vartheta_{n,\frac{1}{2}}}} \biggr\| 
		\\\leq{}& 
		2\biggl( \rho^{2^{n-b}} + \frac{2}{1+\dim_\bC\cV_{\phi_{n}}} \biggr) 
		+ \frac{2}{\dim_\bC\cV_{\vartheta_{n,\frac{1}{2}}}} , 
	\end{align*}
	and that the three rightmost expressions converges to $0$ as $b\leq n\to\infty$. Thus, there are $\wt{c}\in[0,1)$ and $r\in\bZ_{\geq 1}$ with $r\geq b$ such that 
		$\| \Theta_{V_{\pi,\vartheta_{n,i}}} - \wt{\tr}_{\cV_{\cH_\pi}} \| 
		\leq \wt{c} <1$ 
	for all $n\in\bZ_{\geq 1}$ with $n\geq r$ and $i\in\{0,\frac{1}{2},1\}$. 
	
	We take $x\in \bar{\gamma}_\pi(M)'\cap \cB(\cH_\pi)\barotimes M$ arbitrarily. 
	For $n\in\bZ_{\geq 1}$ with $n\geq r$, we let 
	\begin{align*}
		y_n 
		:={}& 
		(\id_{\cB(\cH_\pi)} \otimes \tr_{\cV_{\vartheta_{[1,n]}}} \otimes \id_{\ell^\infty(\Sigma(J))}) \Ad (V_{\pi,\vartheta_{[1,n]}}^{*} \otimes 1_{\ell^\infty(\Sigma(J))}) (\id_{\cB(\cH_\pi)}\otimes E_{n,\infty})(x) 
		\\\in{}& 
		\cB(\cH_\pi) \barotimes \ell^\infty(\Sigma(J)) .
	\end{align*}
	Since 
	\begin{align*}
		&
		\Ad (V_{\pi,\vartheta_{[1,n]}}^{*} \otimes 1_{\ell^\infty(\Sigma(J))}) (\id_{\cB(\cH_\pi)}\otimes E_{n,\infty})(x) 
		\\\in{}& 
		( \id_{\cH_\pi} \otimes M_{n} )' \cap ( \cB(\cH_\pi) \otimes M_{n} ) 
		= 
		\cB(\cH_\pi) \otimes \id_{\cV_{\vartheta_{[1,n]}}} \otimes \ell^\infty(\Sigma(J)) 
	\end{align*}
	by $(\id_{\ell^\infty(\Gamma)}\otimes E_{n,\infty})\bar{\gamma} = \bar{\gamma}_{n}E_{n,\infty}$, 
	we have that, using the leg notation, 
	\begin{align}\label{eq_clm2_thm_IU_JiangSu_2}
		&
		\begin{aligned}
		&
		\Ad (V_{\pi,\vartheta_{[1,n]}} \otimes 1_{\ell^\infty(\Sigma(J))}) (\id_{\cV_{\vartheta_{[1,n]}}}\otimes y_n)_{213} 
		\\={}& 
		(\id_{\cB(\cH_\pi)}\otimes E_{n,\infty})(x) 
		= 
		(\id_{\cB(\cH_\pi)}\otimes E_{n,n+1}E_{n+1,\infty})(x) 
		\\={}& 
		(\id_{\cB(\cH_\pi)}\otimes E_{n,n+1}) \Ad (V_{\pi,\vartheta_{[1,n+1]}} \otimes 1_{\ell^\infty(\Sigma(J))}) (\id_{\cV_{\vartheta_{[1,n+1]}}}\otimes y_{n+1})_{213} . 
		\end{aligned}
	\end{align}
	Since $\id_{\cH_\pi}\otimes U_n$ commutes with $V_{\pi,\vartheta_{[1,n+1]}} \otimes 1_{\ell^\infty(\Sigma(J))}$ and $(\id_{\cV_{\vartheta_{[1,n+1]}}} \otimes y_{n+1})_{213}$, 
	using \eqref{eq_clm2_thm_IU_JiangSu_2} and the definition of $E_{n,n+1}$, we see that for any $(t,i)\in \bigl( \Sigma(J)\cap [0,\frac{1}{2}) \bigr) \times \{0\} \sqcup \bigl( \Sigma(J)\cap (\frac{1}{2},1] \bigr) \times \{1\}$, 
	\begin{align*}
		y_{n}(t) 
		={}& 
		(\id_{\cB(\cH_\pi)}\otimes \tr_{\cV_{\vartheta_{n+1,i}}}) \Ad V_{\pi,\vartheta_{n+1,i}} ( y_{n+1}(2t-i)\otimes \id_{\cV_{\vartheta_{n+1,i}}} ) 
		\\={}& 
		\Theta_{V_{\pi,\vartheta_{n+1,i}}} (y_{n+1}(2t-i)) , 
	\end{align*}
	and 
	\begin{align*}
		y_{n}(\tfrac{1}{2}) 
		={}& 
		\int_{[0,1]} 
		(\id_{\cB(\cH_\pi)}\otimes \tr_{\cV_{\vartheta_{n+1,\frac{1}{2}}}}) \Ad V_{\pi,\vartheta_{n+1,\frac{1}{2}}} ( y_{n+1}(s)\otimes \id_{\cV_{\vartheta_{n+1,\frac{1}{2}}}} ) 
		d\mu_{n+1}(s) 
		\\={}& 
		\int_{[0,1]} 
		\Theta_{V_{\pi,\vartheta_{n+1,\frac{1}{2}}}} (y_{n+1}(s)) 
		d\mu_{n+1}(s) . 
	\end{align*}
	We define $z_n \in \ell^\infty(\Sigma(J))$ by 
	\begin{align*}
		z_n 
		:={}& 
		(\tr_{\cH_{\pi}}\otimes\id_{\ell^\infty(\Sigma(J))})(y_n) 
		\\={}& 
		(\tr_{\cH_{\pi}} \otimes \tr_{\cV_{\vartheta_{[1,n]}}} \otimes \id_{\ell^\infty(\Sigma(J))}) \Ad (V_{\pi,\vartheta_{[1,n]}}^{*} \otimes 1_{\ell^\infty(\Sigma(J))}) (\id_{\cB(\cH_\pi)}\otimes E_{n,\infty})(x) 
		\\={}& 
		(\tr_{\cV_{\vartheta_{[1,n]}}} \otimes \id_{\ell^\infty(\Sigma(J))})(\tr_{\cH_{\pi}}\otimes E_{n,\infty})(x) . 
	\end{align*}
	We see that 
	for any $(t,i)\in \bigl( \Sigma(J)\cap [0,\frac{1}{2}) \bigr) \times \{0\} \sqcup \bigl( \Sigma(J)\cap (\frac{1}{2},1] \bigr) \times \{1\}$, 
	\begin{align}\label{eq_clm2_thm_IU_JiangSu_3}
		&
		\begin{aligned}
			z_n(t) 
			={}& 
			\tr_{\cH_{\pi}}(y_n(t)) 
			= 
			\tr_{\cH_{\pi}}\Theta_{V_{\pi,\vartheta_{n+1,i}}} (y_{n+1}(2t-i)) 
			\\={}& 
			\tr_{\cH_{\pi}} (y_{n+1}(2t-i)) 
			= 
			z_{n+1}(2t-i) , 
		\end{aligned}
	\end{align}
	and thus 
	\begin{align*}
		&
		\| y_{n}(t) - z_n(t)\id_{\cH_\pi} \| 
		= 
		\| y_{n}(t) - z_{n+1}(2t-i)\id_{\cH_\pi} \| 
		\\={}& 
		\| \Theta_{V_{\pi,\vartheta_{n+1,i}}} (y_{n+1}(2t-i) - z_{n+1}(2t-i)\id_{\cH_\pi}) \| 
		\\={}& 
		\| (\Theta_{V_{\pi,\vartheta_{n+1,i}}}-\wt{\tr}_{\cH_\pi}) (y_{n+1}(2t-i) - z_{n+1}(2t-i)\id_{\cH_\pi}) \| 
		\\\leq{}& 
		\wt{c} \sup\{ \| y_{n+1}(s) - z_{n+1}(s)\id_{\cH_\pi} \| \mid s\in\Sigma(J) \} . 
	\end{align*}
	We also have 
	\begin{align}\label{eq_clm2_thm_IU_JiangSu_4}
		&
		\begin{aligned}
			z_n(\tfrac{1}{2}) 
			={}& 
			\tr_{\cH_{\pi}}(y_n(\tfrac{1}{2})) 
			= 
			\tr_{\cH_{\pi}} \biggl( \int_{[0,1]} 
			\Theta_{V_{\pi,\vartheta_{n+1,\frac{1}{2}}}} (y_{n+1}(s)) 
			d\mu_{n+1}(s) \biggr) 
			\\={}& 
			\int_{[0,1]} 
			\tr_{\cH_{\pi}} (y_{n+1}(s)) 
			d\mu_{n+1}(s) 
			= 
			\int_{[0,1]} z_{n+1}(s) d\mu_{n+1}(s) 
			\\={}& 
			\tau_{n+1}(\tr_{\cH_\pi}\otimes E_{n+1,\infty}) (x) 
			= 
			(\tr_{\cH_\pi}\otimes \tau_{\infty}) (x) , 
		\end{aligned}
	\end{align}
	and thus 
	\begin{align*}
		&
		\| y_{n}(\tfrac{1}{2}) - z_n(\tfrac{1}{2})\id_{\cH_\pi} \| 
		= 
		\biggl\| \int_{[0,1]} \Theta_{V_{\pi,\vartheta_{n+1,\frac{1}{2}}}} (y_{n+1}(s) - z_{n+1}(s)\id_{\cH_\pi}) d\mu_{n+1}(s) \biggr\|
		\\\leq{}& 
		\biggl\| \int_{[0,1]} (\Theta_{V_{\pi,\vartheta_{n+1,\frac{1}{2}}}} - \wt{\tr}_{\cH_\pi}) (y_{n+1}(s) - z_{n+1}(s)\id_{\cH_\pi}) d\mu_{n+1}(s) \biggr\| 
		\\\leq{}& 
		\wt{c} \sup\{ \| y_{n+1}(s) - z_{n+1}(s)\id_{\cH_\pi} \| \mid s\in\Sigma(J) \} . 
	\end{align*}
	Thus, for $n,k\in\bZ_{\geq 1}$ with $n\geq r$, 
	\begin{align*}
		&
		\sup \{ \| y_{n}(t) - z_n(t)\id_{\cH_\pi}\| \mid t\in \Sigma(J) \} 
		\\\leq{}& 
		\wt{c}^{k} \sup \{ \| y_{n+k}(t) - z_{n+k}(t)\id_{\cH_\pi}\| \mid  t\in \Sigma(J) \} 
		\\\leq{}& 
		\wt{c}^{k} (\| x \| + \| x \|)
		\xto{k\to\infty} 0, 
	\end{align*}
	and we see that $y_{n} = z_n\id_{\cH_\pi}\in \id_{\cH_\pi}\otimes \ell^\infty(\Sigma(J))$. 
	
	Moreover, inductively on $l(a)$, it follows from \eqref{eq_clm2_thm_IU_JiangSu_3} and \eqref{eq_clm2_thm_IU_JiangSu_4} that $z_n(\Sigma(a))=z_{n+l(a)-1}(\tfrac{1}{2}) = (\tr_{\cH_\pi}\otimes \tau_\infty)(x)$ for all $a\in J$ and $n\in\bZ_{\geq 1}$ with $n\geq r$. 
	In particular, $z_n=z_r$. 
	Therefore, for all $n\in\bZ_{\geq 1}$ with $n\geq r$, we have $y_n=z_r\id_{\cH_\pi}$ and thus, by \eqref{eq_clm2_thm_IU_JiangSu_2}, 
	\begin{align*}
		&
		(\id_{\cH_\pi}\otimes E_{n,\infty})(x) = z_r \id_{\cH_\pi}\otimes \id_{\cV_{\vartheta_{[1,n]}}}\otimes 1_{\ell^\infty(\Sigma(J))} , 
	\end{align*}
	which implies $x=z_r 1_M\in \id_{\cH_\pi}\otimes\bC1_{M}$ by the weak convergence. 
	Therefore, \autoref{clm_thm_IU_JiangSu_irr} holds. 
	
	
	By \autoref{clm_thm_IU_JiangSu_inv} and \autoref{clm_thm_IU_JiangSu_irr}, the left $\Gamma^{\op}$-action $\bar{\gamma}$ on $M$ is outer, and thus $\gamma$ is outer as a left $\Gamma^{\op}$-action on $A_\infty$. This completes the proof of \autoref{thm_IU_JiangSu}. 
	Here, note that $M$ is a factor by \autoref{clm_thm_IU_JiangSu_irr} applied to $\pi=\mathbbm{1}$. 
\end{proof}

\begin{thm}\label{thm_IU_JiangSu_op}
	Let $G$ be a compact quantum group with $\theta\in\Rep\Gamma$ that is tensorially faithful. 
	Then, for any C*-algebra $A$ with $A'\cap\cM(A)\cong\bC$ and $\cZ\otimes A\cong A$, there is an outer left $\Gamma$-action on $A$. 
\end{thm}

\begin{proof}
	We only have to construct an outer left $\Gamma$-action $\gamma$ on  algebra $\cZ$ by considering $\gamma\otimes \id_A$. 
	By \autoref{rem_tensor_faith_op}, we may apply \autoref{thm_IU_JiangSu} to $\Gamma^{\op}$ instead of $\Gamma$, which yields an outer left action on $\cZ$ of $(\Gamma^{\op})^{\op} = \Gamma$. 
\end{proof}

\begin{cor}\label{cor_IU_JiangSu}
	Let $G$ be a compact quantum group whose Pontryagin dual admits a tensorially faithful finite dimensional unitary representation, or one of the following: 
	A finite quantum group, 
	a compact Lie group, 
	$\rU^+(n)$ with $n\in\bZ_{\geq 2}$, 
	$\SO_q(3)$ with $q\in(0,1)$ satisfying $(q+q^{-1})^2\in \bZ_{\geq 4}$, 
	or $\SU_q(2)$ with $q\in(-1,1)\setminus\{0\}$ satisfying $q+q^{-1}\in 2\bZ_{\geq 1} \cup -\bZ_{\geq 2}$. 
	Then, any C*-algebra $A$ satisfying $A'\cap\cM(A)\cong\bC$ and $\cZ\otimes A \cong A$ admits an outer $\Rep G$-action on $A$. 
\end{cor}

Note that \autoref{main_JiangSu} follows from this because of the unitary monoidal equivalences as at the end of \autoref{ssec_freeQG}. 
By \autoref{rem_tensor_faith_op}, the first condition in the statement does not depend on the convention of the Pontryagin dual discussed in \autoref{ssec_prelim_CQGDQG}. 

\begin{proof}
	When there is some tensorially faithful $\theta\in\Rep\Gamma$, 
	this is a direct consequence of \autoref{thm_IU_JiangSu} in view of \autoref{rem_outer_tensorcat}. 
	Thus, the remaining cases follow from examples in \autoref{sec_separateDQG} with the aid of the unitary monoidal equivalences. 
\end{proof}

\section{Variants and applications}\label{sec_applications}

\subsection{RFD quantum groups}\label{ssec_subfactor}

Similarly to \autoref{thm_IU_JiangSu}, we can show the following. 

\begin{thm}\label{thm_RFD_UHF}
	Let $G$ be a compact quantum group with $L^2(G)$ separable such that $\Gamma=\wc{G}$ is RFD. Take any $2\leq d\in\bZ$. Then, there is a trace-preserving outer left $\Gamma^{\op}$-action on the UHF algebra $\bM_{d^\infty}$ that extends to a trace-preserving outer left $\Gamma^{\op}$-action on the von Neumann completion $M$ of $\bM_{d^\infty}$ with respect to the unique faithful tracial state on $\bM_{d^\infty}$. 
	Also, there is a trace-preserving outer left $\Gamma$-action on the UHF algebra $\bM_{d^\infty}$ that extends to a trace-preserving outer left $\Gamma$-action on $M$. 
\end{thm}

It is tempting to expect that \autoref{thm_IU_JiangSu} holds for any RFD discrete quantum group with $\ell^2(\Gamma)$ separable, but the author was unable to show this. 

\begin{proof}
	As in \autoref{cor_IU_JiangSu}, it suffices to show the former statement. 
	We have the well-defined left $\Gamma^{\op}$-action $\gamma_{\varphi}:=\Ad V_{\varphi}(1_{\ell^\infty(\Gamma)}\otimes (-))\colon \cB(\cV_\varphi) \to \cM(c_0(\Gamma)\otimes \cB(\cV_\varphi))$ for each $\varphi\in\Rep\Gamma$. We write $\gamma_{\pi,\varphi}:=\Ad V_{\pi,\varphi}(1_{\cB(\cH_\pi)}\otimes (-)) = (\Pi_\pi\otimes\id_{\cB(\cV_\varphi)})\gamma_{\varphi}$. 
	
	By assumption, we can take $\theta_n\in \Rep \Gamma$ for $n\in\bZ_{\geq 1}$ such that $\prod_{n=1}^{\infty}\Pi_{\theta_n}\colon \cO(G)\to \prod_{n=1}^{\infty}\cB(\cV_{\theta_n})$ is injective. 
	We take $p_n\in\bZ_{\geq 1}$ such that $p_n+\sum_{k=1}^{n}\dim_\bC\cV_{\theta_k} \in \{ d^n \mid n\in\bZ_{\geq 1} \}$ for $n\in\bZ_{\geq 1}$ and define the objects in $\Rep\Gamma$ by 
	\begin{align*}
		\phi_n :={}& \mathbbm{1}^{\oplus p_n}\oplus \bigoplus_{k=1}^{n}\theta_k ,
		\qquad
		\vartheta_{n} :=\phi_1^{\otimes n}\otimes \phi_{2}^{\otimes n}\otimes\cdots\otimes \phi_n^{\otimes n}, 
		\\
		\vartheta_{[n,m]} :={}& \vartheta_{n}\otimes\vartheta_{n+1}\otimes\cdots\otimes\vartheta_{m}, 
		\qquad\text{and}\qquad 
		\vartheta_{\varnothing}:=\mathbbm{1} 
	\end{align*}
	for $m,n\in\bZ_{\geq 1}$ with $n\leq m$. 
	Note that $\dim_\bC\cV_{\vartheta_{[n,m]}}\geq d$ is a power of $d$ for $1\leq n\leq m$ and that $\bigoplus_{k=1}^{n}\theta_k$ is a subobject of $\vartheta_{n}=\vartheta_{[n,n]}$ and thus of $\vartheta_{[1,n]}$ by $p_k\geq 1$ for all $k\in\bZ_{\geq 1}$. 
	We put $\iota_{n,m}\colon \cB(\cV_{\vartheta_{[1,n]}})\ni x\mapsto x\otimes1 \in \cB(\cV_{\vartheta_{[1,n]}}\otimes\cV_{\vartheta_{[n+1,m]}}) = \cB(\cV_{\vartheta_{[1,m]}})$ for $1\leq n\leq m$. 
	For $n\in\bZ_{\geq 1}$, we see that $(\cB(\cV_{\vartheta_{[1,n]}}),\gamma_{\vartheta_{[1,n]}})$ is a well-defined left $\Gamma^{\op}$-C*-algebra and $\iota_{n,m}$ is a unital injective $\Gamma^{\op}$-$*$-homomorphism. 
	We put $A := \varinjlim_{n\to\infty}( \cB(\cV_{\vartheta_{[1,n]}}) , \iota_{n,n+1}) \cong \bM_{d^{\infty}}$ equipped with the left $\Gamma^{\op}$-action $\gamma$ induced by the inductive limit and let $\gamma_{\pi} := (\Pi_\pi\otimes\id_{\bM_{d^{\infty}}})\gamma$. We write $\iota_{n,\infty}\colon \cB(\cV_{\vartheta_{[1,n]}}) \to A$ for the unital injective $*$-homomorphism induced by the inductive limit. 
	
	We write $\tr_A$ for the unique faithful tracial state on $A$ and let $M:=A''\subset \cB(L^2(A,\tr_A))$, which is a hyperfinite $\mathrm{II}_1$ factor with separable predual. 
	Using the biunitarity of $V_{\pi,\vartheta_{[1,n]}}$ and \autoref{lem_tr_pres_action}, we observe that $\gamma_{\pi}$ extends to a unital normal $*$-homomorphism $\bar{\gamma}_\pi\colon M\to \cB(\cH_\pi)\barotimes M$ for each $\pi\in\Rep G$ as in the proof of \autoref{thm_IU_JiangSu}. Thus, we obtain the left $\Gamma^{\op}$-action $\bar{\gamma}$ on the von Neumann completion $M$ of $A$ extending $\gamma$, and $\bar{\gamma}$ is trace-preserving. 
	Also, for $n\in\bZ_{\geq 1}$, we write $\iota_{n,\infty}\colon \cB(\cV_{\vartheta_{[1,n]}})\to A\subset M$ for the unital $\Gamma^{\op}$-$*$-homomorphism induced by the inductive limit. 
	We often regard $\cB(\cV_{\vartheta_{[1,n]}})\subset M$ via $\iota_{n,\infty}$. 
	
	For each $\pi\in\Rep G$, we put $C_\pi := \{ (\omega\otimes\id_{C(G)})(V_\pi) \mid \omega\in\cB(\cH_\pi)_* \}\subset\cO(G)$. 
	Next, we show that for any invertible $\pi\in\Irr G\setminus\{\mathbbm{1}\}$, by regarding $\cB(\cH_\pi)\cong\bC$, the normal $*$-automorphism $\bar{\gamma}_\pi\colon M\to \cB(\cH_\pi)\barotimes M \cong M$ is not inner. By our choice of $\theta_n$, we can take $r\in\bZ_{\geq 1}$ such that $\bigoplus_{k=1}^{r}\Pi_{\theta_k}\colon C_{\mathbbm{1}\oplus\pi} \subset \cO(G) \to \cB(\bigoplus_{k=1}^{r}\cV_{\theta_k})$ and thus $\Pi_{\phi_r}\colon \bC1_{\cO(G)}\oplus C_{\pi} = C_{\mathbbm{1}\oplus\pi}\to \cB(\cV_{\phi_r})$ are injective. 
	Since $V_{\mathbbm{1},\phi_r}=1\in \cZ(\cB(\cH_{\mathbbm{1}}\otimes\cV_{\phi_r}))\cong \cZ(\cB(\cV_{\phi_r}))$, this means that $V_{\pi,\phi_r}\in \cB(\cH_\pi\otimes\cV_{\phi_r})$ is not contained in $\bC1= \cZ(\cB(\cV_{\phi_r}))\cong \cZ(\cB(\cH_\pi\otimes\cV_{\phi_r}))$. 
	Note that $\bar{\gamma}_\pi \in \Aut(M)$ is the normal extension of the limit of $\Ad V_{\pi,\vartheta_{[1,n]}}=\Ad \bigotimes_{k=1}^{n} \bigotimes_{l=1}^{k} V_{\pi,\phi_{l}}^{\otimes k} \in \Aut(\cB(\cV_{\vartheta_{[1,n]}}))$ via $\cB(\cH_\pi)\cong \bC$. The $*$-automorphism $\bar{\gamma}_\pi$ is the tensor product of countably infinitely many copies of $\Ad V_{\pi,\phi_r}$ and some $*$-automorphism, and it is well-known that such a $*$-automorphism is outer (see e.g., \cite[Proposition 4.1]{Kerr-Li-Pichot2010turbulence}). 
	Hence, to see that the left $\Gamma^{\op}$-actions on $A$ and $M$ are outer, it suffices to see that for each $\pi\in\Irr G$, we have $\bar{\gamma}_\pi(M)'\cap \cB(\cH_\pi)\otimes M\cong\bC$. We can show this claim similarly to \autoref{clm_thm_IU_JiangSu_irr} as follows. 
	
	By our choice of $\theta_n$, we can take $r\in \bZ_{\geq 1}$ such that $\bigoplus_{k=1}^{r}\Pi_{\theta_{k}}\colon C_\pi\subset \cO(G)\to \cB(\bigoplus_{k=1}^{r}\cV_{\theta_k})$ is injective. Since $\mathbbm{1}\oplus\bigoplus_{k=1}^{r}\theta_k\leq \phi_{r}$, we see that $\phi_r^{\otimes n}\colon C_\pi\to \cB(\cV_{\phi_{r}^{\otimes n}})$ is injective for all $n\in\bZ_{\geq 1}$. Thus, the unitary $V_{\pi,\phi_r}\in\cU(\cH_\pi\otimes\cV_{\phi_r})$ satisfies the assumption of \autoref{cor_quantum_channel_irred}. 
	
	For $m,n\in\bZ_{\geq 1}$ with $n\leq m$, we let 
	$E_{n,m} := \id_{\cB(\cV_{\vartheta_{[1,n]}})}\otimes \tr_{\cV_{\vartheta_{[n+1,m]}}} \colon \cB(\cV_{\vartheta_{[1,m]}})\to \cB(\cV_{\vartheta_{[1,n]}})$, which is a faithful conditional expectation. 
	Also, there is a normal conditional expectation $E_{n,\infty}\colon M\to \cB(\cV_{\vartheta_{[1,n]}})$ such that $E_{n,\infty}|_{\cB(\cV_{\vartheta_{[1,m]}})}=E_{n,m}$ for all $m,n\in\bZ_{\geq 1}$ with $m\leq n$. Moreover, $E_{n,\infty}(x)$ converges to $x$ as $n\to\infty$ in the weak topology with respect to $L^2(A,\tr_A)$ for all $x\in M$.

	Now, take any $x\in \bar{\gamma}_\pi(M)'\cap \cB(\cH_\pi)\barotimes M$ and $r\leq n\in\bZ_{\geq 1}$. 
	Then, we have 
	\begin{align*}
		(\id_{\cB(\cH_\pi)}\otimes E_{n,\infty})(x) 
		\in{}& \bar{\gamma}_{\pi,\vartheta_{[1,n]}}(\cB(\cV_{\vartheta_{[1,n]}}))'\cap ( \cB(\cH_\pi)\otimes \cB(\cV_{\vartheta_{[1,n]}}) ) 
		\\={}& 
		\Ad V_{\pi,\vartheta_{[1,n]}} (\cB(\cH_\pi)\otimes \id_{\cV_{\vartheta_{[1,n]}}}) . 
	\end{align*}
	We let $y_n := (\id\otimes\tr_{\cV_{\vartheta_{[1,n]}}}) \Ad V_{\pi,\vartheta_{[1,n]}}^* (\id\otimes E_{n,\infty})(x) \in \cB(\cH_\pi)$ so that $\Ad V_{\pi,\vartheta_{[1,n]}} (y_n\otimes\id_{\cV_{\vartheta_{[1,n]}}}) = (\id_{\cB(\cH_\pi)}\otimes E_{n,\infty})(x)$. 
	For $m\in\bZ_{\geq 1}$ with $n+1\leq m$, we have $y_{n} = \Theta_{V_{\pi,\vartheta_{[n+1,m]}}}(y_{m})$ because 
	\begin{align*}
		&
		\Ad V_{\pi,\vartheta_{[1,n]}} (y_n\otimes\id_{\cV_{\vartheta_{[1,n]}}}) 
		= 
		(\id_{\cB(\cH_\pi)}\otimes E_{n,m} E_{m,\infty})(x) 
		\\={}& 
		(\id_{\cB(\cH_\pi)}\otimes E_{n,m}) \Ad V_{\pi,\vartheta_{[1,m]}} (y_m\otimes\id_{\cV_{\vartheta_{[1,m]}}}) 
		\\={}& 
		(\id_{\cB(\cH_\pi)\otimes\cB(\cV_{\vartheta_{[1,n]}})}\otimes \tr_{\cV_{\vartheta_{[n+1,m]}}}) \Ad\bigl( (V_{\pi,\vartheta_{[1,n]}})_{12} (V_{\pi,\vartheta_{[n+1,m]}})_{13} \bigr) (y_m\otimes\id_{\cV_{\vartheta_{[1,m]}}}) 
		\\={}& 
		\Ad V_{\pi,\vartheta_{[1,n]}} (\Theta_{V_{\pi,\vartheta_{[n+1,m]}}}(y_m)\otimes \id_{\cV_{\vartheta_{[1,n]}}} ) . 
	\end{align*}
	Using the notation from \autoref{lem_quantum_channel_irred}, we see that 
	\begin{align*}
		&
		\tr_{\cH_\pi}(y_n) = \tr_{\cH_\pi\otimes\cV_{\vartheta_{[1,n]}}} (\id\otimes E_{n,\infty}) (x) = (\tr_{\cH_\pi}\otimes\tr_M)(x). 
	\end{align*}
	It follows that for $m\in\bZ_{\geq 1}$ with $n+1\leq m$, 
	\begin{align*}
		&
		\| (\id_{\cB(\cH_\pi)}\otimes E_{n,\infty})(x) - (\tr_{\cH_\pi}\otimes \tr_M)(x)\id_{\cH_\pi\otimes\cV_{\vartheta_{[1,n]}}} \| 
		\\={}& 
		\| \Ad V_{\pi,\vartheta_{[1,n]}} \| \| y_n - \tr_{\cH_\pi}(y_n)\id_{\cH_\pi} \| 
		\\={}& 
		\| \Theta_{V_{\pi,\vartheta_{[n+1,m]}}} ( y_{m} - \wt{\tr}_{\cH_\pi}(y_{m}) ) \| 
		\leq 
		\| \Theta_{V_{\pi,\vartheta_{m}}} ( y_{m} - \wt{\tr}_{\cH_\pi}(y_{m}) ) \| 
		\\={}&
		\| \Theta_{V_{\pi,\phi_{1}}}^{m} \cdots \Theta_{V_{\pi,\phi_{r-1}}}^{m} (\Theta_{V_{\pi,\phi_r}}^{m} - \wt{\tr}_{\cH_\pi}) \Theta_{V_{\pi,\phi_{r+1}}}^{m} \cdots \Theta_{V_{\pi,\phi_{m}}}^{m} ( y_{m} ) \| 
		\\\leq{}&
		\| \Theta_{V_{\pi,\phi_r}}^{m} - \wt{\tr}_{\cH_\pi} \| \| x \| 
		\xto{n+1\leq m\to\infty} 0. 
	\end{align*}
	where the last convergence follows from \autoref{cor_quantum_channel_irred} \ref{item_cor_quantum_channel_irred_4} applied to $V_{\pi,\phi_r}$. Thus, $E_{n,\infty}(x)=\tr_M(x)1$ whenever $n\geq r$. 
	Since $E_{n,\infty}(x)$ converges to $x$ weakly as $n\to\infty$, it follows that $x=\tr_M(x)1$. 
	This shows the claim and thus completes the proof of \autoref{thm_RFD_UHF}. 
\end{proof}

\begin{cor}\label{cor_RFD_UHF}
	Let $4\leq d\in \bZ$. Then, there exists an irreducible hyperfinite subfactor $N\subset M$ of the hyperfinite $\mathrm{II}_1$ factor $M$ with separable predual such that it has index $d$ and the trivial standard invariant. 
\end{cor}

Note that \autoref{main_subfactor} \ref{item_main_subfactor_1} follows from this. 

\begin{proof}
	This is a consequence of the existence of an outer action of $\wc{{S_d^+}}$ on the hyperfinite $\mathrm{II}_1$ factor with separable predual $\cR$ by \autoref{thm_RFD_UHF} and \autoref{eg_tensor_faith_Sn+}. 
\end{proof}

When $d\geq 4$ is a square of some integer, \autoref{cor_RFD_UHF} is also a consequence of the existence of an outer action of $\rO^+(\sqrt{d})$ on $\cR$ by \autoref{thm_RFD_UHF} and \autoref{eg_tensor_faith_On+}. 
Note that, as a consequence of a discussion in \autoref{eg_tensor_faith_Un+}, this construction does not rely on the classification of the standard invariant. 

\subsection{Minimal actions on injective factors}\label{ssec_minimal}

\begin{lem}\label{lem_fixed_nuclear}
	Let $G$ be a compact quantum group with $C(G)$ nuclear and $(A,\alpha)$ be a left $G$-C*-algebra. Then, $A$ is nuclear if and only if $A^\alpha := \{ a\in A \mid \alpha(a)=1_{C(G)}\otimes a \}$ is nuclear. 
\end{lem}

Note that $C(G)$ is nuclear if $G$ is coamenable \cite{Bedos-Murphy-Tuset2003amenabilityII}, which is the case for $\SU_q(2)$ and $\SO_q(3)$ for all $q\in[-1,1]\setminus\{0\}$. 
In the proof, the maximal tensor product of C*-algebras is denoted by $\motimes$. 

\begin{proof}
	The same argument as \cite[Theorem 4.5.2]{Brown-Ozawa-book} shows the statement, but we include the proof for convenience of the reader. 
	If $A$ is nuclear, then so is $A^\alpha$ as the range of the conditional expectation $(h_G\otimes\id_A)\alpha\colon A\to A^\alpha$. 
	Conversely, we assume the nuclearity of $A^\alpha$ and take any C*-algebra $B$. 
	
	The $*$-homomorphism $A^\alpha\otimes B = A^\alpha\motimes B \to A\motimes B$ induced by the inclusion $A^\alpha\odot B \subset A\odot B$ must be injective because so is its composition with the canonical quotient map $q\colon A\motimes B\to A\otimes B$. 
	Since $C(G)$ is nuclear, 
	we see $C(G)\otimes (A\motimes B) = C(G)\motimes A\motimes B = (C(G)\otimes A)\motimes B$ and thus $\alpha\odot\id_B \colon A\odot B\to (C(G)\otimes A)\odot B$ induces a well-defined $*$-homomorphism denoted by $\alpha\motimes\id_B\colon A\motimes B\to C(G)\otimes (A\motimes B)$. 
	It is not hard to check that $\alpha\motimes\id_B$ is a continuous left $G$-action on $A\motimes B$. 
	Then, we have the faithful conditional expectations 
	\begin{align*}
		E :={}& 
		(h_G\otimes\id_{A\otimes B})(\alpha\otimes\id_B) 
		\colon A\otimes B\to A^\alpha\otimes B, 
		\\
		E_m :={}& 
		(h_G\otimes\id_{A\motimes B})(\alpha\motimes\id_B) 
		\colon A\motimes B 
		\to 
		\cspan\{ a\otimes b\in A\motimes B \mid a\in A^\alpha, b\in B \} = A^\alpha\otimes B , 
	\end{align*}
	where for the last equality we used the nuclearity of $A^\alpha$. 
	We have $(\alpha\otimes\id_B)q=(\id_{C(G)}\otimes q)(\alpha\motimes\id_B)$ and thus $Eq=qE_m$. 
	It follows from the injectivity of $q$ on $A^\alpha\otimes B\subset A\motimes B$ and the faithfulness of $E_m$ that $(\Ker q)_+=0$ and thus $\Ker q=0$. 
\end{proof}

\begin{rem}\label{rem_biGalois}
	Let $G$ and $H$ be compact quantum groups such that we have a unitary monoidal equivalence $\cF\colon \Rep H\simeq \Rep G$. 
	For each $\pi = (\cH_\pi,V_\pi)\in\Rep H$, we write $\cF(\pi)=(\cG_\pi,U_\pi)\in \Rep G$, i.e., $\cG_\pi:=\cH_{\cF(\pi)}$ and $U_\pi:=V_{\cF(\pi)}$. 
	For $\pi\in\Rep H$, we take orthonormal bases of $\cG_\pi$ and $\cH_\pi$ and express $V_\pi=(v^{\pi}_{i,j})_{i,j}$ and $U_\pi=(u^{\pi}_{k,l})_{k,l}$ using $v^{\pi}_{i,j}\in C(H)$ and $u^{\pi}_{k,l}\in C(G)$. 
	Then, there is a unital C*-algebra $P$ equipped with a continuous left $G$-action $\lambda$ and a continuous right $H$-action $\rho$ satisfying the following (see e.g., \cite[Theorem 2.3.11]{Neshveyev-Tuset-book}). 
	\begin{itemize}[leftmargin=1.5em]
		\item
		$(\lambda\otimes\id_{C(H)})\rho = (\id_{C(G)}\otimes \rho)\lambda$. 
		\item
		there is a unique faithful state $h_{P}$ such that $(h_G\otimes \id_P)\lambda(x) = h_{P}(x)1_P = (\id_P \otimes h_H)\rho(x)$ for all $x\in P$. 
		\item
		For $\pi\in\Irr G\cong\Irr H$, there is a unitary $X_\pi\in \cB(\cH_\pi, \cG_\pi)\otimes P$ such that, by expressing $X_\pi=(x^{\pi}_{i,k})_{i,k}$ with the orthonormal bases of $\cG_\pi$ and $\cH_\pi$ taken above, 
		\begin{itemize}[leftmargin=1.5em]
			\item
			$h_P(x^{\varpi *}_{j,l} x^{\pi}_{i,k}) = \delta_{\pi,\varpi}\delta_{k,l} h_G(u^{\varpi *}_{j,m} u^{\pi}_{i,m})$, 
			$h_P(x^{\pi}_{i,k} x^{\varpi *}_{j,l}) = \delta_{\pi,\varpi}\delta_{i,j} h_H(v^{\pi}_{n,k} v^{\varpi *}_{n,l})$, 
			\item
			$\lambda(x^{\pi}_{i,k}) = \sum_{m=1}^{\dim_\bC\cG_\pi} u^{\pi}_{i,m}\otimes x^{\pi}_{m,k}$, 
			$\rho(x^{\pi}_{i,k}) = \sum_{n=1}^{\dim_\bC\cH_\pi} x^{\pi}_{i,n}\otimes v^{\pi}_{n,k}$, 
			\item
			$P=\cspan\{ x^{\pi}_{m,n} \mid \pi\in\Irr H, m\in\{1,\cdots,\dim_\bC\cG_\pi\}, n\in\{1,\cdots,\dim_\bC\cH_\pi\} \}$, 
		\end{itemize}
		for all $\pi,\varpi\in\Irr H$, $i\in\{1,\cdots,\dim_\bC\cG_\pi\}$, $j\in\{1,\cdots,\dim_\bC\cG_\varpi\}$, $m\in\{1,\cdots, \min\{\dim_\bC\cG_\pi,\dim_\bC\cG_\varpi\}\}$, $k\in\{1,\cdots,\dim_\bC\cH_\pi\}$, $l\in\{1,\cdots,\dim_\bC\cH_\varpi\}$, $n\in\{1,\cdots, \min\{\dim_\bC\cH_\pi,\dim_\bC\cH_\varpi\}\}$. 
	\end{itemize}
	We call this triple $(P,\lambda,\rho)$ the \emph{$(G,H)$-biGalois object} witnessing the unitary monoidal equivalence $\cF$. 
	Moreover, when we consider the GNS $*$-representation of $P$ on $L^2(P)$ associated with $h_P$, then $\lambda$ and $\rho$ uniquely extends to a left $G$-action and a right $H$-action on the von Neumann algebra $Q := P''\subset \cB(L^2(P))$, which are denoted by $\bar{\lambda}$ and $\bar{\rho}$ respectively. Note that 
	\begin{align}\label{eq_rem_biGalois_0}
		&
		(\bar{\lambda}\otimes\id_{L^\infty(H)})\bar{\rho} = (\id_{L^\infty(G)}\otimes \bar{\rho})\bar{\lambda} . 
	\end{align}
\end{rem}

In the setting of \autoref{rem_biGalois}, we implement $\rho$ with a unitary and compare it with $V^H$ for later usage. 
We have the unitary $Y^P\colon L^2(P)\to \cH_P := \bigoplus_{\pi\in\Irr G} \cH_\pi^{\oplus\dim_\bC\cG_\pi}$ sending each $x^{\pi}_{i,k}$ regarded as a vector in $L^2(P)$ to $( \delta_{i,j} \sqrt{h_G(u^{\pi *}_{i,1} u^{\pi}_{i,1})} e^\pi_{k} )_{j=1}^{\dim_\bC\cG_\pi}$, where $(e^\pi_k)_k$ denotes the orthonormal basis of $\cH_\pi$ fixed above, 
and, for all $\tau\in\Irr H$, $i\in\{1,\cdots,\dim_\bC\cG_\tau\}$, $k\in\{1,\cdots,\dim_\bC\cH_\tau\}$, $x\in P$, and $\xi\in L^2(H)$, by regarding $x^{\tau}_{i,k}$, $1_P$, and $x$ as elements in $L^2(P)$, 
\begin{align*}
	&
	\bigoplus_{\pi\in\Irr G} V_{\pi}^{\oplus\dim_\bC\cG_\pi} 
	(Y^P x^{\tau}_{i,k}\otimes \xi) 
	\\={}& 
	\bigoplus_{\pi\in\Irr G} V_{\pi}^{\oplus\dim_\bC\cG_\pi} 
	\Bigl( ( \delta_{i,j} \sqrt{h_G(u^{\tau *}_{i,1} u^{\tau}_{i,1})} e^\tau_{k} )_{j=1}^{\dim_\bC\cG_\tau} \otimes \xi \Bigr) 
	\\={}& 
	\sum_{l=1}^{\dim_\bC\cH_\tau} 
	( \delta_{i,j} \sqrt{h_G(u^{\tau *}_{i,1} u^{\tau}_{i,1})} e^\tau_{l} )_{j=1}^{\dim_\bC\cG_\tau} \otimes v^{\tau}_{l,k} \xi 
	\\={}& 
	\sum_{l=1}^{\dim_\bC\cH_\tau} 
	Y^P x^{\tau}_{i,l}\otimes v^{\tau}_{l,k} \xi 
	= 
	(Y^P \otimes\id_{L^2(H)}) \rho(x^{\tau}_{i,k}) (e^{\mathbbm{1}}_1\otimes \xi) , 
\end{align*}
which implies, by setting $U^P := \bigoplus_{\pi\in\Irr G} V_{\pi}^{\oplus\dim_\bC\cG_\pi}$, 
\begin{align}\label{eq_rem_biGalois_1}
	&
	U^P(Y^Px\otimes \xi) = (Y^P \otimes\id_{L^2(H)}) \rho(x) (e^{\mathbbm{1}}_1\otimes \xi) . 
\end{align}
Thus, for all $x\in P$ we have 
\begin{align}\label{eq_rem_biGalois_2}
	&
	\Ad (Y^P \otimes 1_{C(H)}) \rho(x) = \Ad U^P (Y^P xY^{P*} \otimes\id_{L^2(H)}) 
\end{align}

In particular, when $G=H$ and $\cF=\id_{\Rep H}$, we have $(P,\lambda,\rho)=(C(H),\Delta_H,\Delta_H)$ and it follows from \eqref{eq_rem_biGalois_1} that 
\begin{align}\label{eq_rem_biGalois_3}
	&
	\Ad(Y^{C(H)}\otimes 1_{C(H)})(V^H) = U^{C(H)} = \bigoplus_{\pi\in\Irr H} V_{\pi}^{\oplus\dim_\bC\cH_\pi} . 
\end{align}

Then, $\Ad Y^{C(H)}$ restricts to a unital normal $*$-homomorphism 
\begin{align*}
	&
	\ell^\infty(\Lambda) = 
	\{ (\id_{c_0(\Lambda)}\otimes\omega) (V^H) \in\cB(L^2(H)) \mid \omega\in\cB(L^2(H))_* \}'' 
	\\\to{}& 
	\biggl\{ (\id_{\cB(\bigoplus_{\pi\in\Irr H} \cH_{\pi}^{\oplus\dim_\bC\cH_\pi})}\otimes\omega) \biggl(\bigoplus_{\pi\in\Irr H} V_{\pi}^{\oplus\dim_\bC\cH_\pi}\biggr) \ \bigg|\  \omega\in\cB(L^2(H))_* \biggr\}'' 
	\\\subset{}& 
	\prod_{\pi\in\Irr H} \id_{\bC^{\oplus\dim_\bC\cH_\pi}}\otimes \cB(\cH_\pi) 
\end{align*}
such that when we write $\pr_\pi\colon \prod_{\pi\in\Irr H} \id_{\bC^{\oplus\dim_\bC\cH_\pi}} \otimes \cB(\cH_\pi) \to \cB(\cH_\pi)$ for the projection, $( (\pr_\pi \Ad Y^{C(H)}) \otimes \id_{C(H)} ) (V^H) = V_\pi$. Since $(\pr_\pi \Ad Y^{C(H)}) \otimes \id_{C(H)}$ restricts to a non-degenerate $*$-homomorphism from $c_0(\Lambda)$, it follows from the uniqueness of $\Pi_\pi$ that $\pr_\pi \Ad Y^{C(H)} = \Pi_\pi$. 
We will freely use the notation introduced so far in the proof of \autoref{thm_minimal_action}.

\begin{thm}\label{thm_minimal_action}
	Let $(q,G,\lambda)$ be one of the triples as follows. 
	\begin{enumerate}[leftmargin=*,label=(\roman*)]
		\item\label{item_thm_minimal_action_SU+q(2)}
		$q= n-\sqrt{n^2-1}$ with $n\in\bZ_{\geq 2}$, $G=\SU_q(2)$, and $\lambda \in \{q^{\frac{1}{k}} \mid k\in\bZ_{\geq 1} \}\cup\{1\}$. 
		\item\label{item_thm_minimal_action_SU-q(2)}
		$q=\frac{n-\sqrt{n^2-4}}{2}$ with $n\in\bZ_{\geq 3}$, $G=\SU_{-q}(2)$, and $\lambda \in \{q^{\frac{1}{k}} \mid k\in\bZ_{\geq 1} \}\cup\{1\}$. 
		\item\label{item_thm_minimal_action_SOq(3)}
		$q=\frac{\sqrt{n}-\sqrt{n-4}}{2}$ with $n\in\bZ_{\geq 5}$, $G=\SO_q(3)$, and $\lambda \in \{q^{\frac{2}{k}} \mid k\in\bZ_{\geq 1} \}\cup\{1\}$. 
	\end{enumerate}
	Then, there is a minimal left $G$-action on the Araki--Woods factor of type $\mathrm{III}_\lambda$. 
\end{thm}

Note that \autoref{main_subfactor} \ref{item_main_subfactor_2} follows from this. 

\begin{proof}
	In this situation, there is a compact quantum group $H$ with $\Lambda:=\wc{H}$ that admits some tensorially faithful $\theta\in\Rep\Lambda$ such that $\Rep H\simeq\Rep G$ by \autoref{ssec_freeQG}. 
	Indeed, by \autoref{eg_tensor_faith_Sn+}, \autoref{eg_tensor_faith_On+}, and \autoref{eg_tensor_faith_OJn+}, 
	we may take $H = \rO^+\bigl(\begin{smallmatrix}0 & - 1_{\bM_n} \\ 1_{\bM_n} & 0\end{smallmatrix}\bigr)$ in the case of \ref{item_thm_minimal_action_SU+q(2)}, $H = \rO^+(n)$ in the case of \ref{item_thm_minimal_action_SU-q(2)}, and $H = S_n^+$ in the case of \ref{item_thm_minimal_action_SOq(3)}. 
	Via the unitary monoidal equivalence $\Rep H\simeq\Rep G$, we shall regard $\Irr H=\Irr G$. 
	By applying \autoref{thm_IU_JiangSu} or \autoref{thm_RFD_UHF} to $\Gamma:=\Lambda^{\op}$ (cf.\ \autoref{rem_tensor_faith_op}), there is an outer trace-preserving left $\Lambda$-action $\gamma$ on a unital simple nuclear separable C*-algebra $A$ with a unique faithful tracial state $\tr_A$ such that $\gamma$ normally extends to an outer trace-preserving left $\Lambda$-action, still denoted by $\gamma$, on $M:=A''\subset \cB(L^2(A,\tr_A))$, which is a hyperfinite $\mathrm{II}_1$ factor with separable predual. 
	We write $\tr_M$ for the unique faithful normal tracial state on $M$. 
	For each $\pi\in\Irr H$, we write $\gamma_{\pi} := (\Pi_\pi\otimes\id_A)\gamma$ and $\gamma_P:=\bigoplus_{\pi\in\Irr H}\gamma_{\pi}^{\oplus\dim_\bC\cG_\pi} \colon M\to \cB(\cH_P)\barotimes M$. 
	
	We consider the $(G,H)$-biGalois object $(P,\lambda,\rho)$ and the unital faithful $*$-representation of $P$ on the Hilbert space $L^2(P)$ obtained by the GNS construction of the faithful state $h_P$ on $P$. 
	Then, the braided tensor product 
	\begin{align*}
		&
		B:=P \boxtimes_\Lambda A = \cspan (\rho(P)\otimes1_A)(1_P\otimes\gamma(A)) \subset \cM( P \otimes \cK(L^2(H)) \otimes A) , 
	\end{align*}
	which is a well-defined C*-algebra (see \cite[Proposition 8.3]{Vaes2005new} and \cite[Section 3]{Nest-Voigt2010equivariant}), is unital by construction and 
	equipped with the well-defined continuous left $G$-action $\beta:=\lambda\otimes\id_{\cK(L^2(H)) \otimes A}$. 
	Then, since 
	\begin{align*}
		(h_G\otimes\id_B)\beta 
		= 
		\bigl( (h_G\otimes \id_{P})\lambda \bigr) \otimes \id_{\cK(L^2(H))\otimes A} 
		= 
		h_P \otimes \id_{\cK(L^2(H))\otimes A} , 
	\end{align*}
	we have 
	\begin{align}\label{eq_prf_thm_minimal_action_2}
		&
		\begin{aligned}
			B^\beta ={}& \cspan (h_G\otimes\id_B)\beta(B) 
			= \cspan 1_P\otimes (h_P\otimes \id_{\cK(L^2(H))\otimes A}) (B) 
			\\={}& \cspan (h_P(P)1_P\otimes1_{\cM(\cK(L^2(H))\otimes A)})(1_P\otimes\gamma(A)) = 1_P\otimes\gamma(A) . 
		\end{aligned}
	\end{align}
	Thus, it follows from the nuclearity of $A$ that $P$ is nuclear by \autoref{lem_fixed_nuclear}. 
	
	We define $N\subset Q\barotimes \cB(L^2(P))\barotimes M$ as the von Neumann subalgebra generated by $B$. 
	By \eqref{eq_rem_biGalois_0}, we observe that the unital injective normal $*$-homomorphism $\bar{\beta} := \bar{\lambda}\otimes \id_{\cB(L^2(P))\barotimes M} \colon N\to L^\infty(G)\barotimes N$ is the unique normal extension of $\beta$ and that $\bar{\beta}$ is a left $G$-action on $N$. 
	Since $B$ is nuclear, $N$ must be injective. Since $L^2(P)\otimes L^2(H)\otimes L^2(M,\tr_M)$ is separable, $N_*$ is separable. 
	It follows from $\cspan\lambda(P)(1_{C(G)}\otimes P)=C(G)\otimes P$ (see \cite[Proposition 7.6.2]{DeCommer-thesis}) that $\cspan(1_{C(G)}\otimes B)\beta(B)=C(G)\otimes B$ and thus that $L^\infty(G)$ is generated as a von Neumann algebra by $(\id_{L^\infty(G)}\otimes N_*)\bar{\beta}(N)$. 
	
	Before going further, we give an alternative picture of $N$. 
	We note that 
	\begin{align*}
		B_0:={}&
		\cspan ( Y^P P Y^{P*}\otimes \id_{ L^2(A,\tr_A) } ) \gamma_{P}(A) 
		\subset 
		\cB ( \cH_P \otimes L^2(A,\tr_A) ) 
	\end{align*}
	is a C*-subalgebra that is $*$-isomorphic to $B$. 
	Indeed, since it follows from \eqref{eq_rem_biGalois_3} and $\Pi_\pi=\pr_\pi\Ad Y^{C(H)}$ that 
	\begin{align*}
		&
		\Ad (V_\pi^*\otimes 1_{\cM(A)}) (1_{\cB(\cH_\pi)} \otimes \gamma) 
		= 
		(\id_{\cB(\cH_\pi)}\otimes\gamma)\gamma_{\pi} , 
	\end{align*}
	also by using \eqref{eq_rem_biGalois_2}, the images by the injective normal $*$-homomorphism 
	\begin{align}\label{eq_prf_thm_minimal_action_3}
		&
		\begin{aligned}
			& 
			\Ad \bigl( (Y^{P*}\otimes1_{\cB(L^2(H))}) U^{P} \otimes 1_{M} \bigr) \circ (\id_{\cK(\cH_P)} \otimes \gamma) 
			\\&{}\colon 
			\cB( \cH_P ) \barotimes M 
			\to 
			\cB(L^2(P)) \barotimes \cB(L^2(H)) \barotimes M 
		\end{aligned}
	\end{align}
	of $\gamma_{P}(A)$ and $Y^P P Y^{P*}\otimes 1_{A}$ are $1_P\otimes \gamma(A)$ and $\rho(P)\otimes 1_{A}$ respectively. 
	Thus, \eqref{eq_prf_thm_minimal_action_3} restricts to a $*$-isomorphism from $B_0$ to $B$. If we write $N_0$ for the von Neumann subalgebra of the domain of \eqref{eq_prf_thm_minimal_action_3} generated by $B_0$, then \eqref{eq_prf_thm_minimal_action_3} restricts to a normal $*$-isomorphism from $N_0$ to $N$. 
	
	We are going to show that ${N^{\bar{\beta}}}'\cap N\cong \bC$, from which $\cZ(N)=\bC$ will follow. By \eqref{eq_prf_thm_minimal_action_2}, it suffices to show $(1_P\otimes\gamma(M))'\cap N\cong \bC$. Via the $*$-isomorphism \eqref{eq_prf_thm_minimal_action_3}, we show $\gamma_P(M)' \cap N_0\cong\bC$ instead. 
	We consider the following faithful normal state 
	\begin{align*}
		&
		f\colon 
		N_0 \xto{\id_{\cB(\cH_P)}\otimes \tr_M}  
		Y^P Q Y^{P*} \xto{h_P \Ad Y^{P*}} \bC, 
	\end{align*}
	where $(\id\otimes \tr_M)(N_0) \subset Y^P Q Y^{P*}$ by $(\id_{\cB(\cH_\pi)}\otimes\tr_M)\gamma_{\pi}=\id_{\cH_\pi}\otimes\tr_M$ for all $\pi\in\Irr H$ as $\gamma$ is $\tr_M$-preserving. 
	Note that, when we write $\Omega_P\in L^2(P)$ and $\Omega_M\in L^2(M,\tr_M)$ for the cyclic unit vectors associated with the GNS constructions of $h_P$ and $\tr_M$ respectively, $f$ is implemented by $Y^P\Omega_P\otimes \Omega_M$ with the canonical embedding $N_0\subset \cB\bigl( \cH_P \otimes L^2(M,\tr_M) \bigr)$. 
	
	We take any $x\in \gamma_{P}(M)' \cap N_0$. 
	Then, $x(Y^P\Omega_P\otimes \Omega_M) \in \cH_P\otimes L^2(M,\tr_M)$ satisfies that for all $a\in M$, 
	\begin{align*}
		&
		\gamma_{P}(a)x(Y^P\Omega_P\otimes \Omega_M) 
		= 
		x\gamma_{P}(a)(Y^P\Omega_P\otimes \Omega_M) 
		= 
		x(Y^P\Omega_P\otimes a\Omega_M)
	\end{align*}
	since $Y^P\Omega_P = Y^P x^{\mathbbm{1}}_{1,1} = e^{\mathbbm{1}}_1$ and $\gamma_{\mathbbm{1}}=\id_M$. 
	Thus, the isometry $T_x\colon L^2(M,\tr_M) \ni\xi \mapsto x(Y^P\Omega_P\otimes \xi) \in \cH_P\otimes L^2(M,\tr_M)$ is $(M,M)$-bilinear, where we equipped the right $M$-modules structures on $L^2(M,\tr_M)$ and $\cH_P\otimes L^2(M,\tr_M)$ by $J_M(-)^*J_M$ and $\id_{\cH_P}\otimes J_M(-)^*J_M$ using the modular conjugation $J_M$ associated with $\tr_M$. It follows from the outerness of $\gamma$ and \autoref{rem_outer_tensorcat} that $T_x$ coincides with the inclusion of $L^2(M,\tr_M) \cong \cH_{\mathbbm{1}}\otimes L^2(M,\tr_M)$ into $\cH_P\otimes L^2(M,\tr_M)$ up to scalar multiplication and in particular, $x(Y^P\Omega_P\otimes \Omega_M) \in \bC(Y^P\Omega_P\otimes \Omega_M)$. Since $f$ is faithful, $N_0\ni x\mapsto x(Y^P\Omega_P\otimes \Omega_M) \in \cH_P\otimes L^2(M,\tr_M)$ is injective. 
	Therefore, $x\in \bC 1_{N_0}$, which implies $\gamma_{P}(M)' \cap N_0 \cong \bC$ as desired. 
	
	We have shown that $\bar{\beta}$ is a minimal left action of $G$ on the injective factor $N$ with separable predual. We are going to determine the type of $N\cong N_0$. 
	We fix an orthonormal basis of $\cG_\pi$ for each $\pi\in\Irr H=\Irr G$ as in the end of \autoref{ssec_freeQG}. 
	Take $a,b\in M$ with $n\in\bZ_{\geq 1}$ and $\pi,\varpi\in\Irr H$, $i\in\{1,\cdots,\dim_\bC\cG_\pi\}$, $j\in\{1,\cdots,\dim_\bC\cG_\varpi\}$, $k\in\{1,\cdots,\dim_\bC\cH_\pi\}$, $l\in\{1,\cdots,\dim_\bC\cH_\varpi\}$ arbitrarily. 
	We calculate 
	\begin{align*}
		&
		f\bigl( \gamma_{P}(b)^* (Y^Px^{\varpi *}_{j,l} x^{\pi}_{i,k} Y^{P*} \otimes 1_M) \gamma_{P}(a) \bigr) 
		\\={}& 
		\bigl\bra \gamma_{P}(b) (e^{\mathbbm{1}}_{1}\otimes \Omega_M), (Y^Px^{\varpi *}_{j,l} x^{\pi}_{i,k} Y^{P*} \otimes 1_M) \gamma_P(a) (e^{\mathbbm{1}}_{1}\otimes \Omega_M) \bigr\ket
		\\={}& 
		\bigl\bra e^{\mathbbm{1}}_{1}\otimes b\Omega_M, (Y^Px^{\varpi *}_{j,l} x^{\pi}_{i,k} Y^{P*} \otimes 1_M) (e^{\mathbbm{1}}_{1}\otimes a\Omega_M) \bigr\ket
		\\={}& 
		h_P( x^{\varpi *}_{j,l} x^{\pi}_{i,k} ) \tr_M(b^*a) 
		=\delta_{\pi,\varpi} \delta_{k,l} 
		h_G( u^{\varpi *}_{j,1} u^{\pi}_{i,1} ) \tr_M(a b^*) 
		\\={}& 
		\frac{\delta_{\pi,\varpi} \delta_{i,j} \delta_{k,l}}{\dim_\bC\cH_\pi} |q|^{2i - 1 - \dim_{\bC}\cG_\pi} \tr_M(a b^*) , 
	\end{align*}
	where the last equality holds because $\dim_{\Rep G}\pi = \dim_{\Rep H} \pi = \dim_\bC\cH_\pi$ by the monoidal equivalence and the Kac type of $H$ by \autoref{rem_RFD_Kac}. 
	Also, since $\gamma$ is $\tr_M$-preserving, we compute 
	\begin{align*}
		&
		f\bigl( (Y^P x^{\pi}_{i,k} Y^{P*}\otimes1_M) \gamma_{P}(a) \gamma_{P}(b)^* (Y^P x^{\varpi *}_{j,l} Y^{P*}\otimes1_M) \bigr) 
		\\={}& 
		\tr_M(ab^*) 
		h_P ( x^{\pi}_{i,k} x^{\varpi *}_{j,l} ) 
		= 
		\tr_M(ab^*) 
		\frac{ \delta_{\pi,\varpi} \delta_{i,j} }{\dim_\bC\cH_\pi}
		h_H( v^{\pi}_{1,k} v^{\varpi *}_{1,l} ) 
		\\={}& 
		\frac{ \delta_{\pi,\varpi} \delta_{i,j} \delta_{k,l} }{\dim_\bC\cH_\pi} \tr_M(ab^*) . 
	\end{align*}
	Thus, for all $a,b\in M$ and $\pi,\varpi,i,j,k,l$ as above, we have 
	\begin{align*}
		&
		f\bigl( \gamma_{P}(b)^* (Y^P x^{\varpi *}_{j,l} Y^{P*}\otimes1_M) (Y^P x^{\pi}_{i,k} Y^{P*}\otimes1_M) \gamma_{P}(a) \bigr) 
		\\={}& 
		|q|^{\dim_\bC \cG_{\pi} + 1 - 2i} 
		f\bigl( (Y^P x^{\pi}_{i,k} Y^{P*}\otimes1_M) \gamma_{P}(a) \gamma_{P}(b)^* (Y^P x^{\varpi *}_{j,l} Y^{P*}\otimes1_M) \bigr) ,
	\end{align*}
	and the modular operator $\Delta_f$ associated to $f$ acts on $(Y^P x^{\pi}_{i,k} Y^{P*} \otimes 1_M) \gamma_{P}(a)$ regarded as a vector in $L^2(N_0,f)$ as the multiplication by the scalar $|q|^{\dim_\bC \cG_{\pi} + 1 - 2i}$. 
	
	Let $S(N_0)$ denote Connes' $S$-invariant of $N_0$. We write 
	\begin{align*}
		&
		N_0^f := \{ x\in N_0 \mid \sigma^f_t(x)=x, \forall t\in\bR \}\subset N_0 , 
	\end{align*}
	where $\sigma^f \colon \bR \ni t\mapsto \Ad \Delta_f^{t\sqrt{-1}}\in\Aut(N_0)$ is the modular automorphism associated with $f$. 
	Then, $\gamma_P(M) \subset N_0^f$ contains the image of $N^{\bar{\beta}}$ via the $*$-isomorphism \eqref{eq_prf_thm_minimal_action_3} by the description of $\Delta_f$ above and \eqref{eq_prf_thm_minimal_action_2}, which implies ${N_0^f}'\cap N_0^f \subset {N_0^f}'\cap N_0\cong \bC$ since we saw ${N^{\bar{\beta}}}'\cap N \cong \bC$. 
	Since $N_0$ and $N_0^f$ are factors, it follows from \cite[Corollary 3.2.7 (a)]{Connes1973classification}, \cite[Theorem 28.3 (3)]{Stratila-bookM2} that the intersection of $(0,\infty)$ with $S(N_0)$ is the same as that with the spectrum of $\Delta_f$, which is $\{ |q|^{\dim_\bC \cG_{\pi} + 1 - 2i} \mid \pi\in\Irr G, i\in\{1,\cdots,\dim_\bC\cG_\pi\} \}$. 
	This set equals $\{ |q|^k \mid k\in\bZ \}$ if $G=\SU_{\pm q}(2)$ for $q\in(-1,1)\setminus\{0\}$ and $\{ q^{2k} \mid k\in\bZ \}$ if $G=\SO_q(3)$ for $q\in(0,1)$. 
	Therefore, $N_0$ is of type $\mathrm{III}_{\mu}$, where $\mu=|q|$ in the case of \ref{item_thm_minimal_action_SU+q(2)} and \ref{item_thm_minimal_action_SU-q(2)}, and $\mu=q^2$ in the case of \ref{item_thm_minimal_action_SOq(3)}. 
	Thanks to the classification result \cite{Connes1976classificationinjective}, $N_0$ and thus $N$ are the Araki--Woods factor of type $\mathrm{III}_\mu$. 
	
	Finally, we consider the case of general $\lambda$ as in the statement. We write $R$ for the Araki--Woods factor of type $\mathrm{III}_\lambda$. Then, $N\barotimes R$ is the Araki--Woods factor of type $\mathrm{III}_\lambda$ by \cite[Lemma 5.6, Lemma 3.11, Thorem 7.6]{Araki-Woods1968classification}. 
	We can check that the left $G$-action $\bar{\beta}\otimes\id_R$ on $N\barotimes R$ is minimal, which completes the proof. 
\end{proof}

\appendix

\section{Actions of quantum groups}\label{sec_prelim_action}

We recall the notions of an action of a quantum group and its outerness. 

\begin{df}\label{def_action_DQG}
	Let $G$ be a compact quantum group, $A$ be a C*-algebra, and $M$ be a von Neumann algebra. We regard $c_0(\Gamma)=\bigoplus_{\pi\in\Irr G}\cB(\cH_\pi)$ via the $*$-isomorphism $\prod_{\pi\in\Irr G}\Pi_\pi$. 
	\begin{enumerate}[leftmargin=*,label=(\arabic*)]
		\item\label{item_def_action_DQG_C*}
		An injective $*$-homomorphism $\gamma\colon A\to \prod_{\pi\in\Irr G} (\cB(\cH_\pi)\otimes A)$ that is non-degenerate as a $*$-homomorphism into $\cM(c_0(\Gamma)\otimes A) = \cM(\bigoplus_{\pi\in\Irr G}\cB(\cH_\pi)\otimes A)$ is called a \emph{(C*-algebraic) left $\Gamma$-action} if we have 
		$(\id_{c_0(\Gamma)}\otimes\gamma)\gamma = (\Delta_{G}\otimes\id_A)\gamma$. 
		In this case, the pair $(A,\alpha)$ is called a left $\Gamma$-C*-algebra. 
		\item\label{item_def_action_DQG_W*}
		A \emph{(W*-algebraic) left $\Gamma$-action} on $M$ is a C*-algebraic left $\Gamma$-action $\gamma\colon M\to \prod_{\pi\in\Irr G} (\cB(\cH_\pi)\otimes M)$ such that $\gamma$ is normal as a $*$-homomorphism to $(\prod_{\pi\in\Irr G} \cB(\cH_\pi))\barotimes M$. 
		In this case, the pair $(M,\gamma)$ is called a left $\Gamma$-W*-algebra. 
	\end{enumerate}
\end{df}

\begin{rem}\label{rem_action_tensorcat}
	Let $G$ be a compact quantum group and $(A,\gamma)$ be a left $\Gamma$-C*-algebra. 
	For each $\pi\in\Rep G$, we consider the Hilbert $A$-module $\cH_\pi\otimes_\bC A$, where $A$ is regarded as a $(\bC,A)$-correspondence by the unital inclusion $\bC\to\cM(A)$, and the non-degenerate $*$-homomorphism $\gamma_{\pi} := (\Pi_\pi\otimes\id_A)\gamma\colon A\to \cB(\cH_\pi)\otimes A \subset \cL(\cH_\pi\otimes_\bC A)$. We write $\alpha(\pi)$ for the non-degenerate $(A,A)$-correspondence $(\cH_\pi\otimes_\bC A, \gamma_{\pi})$. 
	Then, together with the natural family of $(A,A)$-bilinear unitaries 
	\begin{align*}
		&
		\fu_{\pi,\varpi}\colon \alpha(\pi)\otimes_A\alpha(\varpi) 
		\ni (\xi\otimes a)\otimes (\eta\otimes b) \mapsto \xi\otimes \gamma_{\varpi}(a)(\eta\otimes b) \in \alpha(\pi\otimes\varpi) , 
	\end{align*}
	the assignment $\pi\mapsto\alpha(\pi)$ induces a well-defined unitary tensor functor $(\alpha,\fu)\colon \Rep G\to \Cor(A)$ (see e.g., \cite[Example 2.24]{Arano-Kitamura-Kubota2024tensor}). 
	Here, $\Cor(A)$ denotes the C*-tensor category of non-degenerate $(A,A)$-correspondences such that $\Hom_{\Cor(A)}((E,\phi),(F,\psi))=\{ X\in \cL(E,F) \mid X\phi(a)=\psi(a)X, \forall a\in A \}$ and $(E,\phi)\otimes(F,\psi):=(E\otimes_\psi F, \phi(-)\otimes\id_F)$ for non-degenerate $(A,A)$-correspondences $(E,\phi)$ and $(F,\psi)$. 
\end{rem}

For a unitary tensor category $\cC$ and a C*-algebra $A$, a unitary tensor functor $(\alpha,\fu)\colon \cC\to \Cor(A)$ is called a \emph{$\cC$-action} on $A$. 
For example, \autoref{rem_action_tensorcat} says that a left $\Gamma$-action on $A$ canonically induces a $\Rep G$-action on $A$. 
Whenever $A$ is non-zero, a unitary tensor functor $(\alpha,\fu)\colon \cC\to \Cor(A)$ is faithful by the semisimplicity and rigidity of $\cC$ (see e.g., \cite[Remark 4.2 (1)]{Kitamura2026actions}). 

\begin{lem}\label{lem_outer_tensorcat}
	Let $\cC$ be a unitary tensor category, $A$ be a C*-algebra with $A'\cap\cM(A)\cong \bC$, and $(\alpha,\fu)\colon \cC\to \Cor(A)$ be a unitary tensor functor. 
	Then, the $\cC$-action $(\alpha,\fu)$ on $A$ is \emph{outer} in the sense that $\alpha$ is fully faithful as a functor 
	if and only if the following two conditions hold for all irreducible objects $\pi\in\cC$. 
	\begin{enumerate}[leftmargin=*,label=(\arabic*)]
		\item\label{item_lem_outer_tensorcat_inv}
		There is no unitary in $\Hom_{\Cor(A)}(\alpha(\mathbbm{1}),\alpha(\pi))$ when $\pi$ is invertible (i.e., $\pi\otimes\overline{\pi}\cong\mathbbm{1}\cong\overline{\pi}\otimes\pi$). 
		\item\label{item_lem_outer_tensorcat_irr}
		$\Hom_{\Cor(A)}(\alpha(\pi),\alpha(\pi))\cong \bC$ when $\pi$ is non-invertible (i.e., not invertible). 
	\end{enumerate}
\end{lem}

\begin{proof}
	The only if direction is clear. We show the converse. 
	For $\pi\in\cC$, we write $\alpha_\pi\colon A\to\cL(\alpha(\pi))$ for the $*$-homomorphism witnessing the left $A$-module structure of $\alpha(\pi)$. 
	In order to show the outerness of $(\alpha,\fu)$, it suffices to check that 
	\begin{align}\label{eq_prf_lem_outer_tensorcat}
		&
		\dim_\bC\Hom_{\Cor(A)}(\alpha(\pi),\alpha(\varpi))=\delta_{\pi,\varpi} 
	\end{align}
	for all irreducible objects $\pi,\varpi\in\cC$ by the semisimplicity of $\cC$. 
	
	Firstly, we consider the case when $\pi=\varpi$. 
	When $\pi$ is non-invertible, \eqref{eq_prf_lem_outer_tensorcat} holds by \ref{item_lem_outer_tensorcat_irr}. When $\pi$ is invertible, then $\alpha(\pi)$ is an $(A,A)$-Morita equivalence by $\alpha(\pi)\otimes_A\alpha(\overline{\pi})\cong A\cong \alpha(\overline{\pi})\otimes_A\alpha(\pi)$. This means that $\alpha_\pi\colon A\to \cK(\alpha(\pi))$ is a $*$-isomorphism, which implies \eqref{eq_prf_lem_outer_tensorcat} by $\alpha_\pi(A)'\cap \cM(\cK(\alpha(\pi))) \cong A'\cap\cM(A)\cong\bC$. 
	
	From the case of $\pi=\varpi$, it is not hard to observe that for any non-zero $x\in\Hom_{\Cor(A)}(\alpha(\mathbbm{1}),\alpha(\tau))$ with $\tau\in\cC$ irreducible, $\|x\|^{-1}x$ is a unitary. 
	
	Next, we consider the case when $\pi\neq\varpi$. Note that the solution of the conjugate equation for each $\tau\in\cC$ induces a solution of the conjugate equation for $\alpha(\tau)\in\Cor(A)$. Thus, we may apply Frobenius reciprocity to get 
	\begin{align*}
		&
		\dim_\bC\Hom_{\Cor(A)}(\alpha(\pi),\alpha(\varpi)) = \dim_\bC\Hom_{\Cor(A)}(\alpha(\mathbbm{1}),\alpha(\overline{\pi})\otimes_A\alpha(\varpi)) 
		\\={}& \dim_\bC\Hom_{\Cor(A)}(\alpha(\mathbbm{1}),\alpha(\overline{\pi}\otimes\varpi)) . 
	\end{align*}
	Suppose that \eqref{eq_prf_lem_outer_tensorcat} does not hold. 
	Then, there must be some irreducible direct summand $\tau$ of $\overline{\pi}\otimes\varpi \in\cC$ such that $\Hom_{\Cor(A)}(\alpha(\mathbbm{1}),\alpha(\tau))\neq 0$. 
	There is a unitary in $\Hom_{\Cor(A)}(\alpha(\mathbbm{1}),\alpha(\tau))$ by the observation above, and $\tau\in\cC$ is non-invertible by \ref{item_lem_outer_tensorcat_inv}. 
	Since $\tau$ is non-invertible, there is $0\neq\sigma\in\cC$ such that $\mathbbm{1}\oplus\sigma\cong \tau\otimes\overline{\tau}$. 
	Since $\alpha$ is always faithful by $A\neq 0$ (as $A'\cap\cM(A)\cong \bC\neq 0$), we have $\alpha(\sigma)\neq 0$, and thus there is an isometry in $\Hom_{\Cor(A)}(\alpha(\mathbbm{1}),\alpha(\tau\otimes\overline{\tau}))$ that is not a unitary, which contradicts 
	$\alpha(\tau\otimes\overline{\tau})\cong \alpha(\tau)\otimes_A\alpha(\overline{\tau}) \cong \alpha(\mathbbm{1})\otimes_A\alpha(\mathbbm{1})\cong \alpha(\mathbbm{1})$ since $\Hom_{\Cor(A)}(\alpha(\mathbbm{1}),\alpha(\mathbbm{1}))\cong\bC$. 
	Thus, \eqref{eq_prf_lem_outer_tensorcat} holds in this case, too. 
\end{proof}

\begin{df}\label{def_outer_DQG}
	Let $G$ be a compact quantum group and $(A,\gamma)$ be a left $\Gamma$-C*-algebra or a left $\Gamma$-W*-algebra. Suppose that $A'\cap\cM(A)\cong\bC$. 
	Then, we say that $\gamma$ is \emph{outer} if for all $\pi\in\Irr G$,  
	\begin{enumerate}[leftmargin=*,label=(\arabic*)]
		\item\label{item_def_outer_DQG_inv}
		there is no unitary $U\in \cU\cM(A)$ such that $(\Pi_\pi\otimes\id_A)\gamma = \Ad U\colon A\to \cB(\cH_\pi)\otimes A\cong A$ when $\pi$ is invertible, and 
		\item\label{item_def_outer_DQG_irr}
		$(\Pi_\pi\otimes\id_A)\gamma(A)'\cap \cB(\cH_\pi)\otimes \cM(A)\cong \bC$ when $\pi$ is non-invertible. 
	\end{enumerate}
\end{df}

\begin{rem}\label{rem_outer_tensorcat}
	Let $G$ be a compact quantum group and $(A,\gamma)$ be a left $\Gamma$-C*-algebra with $A'\cap\cM(A)\cong\bC$. 
	Consider the corresponding $\Rep G$-action $(\alpha,\fu)\colon\Rep G\to\Cor(A)$ as in \autoref{rem_outer_tensorcat}. 
	Then, $\gamma$ is outer in the sense of \autoref{def_outer_DQG} if and only if $(\alpha,\fu)$ is outer in the sense of \autoref{lem_outer_tensorcat}. 
	
	When $A$ is a von Neumann algebra, it would be more common to treat symmetries of unitary tensor categories using, rather than $(A,A)$-correspondences, their completions to Hilbert spaces. 
	For a fixed faithful normal state $f$ on $A$, we write $J_f\colon L^2(A,f)\to L^2(A,f)$ for the modular conjugation associated with $f$. 
	For $\pi\in\Rep G$, we equip $\cH_\pi\otimes L^2(A,f)$ with the left $A$-action by $\gamma_\pi := (\Pi_\pi\otimes\id_A)\gamma$ and the right $A$-action by $\id_{\cH_\pi}\otimes J_f(-)^*J_f$. 
	Then, for $\pi,\varpi\in\Rep G$, the set of right $A$-linear operators in $\cB(\cH_\pi\otimes L^2(A,f),\cH_\varpi\otimes L^2(A,f))$ is the same as $\cB(\cH_\pi,\cH_\varpi)\otimes A \cong \cL(\cH_\pi\otimes A, \cH_\varpi\otimes A)$ since $(J_f A J_f)'=A''=A$. Thus, $\gamma$ is outer if and only if the dimension of the set of $(A,A)$-bilinear operators in $\cB(\cH_\pi\otimes L^2(A,f),\cH_\varpi\otimes L^2(A,f))$ equals $\delta_{\pi,\varpi}$ for all $\pi,\varpi\in\Irr G$. 
\end{rem}

In \autoref{ssec_minimal}, we also consider actions of compact quantum groups. 
\begin{df}\label{def_action_LCQG}
	Let $G$ be a locally compact quantum group, $A$ be a C*-algebra, and $M$ be a von Neumann algebra. 
	\begin{enumerate}[leftmargin=*,label=(\arabic*)]
		\item\label{item_def_action_LCQG_C*}
		A non-degenerate injective $*$-homomorphism $\alpha\colon A\to \cM (C_0(G)\otimes A)$ is a \emph{continuous left $G$-action} if we have 
		$(\id_{C_0(G)}\otimes\alpha)\alpha=(\Delta_G\otimes\id_A)\alpha$ 
		and 
		$\cspan (C_0(G)\otimes1_{\cM(A)})\alpha(A) = C_0(G)\otimes A$. 
		\item\label{item_def_action_LCQG_W*}
		A unital normal injective $*$-homomorphism $\alpha\colon M\to L^\infty(G)\barotimes M$ is a \emph{(W*-algebraic) left $G$-action} if we have $(\id_{L^\infty(G)}\otimes\alpha)\alpha = (\Delta_G\otimes\id_{M})\alpha$. 
		\item\label{item_def_action_LCQG_right}
		We say that $\alpha\colon A\to \cM (A\otimes C_0(G))$ is a \emph{continuous right $G$-action} if $a\mapsto \alpha(a)_{21}$ is a continuous left $G^{\op}$-action. 
		Also, $\alpha\colon M\to M\barotimes L^\infty(G)$ is a \emph{(W*-algebraic) right $G$-action} if $a\mapsto \alpha(a)_{21}$ is a W*-algebraic left $G^{\op}$-action. 
	\end{enumerate}
\end{df}

Consider a discrete quantum group $\Gamma$, a von Neumann algebra $M$, and a C*-algebra $A$. 
Then, the definitions of a W*-algebraic left $\Gamma$-action on $M$ in \autoref{def_action_DQG} \ref{item_def_action_DQG_W*} and in \autoref{def_action_LCQG} \ref{item_def_action_LCQG_C*} coincide. 
Also, a left $\Gamma$-action on $A$ in the sense of \autoref{def_action_LCQG} \ref{item_def_action_LCQG_C*} is equivalent to a continuous left $\Gamma$-action on $A$ in the sense of \autoref{def_action_DQG} \ref{item_def_action_DQG_C*} by \autoref{lem_coame_conti_action}. 

\begin{prop}\label{lem_coame_conti_action}
	Let $G$ be a coamenable locally compact quantum group and $(A,\alpha)$ be a pair of a C*-algebra and $\alpha\colon A\to \cM(C_0(G)\otimes A)$ be a $*$-homomorphism 
	such that for all $x\in C_0(G)$ and $a\in A$ we have $(x\otimes1_{\cM(A)})\alpha(a)\in C_0(G)\otimes A$ and 
	\begin{align*}
		&
		(\id_{C_0(G)}\otimes\alpha)((x\otimes1_{\cM(A)})\alpha(a)) = (x\otimes 1_{\cM(C_0(G)\otimes A)}) (\Delta_G\otimes\id_A)\alpha(a) . 
	\end{align*}
	Then, $\alpha$ is injective if and only if $\cspan(\cB(L^2(G))_*\otimes\id_A)\alpha(A)=A$. 
	
	If moreover $G$ is \emph{regular} in the sense that 
	\begin{align*}
		&
		\cspan \{ (\id_{\cB(L^2(G))}\otimes \omega)(\Sigma V^G) \mid \omega\in\cB(L^2(G))_* \} = \cK(L^2(G)) , 
	\end{align*}
	where $\Sigma$ denotes the unitary $L^2(G)^{\otimes 2}\ni \xi\otimes\eta\mapsto \eta\otimes\xi\in L^2(G)^{\otimes 2}$, then $\alpha$ is injective if and only if $\alpha$ is a continuous left $G$-action on $A$. 
\end{prop}

A locally compact quantum group is coamenable if and only if the counit $\epsilon_G\colon C^u_0(G)\to \bC$ factors through the canonical quotient $\rho_G\colon C^u_0(G)\to C(G)$ (see \autoref{ssec_prelim_Rep}). 
Discrete quantum groups are always coamenable and regular.

\begin{proof}
	For $\alpha$ as in the assumption, 
	$\alpha$ is non-degenerate and satisfies 
	$\cspan (\cB(L^2(G))_*\otimes\id_A)\alpha(A) = A$ if 
	$\cspan (C_0(G)\otimes 1_{\cM(A)})\alpha(A) = C_0(G)\otimes A$. 
	If $G$ is regular, the converse also holds by the computation in the proof of \cite[Proposition~5.8]{Baaj-Skandalis-Vaes2003non-semi-regular}. 
	Thus, we only have to show the former statement. 
	
	By coamenability, we shall identify $C_0(G)$ with $C^u_0(G)$ via $\rho_G$. 
	Note that $\epsilon_G\colon C_0(G)\to\bC$ is surjective and in particular non-degenerate. 
	Since 
	\begin{align*}
		&
		(\epsilon_G\otimes\id_A)\alpha(A)=(\epsilon_G\otimes\id_A)((C_0(G)\otimes1_{\cM(A)}) \alpha(A))\subset A , 
	\end{align*}
	we have that for all $x\in C_0(G)$ and $a\in A$, 
	\begin{align*}
		&
		\begin{aligned}
			&
			\epsilon_G(x)\alpha\bigl( (\epsilon_G\otimes\id_A)\alpha(a) \bigr) 
			= 
			\alpha\bigl( \epsilon_G(x) (\epsilon_G\otimes\id_A)\alpha(a) \bigr) 
			\\={}& 
			\alpha(\epsilon_G\otimes\id_A) \bigl( (x\otimes1_{\cM(A)})\alpha(a) \bigr) 
			= 
			(\epsilon_G\otimes\alpha) \bigl( (x\otimes1_{\cM(A)})\alpha(a) \bigr) 
			\\={}& 
			\epsilon_G(x)(\epsilon_G\otimes\id_{C_0(G)\otimes A}) (\Delta_G\otimes\id_A)\alpha(a) 
			= 
			\epsilon_G(x)\alpha(a) . 
		\end{aligned}
	\end{align*}
	By letting $x\in C_0(G)$ such that $\epsilon_G(x)\neq 0$, we see that for all $a\in A$, 
	\begin{align}\label{eq_lem_coame_conti_action}
		&
		\alpha\bigl( (\epsilon_G\otimes\id_A)\alpha(a) \bigr) = \alpha(a) . 
	\end{align}
	
	Suppose $\alpha$ is injective. 
	Then, it follows from \eqref{eq_lem_coame_conti_action} that $(\epsilon_G\otimes\id_A)\alpha=\id_A$. 
	We take $e\in C_0(G)$ such that $\epsilon_G(e)=1$ and $0\leq e\leq 1_{\cM(C_0(G))}$ and 
	extend $\epsilon_G\colon C_0(G)\to \bC$ to a (possibly non-normal) state $\wt{\epsilon}\in \cB(L^2(G))^*$. 
	We fix $a\in A$ and $\varepsilon\in (0,1]$ arbitrarily. 
	There are $n\in\bZ_{\geq 1}$, $x_1,\cdots,x_n\in C_0(G)$, and $a_1,\cdots,a_n\in A$ such that $\|(e\otimes 1_{\cM(A)})\alpha(a)-\sum\limits_{k=1}^{n}x_k\otimes a_k\|<\varepsilon$. 
	Also, there is a normal state $\omega\in \cB(L^2(G))_*$ such that $\|\omega(x_k) - \wt{\epsilon}(x_k)\|\leq n^{-1}\varepsilon(1+\|a_k\|)^{-1}$. 
	We estimate 
	\begin{align*}
		&
		\| a - (\omega\otimes\id_A)((e\otimes 1_{\cM(A)})\alpha(a)) \|
		\\\leq{}& 
		\| a - (\wt{\epsilon}\otimes\id_A)((e\otimes 1_{\cM(A)})\alpha(a)) \|
		+
		\biggl\| (\wt{\epsilon}\otimes\id_A)\biggl( (e\otimes 1_{\cM(A)})\alpha(a) - \sum\limits_{k=1}^{n}x_k\otimes a_k \biggr) \biggr\| 
		\\+{}&
		\biggl\| ((\wt{\epsilon}-\omega)\otimes\id_A)\biggl( \sum\limits_{k=1}^{n}x_k\otimes a_k \biggr) \biggr\| 
		+
		\biggl\| (\omega\otimes\id_A)\biggl( (e\otimes 1_{\cM(A)})\alpha(a) - \sum\limits_{k=1}^{n}x_k\otimes a_k \biggr) \biggr\| 
		\\\leq{}&
		\| a - (\epsilon_G\otimes\id_A)((e\otimes 1_{\cM(A)})\alpha(a)) \|
		+
		\varepsilon\| \wt{\epsilon}\otimes\id_A \| 
		+
		\biggl( \sum\limits_{k=1}^{n}n^{-1}\varepsilon \biggr) 
		+
		\varepsilon\| \omega\otimes\id_A \|
		\\\leq{}&
		3\varepsilon . 
	\end{align*}
	Since $\omega(e(-))\in\cB(L^2(G))_*$, we see $\cspan (\cB(L^2(G))_*\otimes\id_A)\alpha(A) =A$. 
	
	Conversely, suppose $\cspan (\cB(L^2(G))_*\otimes\id_A)\alpha(A) =A$. 
	Since 
	\begin{align*}
		&
		\cB(L^2(G))_* = \cspan\{ \omega(x(-)) \mid \omega\in\cB(L^2(G))_*, x\in C_0(G) \} 
	\end{align*}
	by the non-degeneracy of $C_0(G) \subset \cB(L^2(G))$, we see 
	\begin{align*}
		&
		(\epsilon_G\otimes\id_A)\alpha(A) 
		= 
		\cspan (\epsilon_G\otimes\id_A) \alpha\bigl( (\cB(L^2(G))_*\otimes\id_A) \alpha(A) \bigr) 
		\\={}&
		\cspan (\cB(L^2(G))_*\otimes\epsilon_G\otimes\id_A) (\id_{C_0(G)}\otimes\alpha) \bigl( (C_0(G)\otimes 1_{\cM(A)}) \alpha(A) \bigr) 
		\\={}&
		\cspan (\cB(L^2(G))_*\otimes\epsilon_G\otimes\id_A) \bigl( (C_0(G)\otimes 1_{\cM(C_0(G)\otimes A)}) (\Delta_G\otimes\id_A) \alpha(A) \bigr) 
		\\={}& 
		\cspan (\cB(L^2(G))_*\otimes\id_A) \bigl( (C_0(G)\otimes 1_{\cM(A)}) \alpha(A) \bigr) 
		= 
		A , 
	\end{align*}
	where we used $(\id_{C_0(G)}\otimes\epsilon_G)\Delta_G=\id_{C_0(G)}$ in the second last equality. 
	Thus, for any $c\in \Ker\alpha$, there is $a\in A$ such that $(\epsilon_G\otimes\id_A)\alpha(a)=c$. 
	From \eqref{eq_lem_coame_conti_action}, we see $0=\alpha(c)=\alpha(a)$. 
	Therefore $0=(\epsilon_G\otimes\id_A)\alpha(a)=c$, which implies the injectivity of $\alpha$. 
\end{proof}

Let $G$ be a locally compact quantum group and $(A,\alpha),(B,\beta)$ be left $G$-C*-algebras. 
A \emph{$G$-equivariant C*-subalgebra} of $(A,\alpha)$ is a C*-subalgebra $D\subset A$ such that $\alpha|_D$ is a well-defined continuous left $G$-action on $D$. 
Then, a $*$-homomorphism $f\colon A\to B$ is a \emph{$G$-$*$-homomorphism} if $(x\otimes 1_{\cM(A)}) \beta f(a) = (\id_{C_0(G)}\otimes f) ( (x\otimes 1_{\cM(A)}) \alpha(a) )$ for all $a\in A$ and $x\in C_0(G)$. When $G=\wc{F}$ for a compact quantum group $F$, this condition is equivalent to $\beta_\pi f(a) = (\id_{\cB(\cH_\pi)}\otimes f)\alpha_{\pi} (a)$ for all $a\in A$ and $\pi\in\Irr G$ by using the convention in \autoref{rem_action_tensorcat}.

\begin{df}[{\cite[Definition 5.5, Definition 6.1]{Vaes2001unitary}}]\label{def_action_minimal}
	Let $G$ be a locally compact quantum group and $(M,\alpha)$ be a left $G$-W*-algebra. We write $M^\alpha:=\{ x\in M \mid \alpha(x)=1\otimes x \}$. 
	We say that $\alpha$ is \emph{minimal} if $M^{\alpha \prime}\cap M\cong\bC$ and $L^\infty(G)$ is generated as a von Neumann subalgebra by $(\id_{L^\infty(G)}\otimes M_*)\alpha(M) \subset L^\infty(G)$. 
\end{df}

	For a compact quantum group $G$, a left $G$-action $\alpha$ on a factor $M$ is \emph{strictly outer} in the sense that $\alpha(M)'\cap (G\barltimes M) \cong\bC$ if 
	it is minimal by \cite[Proposition 6.2]{Vaes2001unitary}. 
	Suppose that there is a minimal action of a compact quantum group $G$ on a factor $M$ of type $\mathrm{I}$ or $\mathrm{II}$. 
	Then, there is a minimal action of $G$ on a factor of type $\mathrm{II}$ by the tensor product of $M$ and some factor of type $\mathrm{II}$. 
	Then, it is shown by \cite[Theorem 3.5 (c), Corollary 3.7]{Vaes2005strictly} that $G$ must be of Kac type. In particular, $G$ cannot be $\SU_{\pm q}(2)$ or $\SO_{q}(3)$ for any $0<q<1$.


\end{document}